\documentclass[11pt,a4paper]{article}
\usepackage{amsmath}
\usepackage{amssymb}
\usepackage{amsfonts}
\usepackage{amsthm}
\usepackage{relsize}
\usepackage{setspace}
\usepackage{geometry}
\usepackage{enumerate}
\usepackage{url}
\usepackage{xspace}
\usepackage{mathrsfs}
\usepackage{tocloft}
\usepackage{ulem}
\usepackage[all]{xy}

\usepackage{hyperref}
\usepackage{authblk}
\usepackage[usenames, dvipsnames]{color}
\usepackage{makeidx}
\makeindex

\newtheorem{thm}{Theorem}[section]
\newtheorem{lem}[thm]{Lemma}
\newtheorem{prop}[thm]{Proposition}
\newtheorem{cor}[thm]{Corollary}

\newtheorem{thmintro}{Theorem}

\newtheorem{corintro}[thmintro]{Corollary}
\newtheorem{propintro}[thmintro]{Proposition}

\theoremstyle{definition}
\newtheorem{defn}[thm]{Definition}
\newtheorem{ex}[thm]{Example}
\newtheorem{que}{Question}
\newtheorem{rem}[thm]{Remark}
\newtheorem*{rem*}{Remark}

\renewcommand{\emph}[1]{\textbf{\textit{#1}}}

\newcommand{\defbold}{\textbf}

\newcommand{\inv}{^{-1}}

\newcommand{\VA}{{[\mathrm{A}]}}
\newcommand{\QZ}{\mathrm{QZ}}
\newcommand{\QC}{\mathrm{QC}}

\newcommand{\PSL}{\mathrm{PSL}}

\newcommand{\CC}{\mathrm{C}}

\newcommand{\N}{\mathrm{N}}
\newcommand{\R}{\mathrm{R}}
\newcommand{\Z}{\mathrm{Z}}

\newcommand{\ie}{\textit{i.e.}\@\xspace}

\newcommand{\tdlc}{t.d.l.c.\@\xspace}
\newcommand{\tdlcsc}{t.d.l.c.s.c.\@\xspace}

\newcommand{\hji}{h.j.i.\@\xspace}

\newcommand{\con}{\mathrm{con}}

\newcommand{\Stab}{\mathrm{Stab}}

\newcommand{\sop}{(S0)\@\xspace}
\newcommand{\sdn}{(S1)\@\xspace}
\newcommand{\sur}{(S2)\@\xspace}

\newcommand{\FCbar}{$\overline{\text{FC}}$}
\newcommand{\triv}{\{1\}}

\newcommand{\nub}{\mathrm{nub}}

\newcommand{\bd}{\partial}
\newcommand{\lla}{\langle \langle}
\newcommand{\rra}{\rangle \rangle}

\newcommand{\ldlat}{\mathcal{LD}}

\newcommand{\sclass}{\mathscr S}

\newcommand{\stab}{\mathrm{stab}}
\newcommand{\rist}{\mathrm{rist}}

\newcommand{\lnorm}{\mathcal{LN}}

\newcommand{\lcent}{\mathcal{LC}}

\newcommand{\lpc}{\eta}
\newcommand{\epc}{\eta^*}

\newcommand{\Aut}{\mathrm{Aut}}

\newcommand{\Sym}{\mathrm{Sym}}

\newcommand{\Comm}{\mathrm{Comm}}

\newcommand{\Mon}{\mathrm{Mon}}

\newcommand{\Ker}{\mathrm{Ker}}

\newcommand{\ssl}{/\!\!/}

\newcommand{\bN}{\mathbf{N}}

\newcommand{\bQ}{\mathbf{Q}}

\newcommand{\bZ}{\mathbf{Z}}

\newcommand{\mcA}{\mathcal{A}}
\newcommand{\mcB}{\mathcal{B}}
\newcommand{\mcC}{\mathcal{C}}
\newcommand{\mcD}{\mathcal{D}}

\newcommand{\mcF}{\mathcal{F}}

\newcommand{\mcH}{\mathcal{H}}

\newcommand{\mcM}{\mathcal{M}}

\newcommand{\mcP}{\mathcal{P}}

\newcommand{\mcR}{\mathcal{R}}

\newcommand{\mcU}{\mathcal{U}}

\newcommand{\mfp}{\mathfrak{p}}

\newcommand{\mfS}{\mathfrak{S}}
\newcommand{\mfX}{\mathfrak{X}}

\begin{document}

\title{Locally normal subgroups \\of totally disconnected groups. \\ Part II: Compactly generated simple groups}


\author[1]{Pierre-Emmanuel Caprace\thanks{F.R.S.-FNRS research associate, supported in part by the ERC (grant \#278469)}}
\author[2]{Colin D. Reid\thanks{Supported in part by ARC Discovery Project DP120100996}}
\author[2]{George A. Willis\thanks{Supported in part by ARC Discovery Project DP0984342}}

\affil[1]{Universit\'e catholique de Louvain, IRMP, Chemin du Cyclotron 2, bte L7.01.02, 1348 Louvain-la-Neuve, Belgique}
\affil[2]{University of Newcastle, School of mathematical and physical sciences, Callaghan, NSW 2308, Australia}

\date{December 21, 2016}

\maketitle

\begin{abstract}
We use the structure lattice, introduced in Part I, to undertake a systematic study of the class $\mathscr S$ consisting of compactly generated, topologically simple, totally disconnected locally compact groups that are non-discrete. Given $G \in \mathscr S$, we show that compact open subgroups of $G$ involve finitely many isomorphism types of composition factors, and do not have any soluble normal subgroup other than the trivial one. By results of Part I, this implies that the centraliser lattice and local decomposition lattice of $G$ are Boolean algebras.  We show that the $G$-action on the Stone space of those Boolean algebras is minimal, strongly proximal, and micro-supported. Building upon those results, we obtain partial answers to the following key problems: Are all groups in $\mathscr S$ abstractly simple? Can a group in $\mathscr S$ be amenable? Can a group in $\mathscr S$ be such that the contraction groups of all of its elements are trivial? 
\end{abstract}

{\small 
\tableofcontents
}


\section{Introduction}


\subsection{Background}

The solution to Hilbert's fifth problem \cite[Theorem~4.6]{MZ} ensures that a connected locally compact group is in fact an inverse limit of Lie groups.  In particular, the general structure theory of connected locally compact groups largely reduces to that of simple Lie groups, soluble Lie groups, and compact groups (which are themselves inverse limits of compact Lie groups), and the structure of connected simple locally compact groups is thoroughly understood.

The possibility of a structure theory of locally compact groups beyond the connected case has become apparent over the last few decades.  At least from a local perspective, one has an immediate reduction to the case of groups that are \defbold{compactly generated}\index{compactly generated}, \ie that admit a compact generating set: any locally compact group $G$ is a directed union of compactly generated open subgroups.  In particular, every connected locally compact group is compactly generated.  There is also a developing structure theory of closed normal subgroups of locally compact groups, for which the base case is groups that are \defbold{topologically simple}\index{topologically simple}\index{simple!topologically}, that is, groups whose only closed normal subgroups are the identity subgroup and the whole group.  In contrast to topological simplicity, we say a topological group is \defbold{abstractly simple}\index{abstractly simple} if it has no proper non-trivial normal subgroup (including dense normal subgroups).

Results in \cite{CM} and their recent extensions in \cite{Reid&Wesolek:chief} suggest that, under mild assumptions that exclude discrete groups, compactly generated topologically simple locally compact groups play a critical role in the structure of general locally compact groups, generalising the status of simple Lie groups in the structure theory of connected locally compact groups.  (A complementary role, analogous to the role of soluble Lie groups in Lie theory, could be played by the class of elementary groups introduced by P. Wesolek \cite{WesolekElementary}.)  We shall exclude discrete simple groups from consideration, since their known behaviour precludes any structure theory having general scope. Indeed, the impossibility of such a theory can be given a precise mathematical formulation, see \cite{TV}.

We thus arrive at the class $\sclass$\index{S@$\sclass$} of non-discrete, compactly generated, topologically simple, totally disconnected, locally compact groups, which is the focus of the present article.  For the sake of brevity, we shall write \textit{\tdlc}\index{tdlc@\tdlc} for \textit{totally disconnected locally compact}.  Many specific families of examples of groups in $\sclass$ are known; see Appendix~\ref{sec:Examples}.  Our goal is to understand the general properties of the groups in $\sclass$.  Some of the results presented here have been announced in \cite{CRW-announcement}; the proofs  rely on general tools developed in \cite{CRW-Part1}.  To make this discussion precise, let us make a list of questions about groups in $\sclass$.  The answer to all of the following questions is known to be `no' for every known example of a group in $\sclass$, which gives some reason to believe the answer should be `no' in general.

For all the questions below, let $G \in \sclass$.

\begin{que}\label{que:abs_simple}
Can $G$ have a proper dense normal subgroup?
\end{que}

\begin{que}\label{que:amenable}
Can $G$ be amenable?
\end{que}

\begin{que}\label{que:anisotropic}
Can every element of $G$ have trivial contraction group?
\end{que}

\begin{que}\label{que:QZ}
Can $G$ have an open subgroup with a non-trivial centre?
\end{que}

\begin{que}\label{que:Cstable}
Can $G$ have an open subgroup with a non-trivial abelian normal subgroup?
\end{que}

A motivating discussion of each of these questions, as well as their interconnections, accompanies the statements of the main results in the subsequent sections of the introduction below. 
Question~\ref{que:QZ} was already answered negatively by Y. Barnea, M. Ershov and T. Weigel \cite{BEW}; Question~\ref{que:Cstable} is answered below (see Theorem~\ref{thmintro:QZ}).  The first three questions are open questions to which we obtain negative answers under more restrictive hypotheses.  Question~\ref{que:anisotropic} is a specialisation of \cite[Problem~4.1]{Willis}, which asks if $G \in \sclass$ can have the property that every element normalises a compact open subgroup.  By results of U. Baumgartner and G. Willis \cite{BaumgartnerWillis}, given an element $g$ of a \tdlc group $G$, then $g$ normalises a compact open subgroup if and only if both $\overline{\con(g)}$ and $\overline{\con(g\inv)}$ are compact.  Thus an example for Question~\ref{que:anisotropic} would also be an example for \cite[Problem~4.1]{Willis}.

By the recent groundbreaking work of K. Juschenko and N. Monod \cite{JM} and its extension due to V.~Nekrashevych \cite{Nekra}, there exist finitely generated infinite simple amenable groups, so the hypothesis in Question~\ref{que:amenable} that $G$ be non-discrete is important.  We emphasise however that the question whether a group $G \in \sclass$ can be amenable is naturally linked with the investigation of the commensurated subgroups of   finitely generated infinite simple amenable groups; see Remark~\ref{rem:amenable} below.

The following basic example (a similar construction is given in \cite[Proposition~3.2]{Willis}) gives some additional motivation for the focus on \emph{compactly generated} topologically simple groups in the questions above.

\begin{ex}\label{ex:non_cg}
Let $S = \mathrm{Sym}(\bZ)$ equipped with the permutation topology.  Let $A = \mathrm{Alt}(\bZ)$ be the group of permutations of $\bZ$ that can be written as a product of an even number of transpositions.  Then $A$ is clearly a dense normal subgroup of $S$.  It is well-known that $\mathrm{Alt}(\bZ)$ is a simple group, for essentially the same reasons as $\mathrm{Alt}(n)$ for $n \ge 5$.

Given $i \in \bZ$, let $\sigma_i$ be the $3$-cycle $(3i \; 3i+1 \; 3i+2)$ acting on $\bZ$, and let $B$ be the smallest closed subgroup of $S$ containing $\{\sigma_i \mid i \in \bZ\}$.  We see that $B$ is an abelian profinite group of exponent $3$.  Moreover, $B$ is commensurated by $A$.  We can thus construct the non-discrete \tdlc group $G = \langle A, B \rangle$, equipped with the topology so that $B$ is embedded as a compact open subgroup.

Now observe the following: $G$ is topologically but not abstractly simple (since $A$ is a proper dense normal subgroup of $G$); $G$ is locally finite, hence amenable; every element of $G$ has trivial contraction group and open centraliser in $G$; and $B$ is an open abelian subgroup of $G$.  So apart from the compact generation hypothesis, $G$ is an example for all five of the questions above.
\end{ex}

\subsection{Locally normal subgroups}
\label{sec:intro_loc_norm}

Our approach in studying the class $\sclass$ is based on the concept of \textbf{locally normal subgroup}\index{locally normal subgroup}, \ie a subgroup whose normaliser is open. {One motivation for considering this concept is the following classical fact (see \cite[Ch.~III, \S7, Prop.~2]{Bbki}): if $G$ is a $p$-adic Lie group\index{p-adic Lie group@$p$-adic Lie group}, then a subalgebra  of the $\mathbf Q_p$-Lie algebra of $G$ is an ideal if and only if it is  the Lie algebra of a compact locally normal subgroup of $G$. Thus compact locally normal subgroups may be viewed as a  group theoretic counterpart of ideals in  Lie theory.} In a general \tdlc group $G$, obvious examples of compact locally normal subgroups are provided by the trivial subgroup, or by compact open subgroups; these locally normal subgroups should be considered as \textbf{trivial}.\index{locally normal subgroup!trivial}

The following theorem answers Questions~\ref{que:QZ} and \ref{que:Cstable}.

\begin{thmintro}[See Theorem~\ref{thm:noqz}]\label{thmintro:QZ}
Let $G$ be a compactly generated \tdlc group that is topologically characteristically simple.  Suppose that $G$ is neither discrete nor compact.  Then the following hold.

\begin{enumerate}[(i)]
\item No element of $G \smallsetminus \triv$ has open centraliser in $G$. 

\item The only virtually soluble locally normal subgroup of $G$ is the identity subgroup $\triv$.
\end{enumerate}
\end{thmintro}

Part (i) in Theorem~\ref{thmintro:QZ} in the case of simple groups is \cite[Theorem 4.8]{BEW}; Part (ii) strengthens a result of Willis (\cite[Theorem~2.2]{Willis}). 

%

\begin{rem*}Independently, Wesolek (\cite{WesolekElementary}) has given another generalisation of \cite[Theorem~2.2]{Willis} and \cite[Theorem 4.8]{BEW} by obtaining a structure theory of second-countable \tdlc groups that have a soluble open subgroup.\end{rem*}

Given Theorem~\ref{thmintro:QZ} and the theory developed in \cite{CRW-Part1}, each $G \in \sclass$ admits three canonical bounded modular lattices (in the sense of partially ordered sets) arising from the arrangement of locally normal subgroups in $G$:

The \defbold{structure lattice}\index{structure lattice}\index{LN(G)@$\lnorm(G)$} $\lnorm(G)$ is defined to be the set of all closed locally normal subgroups, modulo the equivalence relation that $H \sim K$ if $H \cap K$ is open in $H$ and $K$, with ordering induced by inclusion of groups.

The  \textbf{centraliser lattice}\index{centraliser lattice}\index{LC(G)@$\lcent(G)$} $\lcent(G)$ consists of all elements of $\lnorm(G)$ represented by the centraliser of a locally normal subgroup.

The \textbf{local decomposition lattice}\index{local decomposition lattice}\index{LD(G)@$\ldlat(G)$} consists of all elements of $\lnorm(G)$ represented by a direct factor of an open subgroup.

In general one has
\[
\{0,\infty\} \subseteq \ldlat(G) \subseteq \lcent(G) \subseteq \lnorm(G).
\]
Note that all three lattices admit a natural action of $G$ by conjugation, with $0$ and $\infty$ as fixed points.  A point $\alpha \in \lnorm(G)$ is fixed by this action if and only if the compact representatives of $\alpha$ are commensurated by $G$, in other words, for every $g \in G$ the index $|K:K \cap gKg\inv|$ is finite.

\subsection{Abstract simplicity}\label{sec:AbstractSimplicity}

The question as to whether every group in $\sclass$ is abstractly simple is natural enough in itself, but is also important for the general structural theory of \tdlc groups.  We make the following observation, inspired by \cite{CM}. Let $G$ be a compactly generated locally compact group and let $N_1, N_2$ be two distinct closed normal subgroups of $G$ that are maximal. Hence $S_1 = G/N_1$ and $S_2=G/N_2$ are two groups belonging to $\sclass$ and occurring as quotients of $G$. Can one conclude that the product $S_1 \times S_2$ also occurs as a quotient of $G$? The answer would be positive if one had $G= N_1 N_2$ or, equivalently, if the image of $N_2$ in $G/N_1 = S_1$ were full. However, the image of $N_2$ in $G/N_1 = S_1$ is a non-trivial normal subgroup, but it need not be closed \textit{a priori}. In this way, we see that the question of abstract simplicity of the groups in $\sclass$ pops up naturally when considering how a general compactly generated \tdlc group is built out of simple pieces. In fact, this specific question is also a motivation to study the structure lattice. Indeed, since $N_2$ is a closed normal subgroup of $G$, each compact open subgroup of $N_2$ is a compact locally normal subgroup of $G$ which is commensurated by $G$. Therefore its image in $S_1$ is a commensurated compact locally normal subgroup of $S_1$, so that we could conclude that $G=N_1 N_2$ provided that $S_1$ (or $S_2$) had no non-trivial fixed point in its structure lattice.

Let us record three properties that, taken together, are equivalent to abstract simplicity of a non-discrete second-countable \tdlc group (see Theorem~\ref{baireabs}).

\begin{description}
\item{\sop} $G$ has no proper open normal subgroups.\index{S0@\sop}
\item{\sdn} Every non-trivial normal subgroup of~$G$ contains an infinite commensurated compact locally normal subgroup of $G$.\index{S1@\sdn}
\item{\sur} Every infinite commensurated compact subgroup of~$G$ is open.\index{S2@\sur}
\end{description}

Non-discrete topologically simple groups clearly satisfy \sop, so $G \in \sclass$ is abstractly simple if and only if it satisfies \sdn and \sur.

\subsection{A partition of the class $\sclass$}

The properties of the lattices lead to a partition of $\sclass$ into five classes as follows:

\begin{thmintro}[See \S\ref{sec:types}]\label{thmintro:types}
Let $G \in \sclass$.  Then $G$ is of exactly one of the following types:
\begin{itemize}
\item \defbold{locally \hji}: We have $\lnorm(G) = \{0,\infty\}$, or equivalently, every compact open subgroup of $G$ is \textbf{hereditarily just-infinite}\index{hereditarily just-infinite}\index{hji@\hji} (\hji), where a profinite group is said to be \hji if every non-trivial closed locally normal subgroup is open.
\item \defbold{atomic type}: $|\lnorm(G)| > 2$ but $\lcent(G) = \{0,\infty\}$, there is a unique least element of $\lnorm(G) \smallsetminus \{0\}$, the action of $G$ on $\lnorm(G)$ is trivial and $G$ is not abstractly simple.
\item \defbold{non-principal filter type} (abbreviated by \textbf{NPF type}): The set $\lnorm(G) \smallsetminus \{0\}$ is a non-principal filter in $\lnorm(G)$ and $\lcent(G) = \{0,\infty\}$.
\item \defbold{(strictly) weakly decomposable}: $|\lcent(G)|>2$, but $\ldlat(G) = \{0,\infty\}$.
\item \defbold{locally decomposable}: $|\ldlat(G)|>2$.
\end{itemize}
Moreover, the type of $G$ is completely determined by the isomorphism type of $\lnorm(G)$ as a poset.  In particular, if $G,H \in \sclass$ have isomorphic open subgroups  then $G$ and $H$ are of the same type.
\end{thmintro}

Table~1 below summarises the progress we have made towards answering the first three questions given in the initial discussion.  The following conventions are used in the table. We indicate both definitive results for all $G \in \sclass$ of the given type (`Yes' and `No') and cases where no definitive result is known, but either all known examples satisfy the property or all known examples fail to satisfy the property (`Yes?' and `No?').  There are no entries in the table for which some examples are known to satisfy the property and others are known not to satisfy it.  We recall from Section~\ref{sec:AbstractSimplicity} that $G \in \sclass$ is abstractly simple if and only if it satisfies \sdn and \sur.  Property \sur implies that the set $\lnorm(G)^G$ of fixed points of $G$ acting on $\lnorm(G)$ is just $\{0,\infty\}$.  In particular, a necessary condition for abstract simplicity is that either $G$ is locally \hji or $G$ has faithful action on $\lnorm(G)$; we denote the latter by `$\curvearrowright \lnorm$' in the table.  We say $G$ is \defbold{anisotropic} if every element of $G$ has trivial contraction group.

\begin{table}[h!]
\begin{center}
\def\arraystretch{1.5}
\begin{tabular}{ c | c | c | c | c | c }
Property & locally h.j.i. & atomic & NPF & weakly dec. & locally dec. \\ \hline
$\curvearrowright \lnorm$ & No & No & ? & Yes & Yes \\ \hline
\sdn & Yes? & ? & ? & Yes & Yes \\ \hline
\sur & Yes & No & ? & Yes? & Yes \\ \hline
{abstractly simple} & Yes? & No & ? & Yes? & Yes \\ \hline
{anisotropic} & No? & ? & ? & No & No \\ \hline
amenable & No? & ? & ? & No & No \\
\end{tabular}
\label{fig:conjectures}
\caption{Properties and open questions}
\end{center}
\end{table}

Given $G \in \sclass$, there is a natural division between the case where $\lcent(G)$ is trivial (corresponding to the first three types; an equivalent condition is that $\lnorm(G) \smallsetminus \{0\}$ is closed under meets), and the case where $\lcent(G)$ is non-trivial.  In the latter case we obtain stronger results, and this case also includes most known examples.  Indeed, the only known examples with trivial centraliser lattice are linear algebraic groups over local fields, which are locally \hji; thus no examples at all are known in the atomic or NPF cases.  We suspect that no groups in $\sclass$ of atomic type exist, as such a group cannot be abstractly simple; indeed, it would admit an embedding of a group in $\sclass$ of some other type as a proper dense normal subgroup (see Theorem~\ref{thmintro:FixedPointsAtomic} below).  On the other hand, it could well be the case that some of the known examples of simple Kac--Moody groups are of NPF type.  See Appendix~\ref{sec:Examples} for a further discussion of the known examples.

It seems that a major challenge in advancing our understanding of the class $\sclass$ is to construct simple groups with trivial centraliser lattice, or at least with trivial decomposition lattice, that are significantly different from the ones we already know (see Remark~\ref{rem:amenable} for a potential source of a such examples). 


\subsection{Fixed points in the structure lattice}

As noted in Section~\ref{sec:AbstractSimplicity}, a necessary condition for $G \in \sclass$ to be abstractly simple is \sur, which is equivalent to the requirement that $\lnorm(G)^G = \{0,\infty\}$.  The next result shows that fixed points of $G$ acting on $\lnorm(G)$ play an important role in the structure of orbits of $G$ on $\lnorm(G)$.

\begin{thmintro}[See Theorem~\ref{boxcor}]\label{thmintro:FixedPointsJoin}
Let $G\in \sclass$.\index{structure lattice!fixed points in}  For each $\alpha \in \lnorm(G)$, there exist $g_1, \dots, g_n \in G$ such that 
$g_1\alpha \vee \dots \vee g_n\alpha$ is fixed  by $G$.
\end{thmintro}

The following corollary is clear.

\begin{corintro}
Let $G\in \sclass$ and let $H$ be a normal subgroup of $G$.  Suppose $H$ contains an infinite  compact locally normal subgroup.  Then $H$ contains an infinite  \emph{commensurated} compact locally normal subgroup.  In particular, if every non-trivial normal subgroup of $G$ contains an infinite compact locally normal subgroup, then $G$ satisfies \sdn.
\end{corintro}

We also note a finite generation property of groups $G \in \sclass$ with \sur; in particular this applies to every abstractly simple group in $\sclass$.

\begin{corintro}[See Corollary~\ref{cor:FixedPointsGeneration}]\label{corintro:FixedPointsGeneration}
Let $G \in \sclass$ such that $G$ satisfies \sur.  Let $H$ be a non-trivial compact locally normal subgroup of $G$.  Then there exist $g_1, \dots, g_n \in G$ such that 
\[
G = \langle g_1Hg\inv_1, g_2Hg\inv_2, \dots, g_nHg\inv_n \rangle.
\]
\end{corintro}

A further restriction is that the set of non-zero fixed points of $G$ acting on $\lnorm(G)$ forms a filter (Lemma~\ref{lnormfix}).  Moreover, we can show that indeed $G$ has property \sur under some additional hypotheses on $G$ (see Corollary~\ref{cor:FixedPointsSimple}).

Other than locally \hji groups, we do not know of any groups in $\sclass$ such that there is a \hji compact locally normal subgroup or there is a minimal non-zero element of $\lnorm(G)$; the following puts further restrictions on the possible structure of such groups.

\begin{thmintro}[See \S\ref{sec:mono}]\label{thmintro:FixedPointsAtomic}
Let $G\in \sclass$.
\begin{enumerate}[(i)]
\item Any hereditarily just-infinite\index{hereditarily just-infinite} compact locally normal subgroup is commensurated by $G$.
\item Suppose there is a minimal non-zero element of $\alpha$ of $\lnorm(G)$.  Then $G$ is either locally \hji or of atomic type.
\item Suppose that $G$ is of atomic type.  Then there exists $S \in \sclass$, unique up to isomorphism, and a continuous homomorphism $\phi \colon  S \rightarrow G$ such that   $S$ is not of atomic type and $\phi(S)$ is a proper dense normal subgroup of $G$ containing a representative of the atom of $\lnorm(G)$.  In particular, $G$ is not abstractly simple.
\end{enumerate}
\end{thmintro}

%
%

\subsection{Local composition factors}

We next present some additional algebraic features of compact locally normal subgroups of groups in $\sclass$.

For any compactly generated \tdlc group $G$, we observe (Proposition~\ref{localprime:short}) that each compact open subgroup of $G/K$ has finitely many isomorphism types of composition factors, where $K$ is a compact normal subgroup that can be taken to lie in any given identity neighbourhood.  This observation confirms a conjecture formulated in \cite[\S4]{Willis01} and naturally leads to the notion of the \defbold{local prime content} of $G/K$, which is the unique finite set $\eta = \lpc(G/K)$ of primes such that every compact open subgroup of $G/K$ is virtually pro-$\eta$ and has an infinite pro-$p$ subgroup for each $p \in \eta$.

For groups in $G \in \sclass$, we obtain additional control over the local prime content of compact locally normal subgroups of $G$.  In particular, the local prime content of $G$ can be recovered from any non-trivial locally normal subgroup.  We can control the presence of non-abelian composition factors of compact open subgroups in a similar manner.

\begin{thmintro}[See Theorem~\ref{localprimetopsimp}]\label{thmintro:algebraicLN}
Let $G \in \sclass$. 
\begin{enumerate}[(i)]
\item If a non-trivial compact locally normal subgroup of $G$ is a pro-$\pi$ group for some set of primes $\pi$, then every compact open subgroup of $G$ is virtually pro-$\pi$. In particular, for all $p \in \lpc(G)$, each closed locally normal subgroup $L \neq \triv$ has an infinite pro-$p$ subgroup.

\item If a non-trivial compact locally normal subgroup of $G$ is prosoluble, then every compact open subgroup of $G$ is virtually prosoluble.
\end{enumerate}
\end{thmintro}

Theorem~\ref{thmintro:algebraicLN}  was inspired by a result of M. Burger and S. Mozes, who obtained a similar statement in the case of locally primitive tree automorphism groups, see  \cite[Prop.~2.1.2]{BurgerMozes}.

\subsection{Micro-supported actions}\label{sec:MicroSupp}

An  action of a  group $G$ by homeomorphisms on a (possibly connected) topological space $X$ is called \textbf{micro-supported}\index{micro-supported}   if for every non-empty open subset $Y$ of $X$, the pointwise stabiliser of the complement $X \setminus Y$ acts non-trivially on $Y$. A   subset $V$ of $X$ is called \defbold{compressible}\index{compressible} if it is non-empty and if for any non-empty open subset $Y$, there exists $g \in G$ such that $gV \subset Y$.  Numerous   
natural  transformation groups, like homeomorphism groups of various topological spaces (e.g.   closed manifolds, the rationals, the Cantor space, \dots) or diffeomorphism groups of manifolds, happen to be   micro-supported and to have a compressible open set. This has been exploited repeatedly in the literature to show that many of those transformation groups are simple or almost simple, see D.~Epstein's paper~\cite{Epstein} and references therein, or to show \textit{reconstruction theorems}, see M.~Rubin's paper~\cite{Rubin} and references therein\footnote{In \cite[Definition~2.3]{Rubin}, the term \emph{regionally disrigid} is used to qualify what we have called a \emph{micro-supported} action}. A related result, known as \textbf{Higman's simplicity criterion}\index{Higman's simplicity criterion}, valid for abstract groups, may be consulted in \cite[Proposition~C10.2]{BieriStrebel}.

The following result, whose short and self-contained proof will be provided in Section~\ref{sec:SimplicityCriterion} below, provides a uniform explanation  that the conjunction of these two properties naturally yields simple groups. A closely related statement (with formally stronger hypotheses) appears in \cite[Proposition~4.3]{Nekra2013}.

\begin{propintro}\label{prop:SimplicityCriterion}
Let $G$ be a subgroup of the  homeomorphism group of a Hausdorff topological space $X$. If the $G$-action is micro-supported and has a   compressible open set, then the intersection $M$ of all non-trivial normal subgroups of $G$ is  non-trivial. 

If in addition the $M$-action   admits a  compressible open set, then $M$ is simple. 
\end{propintro}

It turns out that examples of transformation groups satisfying the hypotheses of Proposition~\ref{prop:SimplicityCriterion} may also be found among  non-discrete \tdlc groups. The prototypical case is provided by the action of the full automorphism group of a regular locally finite tree $T$ on the set of ends of $T$. The fact that the latter group is almost simple  was first observed by J.~Tits \cite{Tits70}; we refer to Appendix~\ref{sec:Examples} for many other related examples of groups in $\sclass$.

Our next result  shows that for groups $G \in \sclass$, all micro-supported actions on compact totally disconnected spaces are controlled by the $G$-action on the centraliser lattice $\lcent(G)$. Moreover, somewhat surprisingly, the existence of a compressible open set happens to be automatic for micro-supported actions. In order to formulate  precise statements, we recall that the lattices $\lcent(G)$ and $\ldlat(G)$  are both Boolean algebras by \cite[{Theorem~I}]{CRW-Part1}.  By the Stone representation theorem, every Boolean algebra $\mathcal A$ is canonically isomorphic to the lattice of clopen sets of a \textbf{profinite space}\index{profinite space}, \ie a compact zero-dimensional space, which is called the \textbf{Stone space}\index{Stone space} of $\mathcal A$, denoted by $\mfS(\mathcal A)$ and can be constructed as the set of ultrafilters on $\mathcal A$. We shall next see that the dynamics of the $G$-action on $\mfS(\lcent(G))$ is  rich and can be exploited to shed light on the global algebraic properties of $G$.

We also need to  recall basic definitions from topological dynamics. An action of a locally compact group $G$ on a compact space $\Omega$  by homeomorphisms is called \defbold{continuous} if the corresponding homomorphism of $G$ to the homeomorphism group of $\Omega$, endowed with the topology of uniform convergence, is continuous. This implies that the $G$-action on $\Omega$ is continuous,   \ie the action map $G \times \Omega \to \Omega$ is continuous. The $G$-action on $\Omega$ is called \defbold{minimal}\index{minimal action}\index{action!minimal} if every orbit is dense. It is called \defbold{strongly proximal}\index{strongly proximal action}\index{action!strongly proximal} if the closure of every $G$-orbit in the space of probability measures on $\Omega$, endowed with the weak-* topology, contains a Dirac mass. 
Suppose that $\Omega$ is totally disconnected and let $K$ be the kernel of the $G$-action.  The action is called  \textbf{weakly decomposable}\index{weakly decomposable}  (resp.  \textbf{locally weakly decomposable}\index{locally weakly decomposable action}\index{action!locally weakly decomposable}\index{weakly decomposable!locally}) if {it is continuous and if} for every clopen proper subset $\Omega'$ of~$\Omega$, the quotient $F/K$ of the pointwise stabiliser $F$ of $\Omega'$ is non-trivial (resp. non-discrete). Thus, for group actions on profinite spaces, `micro-supported' and `weakly decomposable' are synonyms.

\begin{thmintro}[See \S\ref{sec:strongprox}]\label{thmintro:WeaklyDecomposable}
Let $G \in \sclass$ and let $\Omega$ be the Stone space associated with the centraliser lattice $\lcent(G)$. Then:
\begin{enumerate}[(i)]
\item The $G$-action on $\Omega$ is continuous, minimal, strongly proximal and locally weakly decomposable; moreover $\Omega$ contains a compressible clopen subset. 

\item Given a profinite space $X$ with a continuous  $G$-action, the $G$-action on $X$ is  micro-supported if and only if there is a continuous $G$-equivariant surjective map $\Omega \to X$. In particular every continuous micro-supported $G$-action on a profinite space is minimal, strongly proximal, locally weakly decomposable and has a compressible open set. 

\end{enumerate}
\end{thmintro}

Notice that $\Omega = \mfS(\lcent(G))$ is   a singleton if and only if $\lcent(G) = \{0, \infty\}$. It should be emphasised that, although any group $G$ belonging to $\sclass$ is automatically second-countable (this follows from \cite{KK}), the centraliser lattice for $G \in \sclass$ is not necessarily countable, so $\Omega$ is not necessarily countably based.  For example, one obtains an uncountable decomposition lattice for \tdlc groups $G$ such that $\QZ(G)=\triv$ and some compact open subgroup of $G$ splits as a direct product with infinitely many infinite factors: such a compact open subgroup thus has uncountably many direct factors, each of which represents a distinct class in $\ldlat(G)$.  Examples of this kind arise as Burger--Mozes' universal groups acting on trees with local action prescribed by a suitable finite permutation group (see Appendix~\ref{sec:Examples}).

Theorem~\ref{thmintro:WeaklyDecomposable} notably implies that $\lcent(G) \neq \{0, \infty\}$ if and only if $G$ admits a continuous micro-supported action on a profinite space containing more than one point. Whenever this is the case, Theorem~\ref{thmintro:WeaklyDecomposable} has several consequences on the structure of $G$ which we now proceed to describe. The first one relates to Question~\ref{que:amenable} and stems from the incompatibility between amenability and minimal strongly proximal actions:

\begin{corintro}\label{corintro:amen}
Let  $G \in \sclass$. Any closed cocompact amenable subgroup of $G$ fixes a point in $\Omega = \mfS(\lcent(G))$. In particular, if  $\lcent(G) \neq \{0, \infty\}$, then $G$ is not amenable\index{amenable}, and if $G$ contains a closed cocompact amenable subgroup, then the $G$-action on $\Omega$ is transitive.
\end{corintro}

Here is another consequence of the dynamical properties highlighted in Theorem~\ref{thmintro:WeaklyDecomposable},  obtained by a standard ping-pong argument. 

\begin{corintro}\label{corintro:free}
Let  $G \in \sclass$ be such that $\lcent(G) \neq \{0, \infty\}$. Then $G$ contains a non-abelian discrete free subsemigroup.\index{free subsemigroup}
\end{corintro}

We recall that a \textbf{Polish group}\index{Polish group} is a topological group whose underlying topology is homeomorphic to a separable complete metric space. Every second countable locally compact group is a Polish group, but not conversely. 
Another consequence of Theorem~\ref{thmintro:WeaklyDecomposable} is that we can apply results from \cite{CRW-Part1} to obtain topological rigidity results for groups in $\sclass$:

\begin{corintro}\label{corintro:TopolRigidity}
Let $G \in \sclass$. If  $\lcent(G) \neq \{0, \infty\}$, then the topology of $G$ is the unique $\sigma$-compact locally compact group topology on $G$ and the unique Polish group topology on $G$. In particular every automorphism of $G$ is continuous.  

If in addition $G$ is abstractly simple (for instance if $\ldlat(G)  \neq \{0, \infty\}$ {by Theorem~\ref{thmintro:AbstractSimplicity} below}), then there are exactly two locally compact group topologies on $G$, namely the original one and the discrete topology.
\end{corintro}

Notice that the corresponding result does \textit{not} hold for all simple Lie groups, since the field of complex numbers admits discontinuous automorphisms; we refer to \cite{Kramer} for a comprehensive statement in the case of semi-simple Lie groups.

\subsection{Contraction groups and abstract simplicity}

Let $G$ be topological group. Recall that the \textbf{contraction group}\index{contraction group}\index{con(g)@$\con(g)$} of $g$ is defined by 
$$\con(g) = \{ x \in G \; | \; \lim_{n \to \infty} g^n x g^{-n} = e\}.$$ 
In every known example of $G \in \sclass$, some element of $G$ has non-trivial contraction group; whether this is true for all $G \in \sclass$ is Question~\ref{que:anisotropic} above.
 
In simple Lie groups or in simple algebraic groups over local fields, the contraction group of every element is known to coincide with the unipotent radical of some parabolic subgroup, and is thus always closed. On the other hand, recent results indicate that many non-linear examples of groups in $\sclass$ have non-closed contraction groups (see \cite{BRR} and \cite{Marquis}), while in some specific cases the closedness of contraction groups yields connections with finer algebraic structures related to linear algebraic groups (see \cite{CDM2}). 

As a partial answer to Question~\ref{que:anisotropic}, we prove the following:

\begin{thmintro}[See \S\ref{sec:mono}]\label{thmintro:con}
Let $G \in \sclass$. Then  $\lcent(G) \neq \{0, \infty\}$ if and only if $G$ has a closed subgroup of the form $R = \prod_{i \in \bZ}K_i \rtimes \langle g \rangle$, where $K_i$ is an infinite compact locally normal subgroup of $G$ and $g$ acts by shifting indices.
\end{thmintro}

\begin{corintro}\label{corintro:con}
Let $G \in \sclass$. If $\lcent(G) \neq \{0, \infty\}$, then the contraction group of some element of $G$ is not closed (and hence non-trivial). 
\end{corintro}

In relation with the question of abstract simplicity, we observe that the combination of Proposition~\ref{prop:SimplicityCriterion} and Theorem~\ref{thmintro:WeaklyDecomposable} directly implies that a group in $\sclass$ with a non-trivial centraliser lattice has a smallest dense normal subgroup which is abstractly simple. Below, we shall obtain an alternative proof of the following  stronger statements by combining the main result of  \cite{CRW-TitsCore}  with the conclusions of Corollary~\ref{corintro:con}. 

\begin{thmintro}[See \S\ref{sec:mono}]\label{thmintro:PropertyS1}
Let $G \in \sclass$.  Suppose that at least one of the following holds:
\begin{enumerate}[(i)]
\item $\lcent(G) \neq \{0, \infty\}$;
\item Some compact open subgroup of $G$ is finitely generated as a profinite group.
\end{enumerate}
Then $G$ has property \sdn.  Consequently, $G$ is abstractly simple if and only if every non-trivial commensurated compact subgroup of $G$ is open. 
\end{thmintro}

\begin{thmintro}[See \S\ref{sec:abs_simple}]\label{thmintro:AbstractSimplicity}
Let $G \in \sclass$.  Suppose that at least one of the following holds:
\begin{enumerate}[(i)]
\item $\ldlat(G) \neq \{0, \infty\}$;
\item There is a compact open subgroup $U$ of $G$ such that $U$ is finitely generated as a profinite group and $\overline{[U,U]}$ is open in $G$;
\item Some non-trivial compact locally normal subgroup of $G$ is finitely generated as a profinite group and every non-trivial commensurated compact subgroup of $G$ is open.
\end{enumerate}
Then $G$ is abstractly simple.\index{abstract simplicity}
\end{thmintro}

\subsection*{Acknowledgement}We are grateful to Simon Smith for communicating to us the details of his construction of uncountably many isomorphism classes in $\sclass$ from \cite{SmithNew} prior to publication.  We also thank the anonymous referee for a very thorough and helpful report, with many good suggestions for improvements.

\section{Preliminaries}

In this section we recall some definitions and results from \cite{CM}, \cite{CRW-Part1} and \cite{CRW-TitsCore} that will be used below. 

\subsection{Open compact subgroups and locally normal subgroups}

\begin{defn}
\label{defn:locnorm}
Let $G$ be a totally disconnected, locally compact (\tdlc) group.  Write $\mcB(G)$ for the set of compact open subgroups of $G$, and let $U \in \mcB(G)$.\index{B(G)@$\mcB(G)$}

Let $H$ be a subgroup of $G$.  We say that $H$ is \textbf{commensurate}\index{commensurate} to a subgroup $K \leq G$ if the indices $[H: H \cap K]$ and $[K : H \cap K]$ are both finite. Write $[H]$ for the class of all subgroups $K$ of $G$ such that $K \cap U$ is commensurate with $H \cap U$.  Say $H$ and $K$ are \defbold{locally equivalent}\index{locally equivalent} if $[H]=[K]$.  There is a natural partial ordering on local equivalence classes, given by $[H] \ge [K]$ if $H \cap K \cap U$ is commensurate to $K \cap U$.

A \defbold{locally normal}\index{locally normal subgroup} subgroup of $G$ is a subgroup $H$ such that $\N_G(H)$ is open.  The \defbold{structure lattice}\index{structure lattice} $\lnorm(G)$ of $G$ is the set of local equivalence classes of closed locally normal subgroups of $G$.\end{defn}

The partial order on local equivalence classes given in Definition~\ref{defn:locnorm} may be reformulated, for $\alpha,\beta\in \lnorm(G)$, as $\alpha \le \beta$ if $H \le K$ for some $H \in \alpha$ and $K \in \beta$. 

As noted in \cite{CRW-Part1}, the structure lattice is a modular lattice, admitting an action of the group $\Aut(G)$ of topological group automorphisms by conjugation, and is a local invariant of the group in that $\lnorm(G) = \lnorm(H)$ for any open subgroup $H$ of $G$.  If $G$ is not discrete, there are two `trivial' elements of $\lnorm(G)$, namely the class of the trivial group, which we will denote $0$, and the class of the compact open subgroups, which we will denote $\infty$.  More generally the symbols $0$ and $\infty$ will be used to denote the minimum and maximum elements respectively of a bounded lattice.  When discussing subsets of $\lnorm(G)$, we will regard a subset as non-trivial if it contains an element other than $0$ and $\infty$.

\subsection{Topological countability}

A topological space is \defbold{$\sigma$-compact} if it is a union of countably many compact subspaces; \defbold{first-countable} if at every point, there is a countable base of neighbourhoods; and \defbold{second-countable} if there is a countable base for the topology.

The following well-known facts (the first two of which are easily verified) will for the most part allow us to restrict attention to totally disconnected locally compact \emph{second-countable} (\tdlcsc)\index{t.d.l.c.s.c.} groups.

\begin{lem}\label{lem:KK}Let $G$ be a \tdlc group.
\begin{enumerate}[(i)]
\item $G$ is second-countable if and only if it is $\sigma$-compact and first-countable.

\item If $G$ is compactly generated, then $G$ is $\sigma$-compact. 

\item (See \cite{KK}) If $G$ is $\sigma$-compact, then for every identity neighbourhood $U$ in $G$, there exists $K \subseteq U$ such that $K$ is a compact normal subgroup of $G$ and $G/K$ is second-countable.
\end{enumerate}
\end{lem}

We note in particular the following.

\begin{cor}
Every $G \in \sclass$ is second-countable.
\end{cor}

\subsection{Quasi-centralisers}

\begin{defn}\label{cqdef}Let $G$ be a \tdlc group and let $H$ and $K$ be subgroups of $G$.  The \defbold{quasi-centraliser}\index{quasi-centraliser}\index{QCH(K)@$\QC_H(K)$} of $K$ in $H$ is given by
\[ \QC_H(K) : = \bigcup_{U \in \mcB(G)} \CC_H(K \cap U),\]
where $\mcB(G)$ is the set of compact open subgroups of $G$. {Notice that the quasi-centraliser $\QC_H(K)$ depends only on the local equivalence class of $K$; we shall therefore use the notation 
$$\QC_H([K]) = \QC_H(K).$$ 
If the class $[K]$ is fixed under the conjugation action of $H$, then the quasi-centraliser $\QC_H(K)$ is normal in $H$. A typical example is provided by the  
\defbold{quasi-centre}\index{quasi-centre}\index{QZ(G)@$\QZ(G)$} $\QZ(H)$ of $H$, which is just $\QC_H(H)$.}
\end{defn}

\begin{defn}Let $G$ be a \tdlc group.  A closed subgroup $H$ is \defbold{C-stable}\index{C-stable} in $G$ if $\QC_G(H) \cap \QC_G(\CC_G(H))$ is discrete.  Say a \tdlc group $G$ is \defbold{locally C-stable}\index{locally C-stable} if all closed locally normal subgroups of $G$ are C-stable in $G$.\end{defn}

\begin{prop}[See {\cite[Theorem~3.19]{CRW-Part1}}]\label{cstab}
Let $G$ be a \tdlc group.  Then $G$ is locally C-stable if and only if $\QZ(G)$ is discrete and every non-trivial abelian compact locally normal subgroup of $G$ is contained in $\QZ(G)$. 

If in addition $\QZ(G)=\triv$, then for every closed locally normal subgroup $H$ of $G$, we have $\QZ(H)=\triv$ and $\QC_G(H) = \CC_G(H)$.
\end{prop}

\subsection{Local decomposition}

\begin{defn}
Let $G$ be a \tdlc group.  Define the \defbold{local decomposition lattice}\index{local decomposition lattice}\index{LD(G)@$\ldlat(G)$} $\ldlat(G)$ of $G$ to be the subset of $\lnorm(G)$ consisting of elements $[K]$ where $K$ is a closed direct factor of some compact open subgroup of $G$. The \defbold{complementation map} on $\ldlat(G)$ sends $[K]$ to $[L]$ where $L\in \lnorm(G)$ satisfies $K\cap L = \triv$ and $KL$ is open. 
\end{defn}

\begin{defn}\label{def:LC}
Let $G$ be a locally C-stable \tdlc group.  Define the map $\bot: \lnorm(G) \rightarrow \lnorm(G)$ to be given by $\alpha \mapsto [\QC_G(\alpha)]$.  Define the \defbold{centraliser lattice}\index{centraliser lattice}\index{LC(G)@$\lcent(G)$} $\lcent(G)$ be the set $\{\alpha^\bot \mid \alpha \in \lnorm(G)\}$ together with the map $\bot$ restricted to $\lcent(G)$, partial order inherited from $\lnorm(G)$ and binary operations $\wedge_c$ and $\vee_c$ given by:
\[ \alpha \wedge_c \beta = \alpha \wedge \beta \]
\[ \alpha \vee_c \beta = (\alpha^\bot \wedge \beta^\bot)^\bot.\]
In general we will write $\vee$ instead of $\vee_c$ in contexts where it is clear that we are working inside the centraliser lattice.\end{defn}

\begin{thm}[{\cite[Theorems~4.5 and~5.2]{CRW-Part1}}]\label{thm:Boolean-part1}
Let $G$ be a \tdlc group.
\begin{enumerate}[(i)]

\item Suppose that $G$ has trivial quasi-centre. Then $\ldlat(G)$ is a Boolean algebra and $\bot$ induces the complementation map on $\ldlat(G)$.

\item Suppose in addition that $G$ has no non-trivial compact abelian locally normal subgroups.  Then $\lcent(G)$ is a Boolean algebra, $\bot$ induces the complementation map on $\lcent(G)$ and $\ldlat(G)$ is a subalgebra of $\lcent(G)$.

\end{enumerate}
\end{thm}

\subsection{Decomposing locally compact groups into simple pieces}

In this article, we adopt the general convention that a \textbf{maximal}\index{maximal} subgroup of the group $G$ is one that is maximal amongst the \emph{proper} subgroups of $G$, and a \textbf{minimal}\index{minimal} subgroup is one that is minimal amongst the \emph{non-trivial} subgroups of $G$.  A similar convention applies when discussing more restrictive classes of subgroups, such as closed normal subgroups.

Topologically simple groups appear naturally in the upper structure of compactly generated \tdlc groups.  The following is an analogue of the well-known fact that a non-trivial finitely generated group has a simple quotient.

\begin{thm}[{\cite[Theorem~A]{CM}}]\label{cmquotthm}Let $G$ be a compactly generated locally compact  group.  Then exactly one of the following holds.
\begin{enumerate}[(i)]
\item $G$ has an infinite discrete quotient.
\item $G$ has a cocompact closed normal subgroup that is connected and soluble.
\item $G$ has a cocompact closed normal subgroup $N$ such that $N$ has no infinite discrete quotient, but $N$ has exactly $n$ non-compact topologically simple quotients, where $0 < n < \infty$.\end{enumerate}
\end{thm}

In particular, Theorem~\ref{cmquotthm} provides information on \textit{maximal} closed normal subgroups of $G$. The following related result is concerned with \textit{minimal} ones.  A version of this result was asserted as part of \cite[Proposition~2.6]{CM}; however the statement and proof in the original published version contained an error.  To avoid confusion, we give a proof here of a corrected version of the proposition that will be sufficient for the purposes of the present paper.

\begin{prop}\label{cmsoc}
Let $G$ be a compactly generated \tdlc group which possesses no non-trivial discrete or abelian normal subgroup, and such that there exists an open subgroup $U$ of $G$ for which $\bigcap_{g \in G}gUg\inv = \triv$.  Then every non-trivial closed normal subgroup of $G$ contains a minimal one, and the set $\mcM$ of minimal closed normal subgroups is finite.
\end{prop}

\begin{proof}
By \cite{Abels} (see also Proposition~\ref{prop:Abels} below), there is a locally finite connected graph $\Gamma$ on which $G$ has a continuous vertex-transitive proper action.  The hypothesis that an open subgroup of $G$ has trivial core ensures that we can choose $\Gamma$ so that $G$ acts faithfully; the hypothesis that $G$ has no non-trivial discrete normal subgroup ensures that no non-trivial normal subgroup acts freely on $\Gamma$.  Fix $v \in V\Gamma$ and write $N(v)$ for all vertices in $\Gamma$ at distance at most $1$ from $v$.  Let $A$ be the stabiliser of $v$ in $G$ and let $B$ be the subgroup fixing pointwise the set $N(v)$.  Note that the assumptions on the action ensure that $A$ is a compact subgroup of $G$ and $B$ is an open subgroup of $A$.

Let $\mcC$ be a family of non-trivial closed normal subgroups of $G$.  Suppose that $\mcC$ is filtering, that is, for any finite subset $\{K_1,\dots,K_n\}$ of $\mcC$, there exists $K \in \mcC$ such that $K \le \bigcap^n_{i=1}K_i$.  We claim that $\mcC$ has non-trivial intersection.

Let $K \in \mcC$.  Since $K$ acts faithfully but not freely on $\Gamma$, there exist adjacent vertices $w,x \in V\Gamma$ and $L \le K$ such that $L$ fixes $x$ but not $w$.  Since $G$ acts vertex-transitively and $K$ is normal, by conjugating in $G$ we can assume $x =v$ and take $L = K \cap A$; since $L$ does not fix $w$ we have $K \cap A \nleq B$.  In particular, we see that $K$ has non-empty intersection with the compact set $A \smallsetminus B$.  If we let $R = \bigcap_{K \in \mcC}K$, it now follows by compactness that $R \cap A \nleq B$, so in particular $R$ is non-trivial and the claim is proved.

Since $\mcC$ can in particular be any chain of non-trivial closed normal subgroups of $G$, we conclude by Zorn's lemma that every non-trivial closed normal subgroup of $G$ contains a minimal one.

It remains to suppose that the set $\mcM$ of minimal closed normal subgroups is infinite and obtain a contradiction.  Let $\mcF$ be the set of finite subsets of $\mcM$.  For each $F \in \mcF$, let $M_F = \overline{\langle M \in \mcM \mid M \notin F\rangle}$.  Notice that distinct minimal closed normal subgroups of $G$ commute; since centralisers are closed, it follows that $M_F \le \CC_G(M)$ for all $M \in F$.  Now $\mcC = \{M_F \mid F \in \mcF\}$ is a filtering family of non-trivial closed normal subgroups of $G$, so the intersection $S = \bigcap_{F \in \mcF}M_F$ is non-trivial.  We see that $S \le \CC_G(M)$ for all $M \in F$ and all $F \in \mcF$, so that $S \le \CC_G(M)$ for all $M \in \mcM$, and yet $S$ is in the closed subgroup generated by the elements of $\mcM$.  It follows that $S$ is a non-trivial abelian normal subgroup of $G$, giving the required contradiction.
\end{proof}

\subsection{Further preliminaries}

There is one more lemma from \cite{CRW-Part1} that we will use frequently.

\begin{lem}[{\cite[Lemma~7.4]{CRW-Part1}}]\label{profcomm}
Let $U$ be a first-countable profinite group and let $K$ be a closed subgroup of $U$.  Then $U$ commensurates $K$ if and only if $U$ normalises an open subgroup of $K$.
\end{lem}

\addtocontents{toc}{\protect\setcounter{tocdepth}{2}}

\section{Connections between topological and abstract simplicity}

Our focus in this paper is on \tdlc groups that have restrictions imposed on their normal subgroups; in particular, our methods apply to the class of compactly generated topologically simple \tdlc groups.  A recurring theme is the difference between topological simplicity and abstract simplicity; indeed, as pointed out in the introduction, it is not known whether there are any compactly generated \tdlc groups that are topologically but not abstractly simple.

Our standard assumption is that the group we are working with is a compactly generated, topologically simple \tdlc group.  However, the condition of topological simplicity can sometimes be relaxed, and we will do so where this is convenient.  In addition, in some cases results will pertain not just to the given \tdlc group, but to any sufficiently large subgroup (for instance dense subgroups).  Such results are of interest outside the theory of \tdlc groups, as there is a strong connection between dense subgroups of \tdlc groups and the theory of Hecke pairs, that is, groups with a specified commensurated subgroup.  (For a general discussion of this connection, see \cite{Reid&Wesolek}.)  An example of this connection is given by Proposition~\ref{prop:Hecke} below.

\subsection{Dense embeddings into topologically simple \tdlc groups}
  
\begin{defn} A \defbold{Hecke pair}\index{Hecke pair} of groups $(\Gamma,\Delta)$ is a group $\Gamma$, together with a subgroup $\Delta$ such that $g\Delta g\inv$ is commensurate to $\Delta$ for all $g \in \Gamma$.  
\end{defn}

Each \tdlc group $G$ and compact open subgroup $U$ form a Hecke pair.  Conversely, each Hecke pair $(\Gamma,\Delta)$ canonically determines two totally disconnected groups and compact open subgroups.

\begin{thm}[{\cite[Theorem~7.1]{Belyaev}}; {\cite[Theorem~1.6]{Reid&Wesolek}}]\label{thm:Hecke:complete}
Let $(\Gamma,\Delta)$ be a Hecke pair of groups. Then there are \tdlc groups and homomorphisms 
$$
\beta_\Delta: \Gamma \to \hat{\Gamma}_{\Delta}\mbox{ and }
\beta_{\Gamma\ssl\Delta} : 
\Gamma \to \Gamma\ssl\Delta
$$ with the following properties.
\begin{enumerate}[(i)]
\item $\Delta$ is the inverse image under $\beta_\Delta$ (resp.~$\beta_{\Gamma \ssl \Delta}$) of a compact open subgroup $U\leq \hat{\Gamma}_{\Delta}$ (resp. $U\leq \Gamma\ssl\Delta$). 
\item Given any \tdlc group~$H$ and homomorphism $\beta : \Gamma\to H$ with dense image such that $\Delta = \beta^{-1}(U)$ for some compact open subgroup $U$ of~$H$, there are unique continuous quotient maps $\psi_1 : \hat{\Gamma}_{\Delta} \to H$ and $\psi_2 : H \to \Gamma \ssl \Delta$ with compact kernels such that the following diagram commutes.
\[
\xymatrixcolsep{3pc}\xymatrix{
& \Gamma \ar_{\beta_\Delta}[ld] \ar^{\beta}[d]  \ar^{\beta_{\Gamma \ssl \Delta}}[rd]& \\
\hat{\Gamma}_\Delta  \ar_{\psi_1}[r] & H \ar_{\psi_2}[r]& \Gamma \ssl \Delta. }
\]
\end{enumerate}
 \end{thm}
The maps $\beta_{\Delta}$ and $\beta_{\Gamma\ssl\Delta}$ are called respectively the \textbf{Belyaev} and \textbf{Schlichting} completions of the Hecke pair~$(\Gamma,\Delta)$. The Belyaev completion has the additional property that any homomorphism $\phi \colon \Gamma\to H$ such that $\overline{\phi(\Delta)}$ is profinite lifts to a homomorphism $\hat{\phi} \colon \hat{\Gamma}_{\Delta}\to H$ such that $\phi = \hat{\phi}\circ \beta_{\Delta}$. Moreover, $\hat{\phi}$ is a quotient map if and only if ${\phi(\Gamma)}$ is dense and $\overline{\phi(\Delta)}$ is open. If~$H$ has no compact normal subgroups, then $\Ker\hat{\phi}$ contains the compact kernel of the quotient map to $\Gamma\ssl\Delta$ and $\phi$ factors through the Schlichting completion. If~$\hat{\phi}$ is a quotient map, then $\bar{\phi}$ is as well, that is,~$H$ is a quotient of~$\Gamma\ssl\Delta$.

The Belyaev and Schlichting completions of $(\Gamma,\Delta)$ need not be topologically simple even if~$\Gamma$ is. The Schlichting completion of the Hecke pair $(\PSL_n(\bQ),\PSL_n(\bZ))$, for instance, is isomorphic to the product $\prod_{p \in \Pi}\PSL_n(\bQ_p)$, where $\Pi$ is the set of prime numbers. Proposition~\ref{prop:Hecke} implies that a finite number of simple quotients may be recovered nevertheless when~$\Gamma$ is finitely generated. The argument requires, not simplicity, but only a condition on how~$\Gamma$ embeds into its completion and that condition will be introduced first. The condition is based in the following ideas, which also recur in later sections.

Let $G$ be a topological group. The \defbold{monolith}\index{monolith}\index{Mon(G)@$\Mon(G)$} $\Mon(G)$ is the intersection of all non-trivial closed normal subgroups of $G$, see~\cite{CM}, and~$G$ is said to be \defbold{monolithic}\index{monolithic} if $\Mon(G)>1$.  If $G$ is monolithic then its monolith is topologically characteristically simple; conversely, if $G$ is a topologically characteristically simple \tdlc group, then $G \rtimes \Aut(G)$ (with $G$ embedded as an open subgroup) is monolithic.  
{
The following consequence of Proposition~\ref{cmsoc} illustrates that monolithic groups appear naturally in the structure of general \tdlc groups. 
\begin{cor} \label{cor:MonolithicQuotient}
Let $G$ be a compactly generated \tdlc group which possesses no non-trivial discrete, compact or abelian normal subgroup, and let $\mcM$ be as in Proposition~\ref{cmsoc}.  Then for any $M \in \mcM$, the quotient $G/\CC_G(M)$ has monolith ${\overline{M\CC_G(M)}/\CC_G(M)} > 1$.
\end{cor}

\begin{proof}
Let $N$ be the pre-image in $G$ of a non-trivial closed normal subgroup of the quotient $G/\CC_G(M)$. Then $N$ contains properly $\CC_G(M)$. In particular $N$ does not commute with $M$, and so $N \cap M$ is non-trivial. By the minimality of $M$, this implies that $N$ contains $M$. Hence $N$ contains the closure $\overline{M\CC_G(M)}$. Thus every non-trivial closed normal subgroup of $G/\CC_G(M)$ contains the closure of the image of $M$. It remains to check that $M$ has non-trivial image in the quotient $G/\CC_G(M)$. This is indeed the case, since $M$ is non-abelian and thus is not contained in its own centraliser.
\end{proof}

}

By definition, the image of the monolith  $\Mon(G)$ in any quotient of $G$ is either injective or trivial. 
Our purposes require to consider   other subgroups $\Gamma \leq G$ enjoying that property.  We   say that $\Gamma$ is \defbold{relatively simple}\index{relatively simple}\index{simple!relatively} in $G$  if for every closed normal subgroup~$N$ of~$G$ the map $\Gamma\to G/N$ is either injective or trivial. We say that $\Gamma$ is \defbold{relatively just-infinite}\index{just-infinite!relatively} in $G$ if for every closed normal subgroup~$N$ of~$G$ the map $\Gamma\to G/N$ is either injective or has finite image.  

Clearly, if $\Gamma$ or $G$ is topologically simple (resp. just-infinite), then $\Gamma$ is relatively simple (resp. relatively just-infinite) in $G$.   
The next proposition, whose straightforward proof is omitted, implies that the converse need not hold. 

\begin{prop}\label{prop:mono_embed}
Let $\Gamma\leq G$ be relatively simple (resp. relatively just-infinite). Then every non-trivial subgroup of $\Gamma$ is relatively simple (resp. relatively just-infinite) in $G$. Moreover for every closed normal subgroup $N$ of $G$, the quotient group  $\Gamma N/N$ is  relatively simple (resp. relatively just-infinite) in $G/N$.
\end{prop}

Since the group $\PSL_n(\bQ)$ is simple, it is a relatively simple subgroup of its   Schlichting completion of $(\PSL_n(\bQ),\PSL_n(\bZ))$. Therefore, by Proposition~\ref{prop:mono_embed}, so are the embeddings of the subgroups~$\PSL_n(\bZ[1/m])$ for each $m>0$ even though they are not simple. 
We are interested in the general case when the group $\Gamma$ embeds as a relatively simple or relatively just-infinite subgroup of the Belyaev or Schlichting completions of $(\Gamma,\Delta)$; in this case the embedding has the additional property of being dense.

\begin{lem}\label{psdenselem}
Let $G$ be a non-discrete \tdlc group with a dense subgroup $\Gamma$.  

If $\Gamma$ is relatively simple (resp. relatively just-infinite), then every discrete  quotient of $G$ is trivial (resp. finite).
\end{lem}

\begin{proof}
Let $O$ be an open normal subgroup of $G$.  Then $O \cap \Gamma > 1$ since $\Gamma$ is dense and $G$ is non-discrete. Hence the image of $\Gamma$ in $G/O$ is not injective, and is thus trivial (resp. finite). The result follows since $\Gamma$ is dense, and thus has dense image in $G/O$.
\end{proof}

\begin{prop}\label{prop:Hecke}
Let $\Gamma$ be a group with an infinite commensurated subgroup $\Delta < \Gamma$ of infinite index. Suppose that $\Gamma$ is generated by finitely many cosets of that subgroup (e.g. $\Gamma$ is finitely generated). 

\begin{enumerate}[(i)]

\item If $\Gamma$ is just-infinite, then $\Gamma$ embeds as a dense subgroup of some compactly generated monolithic \tdlc group $G$ whose monolith is cocompact and a \textbf{quasi-product}\index{quasi-product} of finitely many pairwise isomorphic groups in $\sclass$ (we refer to \cite{CM} for the definition of a quasi-product, a notion which will not be used anywhere else in the present paper). 

\item If $\Gamma$ is hereditarily just-infinite, then a finite index subgroup of $\Gamma$ embeds as a dense subgroup of some compactly generated monolithic \tdlc group $G$ whose monolith is cocompact and belongs to $\sclass$. 

\item If $\Gamma$ is simple, then $\Gamma$ embeds as a dense subgroup of some group $G \in \sclass$. 

\end{enumerate}
\end{prop}

\begin{proof}
We assume that $\Gamma$ is just-infinite. Then Theorem~\ref{thm:Hecke:complete} yields a non-discrete \tdlc group $H= \hat{\Gamma}_{\Delta}$ with a dense embedding of $\Gamma$ as a relatively just-infinite subgroup. 
Moreover since $\Gamma$  is generated by finitely many cosets of $\Delta$, we see  that $H$ is compactly generated, and  since $\Delta$ is of infinite index in $\Gamma$, the group  $H$ is not compact. By \cite[Proposition~5.2]{CM}, the group $H$ has a closed normal subgroup $Q$ such that the quotient $H/Q$ is not compact, and moreover every closed normal subgroup of $H$ properly containing $Q$ is cocompact (in the terminology of \cite{CM}, the group $H/Q$ is \emph{just-non-compact}\index{just-non-compact}).  By Lemma~\ref{psdenselem}, the group $H$ has no infinite discrete quotient. In particular $H/Q$ is non-discrete. Since $H/Q$ is totally disconnected and non-compact, it does does not have a connected cocompact normal subgroup either. Therefore \cite[Theorem~E]{CM} ensures that $H/Q$ is a monolithic group whose monolith is the quasi-product of finitely many pairwise isomorphic groups in $\sclass$. This proves~(i). 

Let now $H_1$ be the finite index open normal subgroup of $H$ containing $Q$ which is the kernel of the $H$-action on the quasi-factors $N_1/Q, \dots, N_n/Q$ of the monolith of $H/Q$. Since the quasi-factors commute pairwise, we have $N_i \leq H_1$ for all $i$. Moreover, it follows from Corollary~\ref{cor:MonolithicQuotient} that $G= H_1/C_{H_1}(N_1)$ is monolithic, whose monolith is cocompact and belongs to $\sclass$. Assuming now in addition that $\Gamma$ is hereditarily just-infinite, we see that the intersection $\Gamma_1 = H_1 \cap \Gamma$ has finite index in $\Gamma$, and is thus just-infinite and dense in $H_1$. Therefore $\Gamma_1$ is relatively just-infinite in $H_1$, and thus maps injectively onto a dense subgroup of $G$. This proves~(ii). 

Finally, in case $\Gamma$ is simple, every finite quotient of $\Gamma$ is trivial, which implies that every profinite quotient of $H$ is trivial as well. This forces the just-non-compact group $H/Q$ to be topologically simple and, hence, to belong to $\sclass$. 
\end{proof}

Proposition~\ref{prop:Hecke}  is well illustrated by the Hecke pair $(\PSL_n(\bZ[1/m]),\PSL_n(\bZ))$. It may be verified, for any~$m>1$, that $\PSL_n(\bZ[1/m])$ is generated by finitely many cosets of~$\PSL_n(\bZ)$ and, moreover, $\PSL_n(\bZ[1/m])$ is   embedded, for any set of primes~$\pi$, as a relatively simple subgroup in $\prod_{p \in \pi}\PSL_n(\bQ_p)$ by restricting the embedding of the simple group~$\PSL_n(\bQ)$. This embedding of~$\PSL_n(\bZ[1/m])$ in~$\prod_{p \in \pi}\PSL_n(\bQ_p)$ is dense if and only if $\pi$ is a set of prime divisors of~$m$. The topologically simple groups afforded by Proposition~\ref{prop:Hecke} are then the groups~$\PSL_n(\bQ_p)$, $p\in\pi$.

Proposition~\ref{prop:Hecke}(ii) highlights the relevance of the study of the class $\sclass$ to the study of commensurated subgroups of hereditarily just-infinite groups. In particular the investigation of  $\sclass$ is relevant to the \emph{Margulis--Zimmer conjecture}\index{Margulis--Zimmer conjecture} according to which a lattice in a connected centreless simple Lie group  of rank~$\geq 2$ has no commensurated subgroup other than the finite or the finite index subgroups (see \cite{ShalomWillis} for a thorough discussion of that conjecture as well as partial results on non-uniform lattices). 

\begin{rem}\label{rem:amenable}
Proposition~\ref{prop:Hecke} is also a potential source of new topologically simple \tdlc groups. After all, we have seen that the well-known groups $\PSL_n(\bQ_p)$ arise in this way. Should there exist a finitely generated, simple, amenable group with an infinite and proper commensurated subgroup, then this construction will produce a non-discrete, topologically simple, amenable \tdlc group. Similarly, should there exist a finitely generated simple group $\Gamma$ with an infinite and proper commensurated subgroup $\Delta$ such that every element of $\Gamma$ has finite order, then there would be a non-discrete, topologically simple, anisotropic \tdlc group.
\end{rem}

\begin{rem}
A variation on Proposition~\ref{prop:Hecke} was recently used by Ph.~Wesolek in combination with the normal subgroup structure theory of \tdlc groups to show that in a finitely generated just-infinite branch group, every commensurated subgroup is either finite or of finite index (see \cite{Wesolek_branch}). 
\end{rem}

\subsection{Characterising abstract simplicity in second-countable \tdlc groups}

From now on, it will be convenient to restrict attention to second-countable \tdlc groups.  As noted in the preliminaries, this is no real restriction for most applications in the present paper.

The main result of this section characterises abstract simplicity of \tdlcsc groups.  We will return to the named properties \sdn and \sur later in the article.

\begin{thm}\label{baireabs}
Let $G$ be a non-discrete \tdlcsc group.  Then $G$ is abstractly simple if and only if it has the following three properties:

\sop $G$ has no proper open normal subgroups.

\sdn Every non-trivial normal subgroup of $G$ contains an infinite commensurated compact locally normal subgroup of $G$.

\sur Every infinite commensurated compact subgroup of $G$ is open.

If $G$ is abstractly simple then the following stronger version of \sur holds:

Suppose $U$ is an open subgroup of $G$, and $K$ is a normal subgroup of $U$ (not necessarily closed) that is commensurated by $G$.  Then $K$ has countably many cosets in $G$ and $\overline{K}$ is an open subgroup of $G$.
\end{thm}

Given a non-discrete \tdlc group $G$, it is obvious that \sop, \sdn and \sur together imply abstract simplicity of $G$, and that \sop and \sdn are necessary for abstract simplicity. The non-trivial part of the proof is to show that \sur is also necessary when $G$ is second-countable. In particular, it will follow that given $G \in \sclass$, then $G$ is abstractly simple if and only if it satisfies both \sdn and \sur.

Let us first recall that by Lemma~\ref{profcomm}, given any first-countable profinite group $U$ and closed subgroup $K$ of $U$, then $U$ commensurates $K$ if and only if $U$ normalises an open subgroup of $K$.  In other words:

\begin{lem}\label{lem:comm:locnorm}
Given a first-countable \tdlc group $G$ and a commensurated compact subgroup $K$ of $G$, then there is a finite index closed subgroup of $K$ that is locally normal in $G$.  In particular, $G$ satisfies \sur if and only if $G$ has no non-trivial fixed points in its action on $\lnorm(G)$.
\end{lem}

A simple observation shows that a \tdlc group with property \sur cannot have many closed normal subgroups:

\begin{lem}\label{lem:simpleSU}
Let $G$ be a \tdlc group with property \sur. Then any closed normal subgroup of $G$ is discrete or open. 
\end{lem}
\begin{proof}
Let $N$ be a closed normal subgroup of $G$ and let $U$ be a compact open subgroup of $G$. The conjugation action of $G$ on $N$ is by automorphisms of the locally compact group and therefore commensurates the compact open subgroup $N \cap U$. Then property \sur implies that $N \cap U$ is either finite or open, thereby yielding the desired conclusion.
\end{proof}

In the present context, it is useful to have a notion of `size' of certain subsets of a \tdlc group.  Recall that infinite compact groups are homogeneous Baire spaces, so in particular they are perfect sets and uncountable.  Moreover, a closed subgroup of a compact group that is not open has empty interior, and therefore has uncountably many cosets.

\begin{defn}\label{def:LocalSize}
Let $G$ be a \tdlc group and let $K$ be a subgroup of $G$ such that $\N_G(K)$ is open.  Say a subset $X$ of $G$ is \defbold{$K$-meagre}\index{meagre} if $X$ is contained in the union of countably many left cosets of $K$ and \defbold{$K$-large}\index{large} if it contains a coset of a subgroup of $K$ of finite index.  Say $K$ is the \defbold{local size}\index{local size} of $X$ (or $X$ is \defbold{$K$-sized}\index{sized@$K$-sized}) if $X$ is $K$-meagre and $K$-large.

The notions of $K$-meagre, $K$-large and $K$-sized are determined by the commensurability class of $K$, so given $\alpha \in \lnorm(G)$, we can define $\alpha$-meagre, $\alpha$-large and $\alpha$-sized to mean $K$-meagre, $K$-large and $K$-sized respectively, where $K$ is some (any) compact representative of~$\alpha$.  Say $X$ is \defbold{locally sized}\index{locally sized} if $X$ has local size $\alpha$ for some $\alpha \in \lnorm(G)$.\end{defn}

The following well-known and easily proved fact will be used without further comment:

\begin{lem}Let $G$ be a group and let $H$ be a commensurated subgroup of $G$.  Then every right coset of $H$ is contained in the union of finitely many left cosets, and vice versa.\end{lem}


\begin{lem}\label{countsize}Let $G$ be a \tdlcsc group.
\begin{enumerate}[(i)]
\item Let $X$ be a locally sized subset of $G$.  Then $X$ is $\alpha$-sized for exactly one $\alpha \in \lnorm(G)$, which is both the least element of $\lnorm(G)$ for which $X$ is $\alpha$-meagre, and the greatest element of $\lnorm(G)$ for which $X$ is $\alpha$-large.
\item Let $K$ be a subgroup of $G$ with open normaliser, and suppose that $G = \Comm_G(K)$.  Then the product of any two $K$-meagre subsets of $G$ is $K$-meagre.  Given a {(not necessarily closed)} $K$-meagre subgroup $H$ of $G$ such that $\N_G(H)$ is open, then $H$ is contained in a $K$-meagre normal subgroup of $G$.

\item Let $\alpha \in \lnorm(G)$. 
Then the following are equivalent:
\begin{enumerate}[(a)]
\item $G = \Stab_G(\alpha)$;
\item $G$ has an $\alpha$-sized normal subgroup;
\item Every $\alpha$-meagre subgroup of $G$ whose normaliser is open is contained in an $\alpha$-sized normal subgroup of $G$.\end{enumerate}
\end{enumerate}\end{lem}

\begin{proof}(i) Let $\alpha, \beta \in \lnorm(G)$ and suppose that~$X$ is $\alpha$-large and $\beta$-meagre.  Then, since these properties are invariant under left translation, it may be supposed that~$X$ contains a representative,~$A$, of~$\alpha$ and~$A$ is $\beta$-meagre because~$X$ is. Hence~$A$ is the union of countably many left cosets of $A \cap B$, where~$B$ is any representative of~$\beta$. Since $A$ is a Baire space, it follows that some (and hence all) left cosets of $A \cap B$ are open in $A$, and then, since~$A$ is compact, that $|A:A \cap B|$ is finite. Hence $\alpha \le \beta$ and the local size $\alpha$ of $X$ (if it exists) must be both the least element of $\lnorm(G)$ for which $X$ is $\alpha$-meagre and the greatest element of $\lnorm(G)$ for which $X$ is $\alpha$-large.

(ii) Let $X$ and $Y$ be $K$-meagre sets.  Then $XY$ is contained in the union of countably many sets of the form $xKyK$ for $x,y\in G$ and in turn, since~$K$ is commensurated by~$G$, each such set is the union of finitely many left cosets of $K$.  Hence $XY$ is $K$-meagre.

Repeating the argument, we see that the product of any finite set of $K$-meagre sets is $K$-meagre.  Now let $H$ be a $K$-meagre subgroup of $G$ such that $\N_G(H)$ is open.  By second-countability, there are only countably many cosets of $\N_G(H)$ in $G$, hence only countably many $G$-conjugates of $H$.  Hence we may make a countable list $H_1,H_2,H_3,\dots$ consisting of $G$-conjugates of $H$ such that every $G$-conjugate appears infinitely often.  Conjugates of~$K$-meagre sets are~$K$-meagre, since $K$ is commensurated by $G$.  In particular, we see that $H_i$ is $K$-meagre for all $i$.  Now notice that the set
\[ L = \bigcup_{i \in \bN}(H_1H_2\dots H_i)\]
is a normal subgroup of $G$ containing $H$ and by construction $L$ is $K$-meagre.

(iii) It is clear that (c) implies (b) because~$\alpha$ has a representative with open normaliser and every representative of~$\alpha$ is $\alpha$-meagre; (b) implies (a) by part (i) because the $\alpha$-sized  normal subgroup is also $g.\alpha$-sized for every $g\in G$; and (a) implies (c) by applying part (ii) with~$H$ any representative of~$\alpha$.\end{proof}

\begin{proof}[Proof of Theorem \ref{baireabs}]Suppose $G$ has the given properties.  Let $H$ be a non-trivial normal subgroup of $G$.  Then $H$ contains a non-trivial commensurated compact locally normal subgroup $L$ of $G$ by property \sdn; by property \sur, $L$ and hence $H$ is open, so $H = G$ since $G$ has no proper open normal subgroups by property \sop.

Conversely, suppose $G$ is abstractly simple.  Certainly, $G$ contains a compact open subgroup $U$; in particular $U$ is infinite, commensurated and locally normal.  Since $G$ is the only non-trivial normal subgroup of $G$, it follows that $G$ satisfies \sop and \sdn.  To prove \sur, let $K$ be a commensurated compact subgroup of $G$.  By applying Lemma~\ref{lem:comm:locnorm} and replacing $K$ with a subgroup of finite index we may assume $K$ is locally normal in $G$.  Then by Lemma~\ref{countsize} (ii), there is a normal subgroup $L$ of $G$ that contains $K$, such that $L$ is $K$-meagre.  Since $G$ is abstractly simple, it follows that $G = L$, so $G$ is $K$-meagre, in other words, $K$ has countable index in $G$.  The Baire Category Theorem then ensures that ${K}$ is open in $G$, proving that $G$ satisfies \sur as required.
\end{proof}

Say a locally compact topological group $G$ is \defbold{unrefinable}\index{unrefinable} if $G$ is non-discrete, and the only locally compact group topology on $G$ that properly refines the given topology is the discrete one.  Recall the following from \cite{CRW-Part1}:

\begin{lem}[{\cite[Theorem~IV]{CRW-Part1} and Lemma~\ref{profcomm}}]\label{reflem}Let $G$ be a first-countable \tdlc group.  Then for every commensurated compact subgroup $K$ of $G$, there is a closed subgroup of finite index in $K$ that is locally normal in $G$.  Consequently, there is a natural bijection between the elements of $\lnorm(G)$ that are fixed by the action of $G$, and refinements of the topology of $G$ that are  locally compact and compatible with the group structure.\end{lem}

It is clear that for any \tdlc group $G$, if $G$ is unrefinable then $G$ has \sur.  By Lemma~\ref{reflem}, the converse holds for \tdlcsc groups.  Theorem~\ref{baireabs} implies that this applies in particular to abstractly simple \tdlcsc groups.

\begin{cor}\label{unrefcor}Let $G$ be a non-discrete \tdlcsc group.  Then $G$ is unrefinable if and only if $G$ has property \sur.  In particular, every abstractly simple \tdlcsc group is unrefinable.\end{cor}

\section{Some local properties of compactly generated \tdlc groups}

\subsection{Cayley--Abels graphs}\label{sec:CayleyAbels}

We recall the now classical construction, due to Abels, of a family of connected locally finite graphs associated with any compactly generated \tdlc group.  (See \cite[Beispiel 5.2]{Abels}, or \cite[\S11]{Mon} for more details.)

\begin{prop}\label{prop:Abels}
Let $G$ be a compactly generated \tdlc group, let $U$ be a compact open subgroup of $G$, let $A$ be a compact symmetric subset of $G$ such that $G = \langle U, A \rangle$ and let $D$ be a dense symmetric subset of $G$.
\begin{enumerate}[(i)]
\item There exists a finite symmetric subset $B$ of $G$ such that $B \subseteq D$ and 
\[
BU = UB = UBU = UAU.
\]
\item For any subset $B$ satisfying part (i), then $G = \langle B \rangle U$ and the coset space $G/U$ carries the structure of a locally finite connected graph, invariant under the natural $G$-action, where $gU$ is adjacent to $hU$ if and only if $gU \neq hU$ and $Ug\inv hU \subseteq UBU$.
\end{enumerate}
\end{prop}

\begin{proof}
(i) Since $AU$ is compact, it is contained in a union of finitely many right cosets of $U$, so $UAU$ is a union of finitely many right cosets of $U$; say $UAU = \bigcup_{b \in B_1}Ub$ for $B_1$ finite.  Similarly, $UAU = \bigcup_{b \in B_2}bU$ for $B_2$ finite.  Since all the right and left cosets of $U$ are open, we are free to choose $B_1$ and $B_2$ to be subsets of $D$.  Now set
\[
B = B_1 \cup B\inv_1 \cup B_2 \cup B\inv_2;
\]
it is straightforward to verify that $B$ has the required properties.

(ii) It is clear that the given adjacency relation is well-defined, and the fact that $B = B\inv$ ensures that it is a symmetric relation.  Moreover, given $g,h,k \in G$, we see that $U(kg)\inv (kh)U = Ug\inv h U$, so $kgU$ is adjacent to $khU$ if and only if $gU$ is adjacent to $hU$, in other words, $G$ preserves the graph structure.  Since $UB = UAU$ and $G = \langle U,A \rangle$, we have $G = \langle U, B \rangle$; the fact that $G =  \langle B \rangle U$ follows by repeatedly applying the equation $UB = UBU$.  Given $h \in \langle B \rangle$, there is clearly a path from $U$ to $hU$, ensuring that the graph defined on $G/U$ is connected.

Since $G$ acts transitively on $G/U$ and preserves adjacency, to see that the graph is locally finite it is enough to see that the trivial coset $U$ has finitely many neighbours.  Indeed, $U$ is adjacent to $gU$ for $g \in G \smallsetminus U$ if and only if $gU \subseteq UBU$.  But $UBU = BU$ is the union of finitely many left cosets of $U$, so there are only finitely many possibilities for $gU$.
\end{proof}

The edge set of the graph provided by Proposition~\ref{prop:Abels} depends on the set $UBU$ and will be called \defbold{a Cayley--Abels graph}\index{Cayley--Abels graph} associated with $(G, U)$. All Cayley--Abels graphs of~$G$ are quasi-isometric.

The next lemma in turn uses this information to strengthen considerably the observation that, if~$G$ is non-discrete, compactly generated and topologically simple and~$U$ is a compact open subgroup of~$G$, then finitely many conjugates of $U$ suffice to generate $G$.  
\begin{lem}\label{lem:CptGenNC}
Let $G$ be a compactly generated \tdlc group and let $L$ be a (possibly non-closed) locally normal subgroup. Suppose the abstract normal closure $D = \lla L \rra$ is dense in $G$. Then there are~$g_1$, \dots, $g_n$ in~$D$ such that $D = \langle L\cup  g_1Lg_1^{-1}\cup \dots \cup g_nLg_n^{-1} \rangle$.
\end{lem}

\begin{proof}
Let $U$ be a compact open subgroup that normalises~$L$. Since $D$ is dense in $G$, by Proposition~\ref{prop:Abels} we can find a finite set $\Sigma\subset D$ such that $G = \langle \Sigma \rangle U$. Since~$U$ normalises~$L$, the conjugation action of $\langle \Sigma \rangle$ is thus transitive on the $G$-conjugacy class of $L$. Hence the subgroup $ \langle \Sigma \cup L \rangle$ of~$D$ contains the entire conjugacy class of $L$, which, since $D$ is generated by that conjugacy class, implies that $D = \langle \Sigma \cup L \rangle$. Since it belongs to~$D$, each element of~$\Sigma$ is a product of a finite number of conjugates of elements of~$L$. We may therefore choose~$g_1, \dots, g_n \in \langle \Sigma \rangle$ such that $\Sigma \subset \langle L\cup g_1Lg_1^{-1}\cup \dots \cup g_nLg_n^{-1}\rangle $, and the claim follows.
\end{proof}

\subsection{Quasi-centralisers of commensurated subgroups}

We now prove two related propositions.  The second may be be viewed as a generalisation of \cite[Theorems~4.8 and 4.9]{BEW}.

\begin{prop}\label{prop:TrivialQC}
Let $G$ be a \tdlcsc group, let $N$ be a minimal closed normal subgroup of $G$ and let $H$ be an infinite closed subgroup of $N$ such that $H$ is commensurated by $G$ and $\QZ(H)=\triv$.  Then the quasi-centraliser $\QC_G(H)$ is a closed normal subgroup that intersects $N$ trivially. In particular, if $G$ is monolithic and $\QZ(N)=\triv$, then $\QC_G(H)=\QC_G(N)=\triv$. 
\end{prop}

\begin{proof}
The hypothesis that $\QZ(H)=\triv$ ensures that $H$ is non-discrete.  By~Lemma~\ref{profcomm}, there exists a compact open subgroup $U$ of $G$ such that $H \cap U$ is locally normal in $G$.  Since   $\QZ(H \cap U) = \QZ(H)=\triv$, it   follows from \cite[Lemma~3.8(i)]{CRW-Part1} that $\QC_G(H \cap U) \cap \N_G(H \cap U) = \CC_G(H \cap U)$.  In particular, since $\N_G(H \cap U)$ is open in $G$ and $\CC_G(H \cap U)$ is closed, we deduce that $\QC_G(H \cap U) = \QC_G(H)$ is closed in $G$. Making the abbreviation $Q = \QC_G(H) \cap N$, we deduce that   $Q$ is closed.  Moreover $\QC_G(H)$ is normal in $G$, since $G$ commensurates $H$. Therefore, if $Q$ is non-trivial, we must have  $Q = N$ by minimality of $N$.  Recalling that $H \leq N$ by hypothesis, we obtain $H \leq Q \leq \QC_G(H)$, so $H = \QC_H(H) = \QZ(H)$,   which is absurd as $H$ is infinite and $\QZ(H)=\triv$.  From this contradiction we deduce that $\QC_G(H) \cap N=\triv$.  The final conclusion is clear.
\end{proof}

\begin{prop}\label{prop:TrivialQC:2}
Let $G \in \sclass$. Then every infinite compact subgroup that is commensurated by $G$ has trivial quasi-centraliser.  In particular, $\QZ(G)=\triv$.
\end{prop}
\begin{proof}
Let $H$ be an infinite compact locally normal subgroup that is commensurated by $G$.  Invoking Lemma~\ref{profcomm} as in Proposition~\ref{prop:TrivialQC}, we may assume that $H$ is locally normal.

Let $Q = \QC_G(H)$.  We see that $Q$ is normal in $G$, since $G$ fixes $[H]$.  Suppose that $Q \ne 1$.  Then $Q$ is dense in $G$, since $G$ is topologically simple.

Let $U$ be a compact open subgroup of $G$ that normalises~$H$.  Then $G$ is generated by $UX$ where $X$ is a finite subset of $G$ which, by Proposition~\ref{prop:Abels}, may be chosen to be a subset of $Q$.  Since $X$ is finite, there is an open subgroup $K$ of finite index in $H$ that is centralised by $X$. Choose~$V$ open and normal in~$U$ such that $L = V\cap H \leq K$. Then $L$ is normalised by all elements of $U \cup X$, and is a non-trivial compact normal subgroup of $G$. Since $G$ is topologically simple, we conclude that $G$ is compact, hence profinite.  But the only topologically simple profinite groups are the finite simple groups, which are discrete.  As $G$ is assumed non-discrete, we have a contradiction.  Thus $Q$ must be trivial as required.  The final conclusion follows by considering $\QC_G(U)$ where $U$ is a compact open subgroup of $G$.
\end{proof}

\subsection{The local prime content}\label{section:localprime}

Cayley--Abels graphs can be applied to obtain a restriction on the primes involved in the compact open subgroups of a compactly generated \tdlc group.

\begin{defn}Let $G$ be a \tdlc group.  The \defbold{local prime content}\index{local prime content}\index{eta(G)@$\lpc(G)$} $\lpc(G)$ of $G$ is the set of primes $p$ for which $G$ contains an infinite pro-$p$ subgroup.\end{defn}

\begin{prop}\label{localprime:short}  Let $G$ be a non-discrete compactly generated \tdlc group.  Then $G$ has a compact normal subgroup $K$, which can be chosen to be a subgroup of any given open neighbourhood of the identity, such that for every compact open subgroup $U$ of $G$, the composition factors of $UK/K$ are of bounded order.  In particular, the set $\eta = \lpc(G/K)$ is finite and $UK/K$ is virtually pro-$\eta$ for every compact open subgroup $U$.
\end{prop}

\begin{proof}
By Lemma~\ref{lem:KK} we may assume that $G$ is second-countable.  Let $O$ be an open neighbourhood of the identity in $G$.  Then by van Dantzig's theorem, $O$ contains a compact open subgroup $V$ of $G$.  Let $\Gamma$ be a Cayley--Abels graph associated to $(G,V)$.  Let $K$ be the kernel of the action of $G$ on $\Gamma$.  Then $K = \bigcap_{g \in G} gVg^{-1}$, so $K$ is compact and $K \subseteq O$.  From now on we may assume $K=\triv$. Vertex transitivity of the action of~$G$ on~$\Gamma$ implies that all vertices have the same valency,~$d$.

Let $U$ be a compact open subgroup of $G$ and consider the action of $U$ on $\Gamma$. Let $U_0 = U \cap V$ and  for all $n >0$, let $U_n$ be the pointwise stabiliser of the $n$-sphere   around the base vertex $v_0 = V$ in $U_0$.  Thus the collection $\{U_n\}_{n \ge 0}$ is a descending chain of open normal subgroups of $U_0$ with trivial intersection.  For each $n \ge 0$, the finite group $U_n/U_{n+1}$ fixes pointwise the $n$-sphere around $v_0$ and acts faithfully on the $n+1$-sphere.  Since $\Gamma$ is $d$-regular, the orbits of $U_n/U_{n+1}$ on the $n+1$-sphere have size at most $d$, so every composition factor of $U_1$ embeds in $\Sym(d)^m$ for some $m$ and thus has order dividing $d!$.  In particular, $U_1$ is a pro-$\pi$ group where $\pi$ is the set of primes less than $d+1$. Hence~$\eta\subseteq \pi$ and is finite.  For each of the finitely many primes $p \in \pi \smallsetminus \eta$, the $p$-Sylow subgroup $P_p$ of $U_1$ is finite, so there exists an open normal subgroup $V_p$ of $U_1$ that intersects $P_p$ trivially.  Sylow's theorem then implies $V_p$ is pro-$p'$, and by intersecting finitely many such open subgroups of $U_1$, we obtain an open pro-$\eta$ subgroup $V$, which is then also of finite index in $U$. Hence~$U$ is virtually pro-$\eta$.
\end{proof}

In order to obtain stronger results of this kind, we recall some vocabulary from finite group theory, adapted to the profinite setting.

\begin{defn}
Let $G$ be a profinite group.  For a set of primes $\pi$, the \defbold{$\pi$-core}\index{pi-core@$\pi$-core}\index{Opi(G)@$O_\pi(G)$} $O_\pi(G)$ is the largest normal pro-$\pi$ subgroup of $G$, and the  \defbold{$\pi$-residual}\index{pi-residual@$\pi$-residual}\index{Opj(G)@$O^\pi(G)$} $O^\pi(G)$ is the smallest closed normal subgroup of $G$ such that $G/O^\pi(G)$ is pro-$\pi$. The \defbold{prosoluble core}\index{prosoluble core}\index{Oinfty(G)@$O_\infty(G)$} $O_\infty(G)$ is the largest normal prosoluble subgroup of $G$, while the  \defbold{prosoluble residual}\index{prosoluble residual}\index{Ojnfty(G)@$O^\infty(G)$} $O^\infty(G)$ is the smallest closed normal subgroup of $G$ such that $G/O^\infty(G)$ is prosoluble.
\end{defn}

\begin{lem}\label{corelem}Let $G$ be a \tdlc group, let $U$ and $V$ be compact open subgroups of $G$ and let $\pi$ be a set of primes.  Then $O_\pi(U)$ is commensurate with $O_\pi(V)$; similarly, $O_{\infty}(U)$ is commensurate with $O_{\infty}(V)$.  Consequently
\[
G = \Comm_G(O_\pi(U)) = \Comm_G(O_{\infty}(U)).
\]
\end{lem}

\begin{proof}Let $W$ be an open subgroup of $U \cap V$ that is normal in $V$.  Then $O_\pi(U) \cap W$ is pro-$\pi$ and normal in $W$, so contained in $O_\pi(W)$.  In turn $O_\pi(W)$ is pro-$\pi$ and normal in $V$, so $O_\pi(W) \le O_\pi(V)$.  Hence $O_\pi(V)$ contains an open subgroup of $O_\pi(U)$ and {\it vice versa\/} by symmetry. Hence $O_\pi(U)$ is commensurate with $O_\pi(V)$.

The proof of the analogous statement about the prosoluble residual is the same.
\end{proof}

Before proving the main result of this section, Theorem~\ref{localprime}, we recall an observation of Wielandt (originally made in the context of finite groups) about the normalisers of $\pi$-residuals and prosoluble residuals (compare Lemma~2.2.3 from \cite{BurgerMozes}).

\begin{lem}\label{lem:subnormal}
Let $G$ be a profinite group and let $S$ be a subnormal subgroup of $G$.  Then for each set of primes $\pi$, the group $O_\pi(G)$ normalises $O^\pi(S)$. Similarly $O_\infty(G)$ normalises $O^\infty(S)$. 
\end{lem}

\begin{proof}
Let $*$ stand for either a set of primes $\pi$, or $\infty$, and let $\mcC$ be the class of pro-$\pi$ groups or the class of pro-(finite soluble) groups respectively.  Let $H = O_*(G)S$.  Since 
$$
S/(S \cap O^*(H)) \cong SO^*(H)/O^*(H)
$$ 
and $\mcC$ is closed under taking closed subgroups, we see that $S/(S \cap O^*(H))$ is in $\mcC$. Hence $O^*(S) \le O^*(H)$.  We claim that in fact $O^*(S) = O^*(H)$. 

Since $S$ is subnormal in $G$, it is also subnormal in $H$, and there is a series 
$$S = H_0 \lhd H_1 \lhd \dots \lhd H_n = H.$$

{In order to prove the claim, observe first that, since $H = O_*(G)S$ and since $S \leq H_i$ for all $i$, we have $H_i = H_i \cap O_*(G) S = O_*(H_i) S$ for all $i$. We now proceed to prove the claim  by induction on $n$.}  By the inductive hypothesis, we may assume $O^*(S) = O^*(H_{n-1})$.  Now $O^*(H_{n-1})$ is characteristic in $H_{n-1}$, hence normal in $H_n$.  Thus we have a quotient $H_n/O^*(H_{n-1})$ of $H_n$, which is an extension of $A = H_n/H_{n-1}$ by $B = H_{n-1}/O^*(H_{n-1})$.  {Since $H_n = O_*(H_n) S$ and $S \leq H_{n-1}$,} we see that $A$ is isomorphic to a quotient of $O_*(H_n)$, which is in $\mcC$, so $A \in \mcC$, whilst $B \in \mcC$ by the definition of $O^*(H_{n-1})$.  Since $\mcC$ is closed under extensions within the class of profinite groups, we have $H_n/O^*(H_{n-1}) \in \mcC$, so $O^*(H_n) \le O^*(H_{n-1})$ and hence $O^*(H_n) = O^*(H_{n-1}) = O^*(S)$. 

Since $H_n = H$ and since $O^*(H)$ is normal in~$H$ and~$O_*(G)$ is a subgroup of $H$, it follows that~$O_*(G)$ normalises~$O^*(S)$.
\end{proof}

It is a fact well-known and frequently used that, if~$H$ and~$K$ are subgroups of~$G$ that normalise each other and have trivial intersection, then~$H$ centralises~$K$. The following lemma, recalled from \cite{CRW-Part1}, is the local version of this fact.

\begin{lem}[{\cite[Lemma~3.9]{CRW-Part1}}]\label{finint}Let $G$ be a \tdlc group and let $H$ and $K$ be closed subgroups of $G$ such that $\QZ(H)=\triv$ and such that ${|H \cap K|}$, ${|H:\N_H(K)|}$ and ${|K:\N_K(H)|}$ are all finite.  Then $H$ centralises a finite index open subgroup of $K$.
\end{lem}

\begin{cor}\label{finintcor}
Let $G$ be a \tdlc group and let $H$ and $K$ be non-trivial compact locally normal subgroups of $G$ such that $\QZ(H) = \QC_G(K) = \triv$.  Then $H \cap K$ is infinite.\end{cor}

\begin{proof}
Since $H$ and $K$ are locally normal, we see that ${|H:\N_H(K)|}$ and ${|K:\N_K(H)|}$ are finite.  Suppose also that ${H \cap K}$ is finite.  By Lemma~\ref{finint}, $H$ centralises an open subgroup of $K$, that is, $H \le \QC_G(K)$.  Since $\QC_G(K)$ is trivial but $H$ is not, we have a contradiction.
\end{proof}

Theorems~\ref{localprime} and~\ref{localprimetopsimp} are the culmination of this section. Before proceeding to them, we record the following subsidiary fact separately. 

{

\begin{prop}\label{prop:HelpLocalPrime}
Let $G$ be a non-trivial compactly generated \tdlc group.  Suppose that $G$ has trivial quasi-centre, no non-trivial compact normal subgroups and no non-trivial compact abelian locally normal subgroups.  Let $H$ be a non-trivial closed locally normal subgroup of $G$. Then there is a minimal closed normal subgroup $M$ of $G$ such that $H \cap M$ is non-discrete.
\end{prop}

\begin{proof}
By Proposition~\ref{cmsoc}, every non-trivial closed normal subgroup of $G$ contains a minimal one, and the set $\{M_1,\dots, M_n\}$ of minimal closed normal subgroups of $G$ is finite.  Proposition~\ref{cstab} ensures that $\QC_G(M_i) = \CC_G(M_i)$ for all $i$.  The intersection $\bigcap^n_{i=1}\QC_G(M_i)$ is thus a closed normal subgroup of~$G$. This intersection centralises, and therefore does not contain, each of the non-abelian subgroups $M_i$.  Hence~$\bigcap^n_{i=1}\QC_G(M_i)$ is trivial.

By Proposition~\ref{cstab} we have $\QZ(H)=\triv$, so $H$ is non-discrete.  Since~$\bigcap^n_{i=1}\QC_G(M_i) = \triv$, there is $i \in \{1,\dots,n\}$ such that $\QC_G(M_i)$ does not contain an open subgroup of $H$.  In particular, no open subgroup of $H$ centralises any open subgroup of $M_i$.  Given a compact open subgroup $U$ of $G$, then $H' = H \cap U$ and $M'_i = M_i \cap U$ are both compact and locally normal in $G$.  It then follows that the indices $|H': \N_{H'}(M'_i)|$ and $|M'_i:\N_{M'_i}(H)|$ are both finite, whilst $\QZ(H')=\triv$.  By Lemma~\ref{finint} it follows that $H' \cap M'_i$ is an infinite compact group, so $H \cap M_i$ is non-discrete.
\end{proof}
}

\begin{thm}\label{localprime}
Let $G$ be a non-trivial compactly generated \tdlc group.  Suppose that $G$ has trivial quasi-centre, no non-trivial compact normal subgroups and no non-trivial compact abelian locally normal subgroups. 
Let $\{M_1,\dots, M_n\}$ be the set of minimal non-trivial closed normal subgroups of $G$ (which is indeed finite by Proposition~\ref{cmsoc}) and let $\eta_i = \lpc(M_i)$ for $1 \le i \le n$. 

Then there are non-empty subsets $\epc_i \subseteq \eta_i$ for $1 \le i \le n$ such that given any non-trivial compact locally normal subgroup $H$ of $G$ and any set of primes $\xi$, $O_\xi(H)$ is non-trivial if and only if there is some $i \in \{1,\dots,n\}$ such that $\xi \supseteq \epc_i$ and $H \cap M_i >1$. Moreover, if~$\xi = \bigcup^n_{i=1}\epc_i$, then $O_\xi(H)$ is infinite for all such~$H$.
\end{thm}

\begin{proof}
Recall that by Proposition~\ref{cstab}, every closed locally normal subgroup of $G$ has trivial quasi-centre.

Fix $i \in \{1,\dots,n\}$.  Let $\mcD_i$ be the collection of those sets, $\xi$, of primes such that for some compact open subgroup $U$ of $G$, the $\xi$-core $O_\xi(U \cap M_i)$ is infinite.  Notice that if $\xi \in \mcD_i$, then in fact $O_\xi(U \cap M_i)$ is infinite for every compact open subgroup $U$.  Let $U$ be a compact open subgroup of $G$ and let $\xi_1,\xi_2 \in \mcD_i$.  Then the compact subgroups $O_{\xi_1}(U \cap M_i)$ and $O_{\xi_2}(U \cap M_i)$ are locally normal in~$G$, and thus both have trivial quasi-centre, and $O_{\xi_1}(U \cap M_i)$ and $O_{\xi_2}(U \cap M_i)$ normalise each other.  At the same time, $O_{\xi_1}(U \cap M_i)$ and $O_{\xi_2}(U \cap M_i)$ both have trivial quasi-centraliser in $M_i$ by Proposition~\ref{prop:TrivialQC} because, as Lemma~\ref{corelem} shows, they are commensurated in~$G$.  By applying Corollary~\ref{finintcor} to the subgroups $O_{\xi_1}(U \cap M_i)$ and $O_{\xi_2}(U \cap M_i)$ of the group $M_i$, it follows that
\[
O_{\xi_1}(U \cap M_i) \cap O_{\xi_2}(U \cap M_i) = O_{\xi_1 \cap \xi_2}(U \cap M_i)
\]
is infinite. This shows that  $\mcD_i$ is closed under finite intersections.  Moreover, $\mcD_i$ contains $\lpc(G)$, which is a finite set by Proposition~\ref{localprime:short}. Thus the poset $\mcD_i$ has a smallest element $\epc_i \subseteq \lpc(G)$, which is necessarily non-empty by the definition of $\mcD_i$.

Now let $H$ be a non-trivial compact locally normal subgroup of $G$, let $\xi$ be a set of primes and fix a compact open subgroup $U < G$ containing $H$ as a normal subgroup.

Suppose there is some $i$ such that $i \in \{1,\dots,n\}$ such that $\xi \supseteq \epc_i$ and $H \cap M_i >1$.  Then $\QC_{M_i}(O_\xi(U \cap M_i))$ is trivial, as above, and hence $O_\xi(U \cap M_i) \cap H$ is infinite by Corollary~\ref{finintcor}. Since $O_\xi(U \cap M_i) \cap H$ is a normal pro-$\xi$ subgroup of $H$, it is contained in~$O_\xi(H)$ and we have thus shown that~$O_\xi(H)$ is infinite.  Conversely, suppose that $O_\xi(H)$ is non-trivial.  Then by Proposition~\ref{prop:HelpLocalPrime}, there exists $i \in \{1,\dots,n\}$ such that $O_\xi(H) \cap M_i$ is infinite, so in particular $H \cap M_i > 1$.  At the same time, $O_\xi(H)$ is a characteristic subgroup of $H$, and hence a normal subgroup of $U$, so $O_\xi(H) \cap M_i$ is an infinite normal pro-$\xi$ subgroup of $U \cap M_i$.  Thus $O_{\xi}(U \cap M_i)$ is infinite, in other words $\xi \in \mcD_i$, ensuring that $\xi \supseteq \epc_i$.  We have now shown that $O_\xi(H)$ is non-trivial if and only if there is some $i \in \{1,\dots,n\}$ such that $\xi \supseteq \epc_i$ and $H \cap M_i >1$.

Finally, suppose that $\xi = \bigcup^n_{i=1}\epc_i$.  By Proposition~\ref{prop:HelpLocalPrime} there exists $i \in \{1,\dots,n\}$ such that $H \cap M_i$ is non-discrete.  Since $\xi \supseteq \epc_i$, the previous paragraph implies $O_{\xi}(H)$ is infinite, as required.
\end{proof}

We can prove a stronger result for locally C-stable \tdlc groups that are in $\sclass$.  In fact it will be seen later, in Theorem~\ref{thm:noqz}, that every group in~$\sclass$ is locally C-stable, so the following theorem applies to all groups in $\sclass$.

\begin{thm}\label{localprimetopsimp}
Let $G$ be a non-trivial compactly generated, topologically simple, locally C-stable \tdlc group. Then the following hold.
\begin{enumerate}[(i)]
\item Every non-trivial compact locally normal subgroup~$H\leq G$ is virtually pro-$\lpc(G)$ and, if~$H$ is virtually pro-$\xi$ for some set of prime numbers~$\xi$, then~$\xi\supseteq\lpc(G)$.
\item If~$G$ contains a non-trivial prosoluble locally normal subgroup, then every compact open $U\leq G$ is virtually prosoluble. 
\end{enumerate}
\end{thm}

\begin{proof}
Let us first note that $G$ satisfies the hypotheses of Theorem~\ref{localprime}: we have $\QZ(G)=\triv$ by Proposition~\ref{prop:TrivialQC:2} and the absence of non-trivial compact abelian locally normal subgroups is by Proposition~\ref{cstab}.

For (i), by Theorem~\ref{localprime} there is $\epc\subseteq \lpc(G)$ such that, given any non-trivial compact locally normal subgroup $H$ of $G$ and any set of primes $\xi$, $O_\xi(H)$ is non-trivial if and only if $\xi \supseteq \epc$. We have to show that $\epc= \lpc(G)$; it suffices to show that some compact open subgroup $U$ of $G$ is virtually pro-$\epc$.  Set $L = O_{\epc}(U)$.

For (ii), we take $U$ to be a compact open subgroup and set $L = O_{\infty}(U)$; in this case we may assume that $L$ is non-trivial, and our aim is to show that $U$ is virtually prosoluble.

In both cases $L$ is a non-trivial compact locally normal subgroup.  From this point onwards the proofs of (i) and (ii) are similar: let $\star$ stand for $\epc$ for the proof of (i), and $\infty$ for the proof of (ii).

Lemma~\ref{lem:CptGenNC} implies that the normal closure of $L$ is generated by a finite set of conjugates of $L$. Let thus $g_1, \dots, g_n \in G$ be such that $D= \langle L \cup g_1 L g_1\inv \cup \dots g_n L g_n\inv \rangle$ is dense and normal in $G$. Set $g_0 = 1$ and $U_i = g_i U g_i\inv$ for all $i=0, \dots, n$. Let also $V_0 $ be an open normal subgroup of $U_0$ contained in $ \bigcap_{i=0}^n U_i$. 
For all $i >0$, define inductively a group $V_i$ as the normal core of $V_{i-1}$ in the group $U_i$. Thus $V_i \lhd U_i$ and for all $i \in \{0, 1, \dots, n\}$,   we get a subnormal chain 
$$
V_n \lhd V_{n-1} \lhd \dots \lhd V_i \lhd  U_i.
$$
From Lemma~\ref{lem:subnormal}, we infer that $O^{\star}(V_n)$ is normalised by $O_{\star}(U_i)$ for all $i = 0, \dots, n$. Notice that $O_{\star}(U_i) = g_i L g_i\inv$. We deduce that $ O^{\star}(V_n)$ is normalised by $D$. Since $ O^{\star}(V_n)$ is compact, its normaliser in $G$ is closed. Therefore  $ O^{\star}(V_n)$ is normal in $G$ since $D$ is dense. 

Topological simplicity of~$G$ then implies that $ O^{\star}(V_n)$ is trivial. In other words $V_n$ is a pro-${\epc}$ group for part (i), respectively a prosoluble group for part (ii). Since $V_n$ is open in $U$ by construction, this proves that $U$ is virtually pro-${\epc}$ or prosoluble respectively, as required.
\end{proof}

\section{Properties of the structure lattice}

\subsection{The orbit join property}
\label{sec:orbit_join}

Recall that $\sclass$ is the class of non-discrete compactly generated topologically simple \tdlc groups.  We have seen (Theorem \ref{baireabs}) that if $G \in \sclass$ is abstractly simple, then $G$ has no non-trivial fixed points in $\lnorm(G)$, and thus all $G$-orbits on $\lnorm(G) \smallsetminus \{0,\infty\}$ are infinite.  However, it turns out that the orbits of $G$ on $\lnorm(G)$ are still `compact' in a certain sense: each orbit has a least upper bound, which is the join of finitely many elements of the orbit.  In the case that $G$ satisfies \sur, it follows that every compact locally normal subgroup $H$ of $G$ is relatively large, in the sense that there exists a finite set of $G$-conjugates of $H$ whose product has non-empty interior.  Even without assuming \sur, we see that in the conjugation action of $G$ on $\lnorm(G)$, the lattice $\lnorm(G)^G$ of fixed points plays an important role.  We work here in the more general setting of Hecke pairs~$(G,U)$ and, in this setting, a subgroup~$K\leq G$ is \textbf{bounded} if it is contained in a finite number of~$U$-cosets and is \textbf{locally normal} if~$N_G(K)\cap U$ has finite index in~$U$.

\begin{lem}\label{boxlem}
Let $(G,U)$ be a Hecke pair such that $G$ is generated by finitely many cosets of $U$. Let $\kappa$ be a set of bounded locally normal subgroups which is invariant under the $U$-action by conjugation, and such that $G = \langle U, \kappa \rangle$.  Then there is a finite set $\{K_1,\dots,K_n\}\subseteq \kappa$ which is a union of $U$-orbits under the conjugation action, such that   $G = \langle U, K_1, \dots, K_n \rangle$.  Moreover, setting $V =  \bigcap^n_{i=1}\N_U(K_i)$ and let $L_i = V \cap K_i \leq U$ for all $i$, we have the following:
\begin{enumerate}[(i)]
\item for each $i$, $L_i$ is a subgroup of finite index in $K_i$; 
\item $L_i$ normalises both $K_j$ and $L_j$ for all pairs $(i,j)$; and
\item $L = \prod^n_{i=1} L_i$ is a normal subgroup of $U$ that is commensurated by $G$.
\end{enumerate}
\end{lem}

\begin{proof}We have $G = \langle U,X \rangle$ for some finite set $X$ where, since $G = \langle U, \kappa \rangle$, we can choose $X = \{x_1,\dots,x_n\}$ and $\{K_1,\dots,K_n\} \subseteq \kappa$ such that $x_i \in K_i$.  Then $G = \langle U, K_1,\dots, K_n \rangle$.  Since every element of $\kappa$ is locally normal, it has only finitely many conjugates under the action of $U$, and so we may assume that $\{K_1,\dots,K_n\}$ is a union of $U$-conjugacy classes by enlarging~$n$. Set $V = \bigcap^n_{i=1}\N_U(K_i)$ and let $L_i = V \cap K_i$. Clearly $|K_i:L_i|$ is finite, and $L_i$ normalises $K_j$ and $L_j$ for all pairs $(i,j)$.  Furthermore, the conjugation action of $U$ preserves the set $\{L_1,\dots,L_n\}$ because $V$ is normal in $U$. Hence $L = \prod^n_{i=1} L_i$ is a normal subgroup of $U$. Since~$L$ is a finite index subgroup of~$K_iL$ for each $i$, it follows that $K_iL$ commensurates $L$ and hence that $\Comm_G(L) \ge \langle U, K_1L,\dots,K_nL\rangle = G$. 
\end{proof}

We now prove a version of Theorem~\ref{thmintro:FixedPointsJoin} and its corollary.

\begin{thm}\label{boxcor}
Let $G$ be in~$\sclass$. 
{Then, for each $\alpha \in \lnorm(G)$, there is a unique smallest fixed point,~$\alpha^*\in \lnorm(G)^{G}$, that is greater than or equal to~$\alpha$}. Moreover, there is a finite subset $\{g_1,\dots,g_n\}$ of $G$ such that $\alpha^* = \bigvee^n_{i=1}g_i\alpha$.
\end{thm}

\begin{proof}
Choose a compact locally normal subgroup~$K$ of~$G$ such that $[K] = \alpha$ and let $U$ be a compact open subgroup of~$G$.  Then $G = \langle U, \kappa \rangle$, where $\kappa$ is the $G$-conjugacy class of~$K$.  Since $\N_G(K)$ is open, $U$ acts on $\kappa$ by conjugation with finite orbits, and the compactness of~$K$ guarantees that $|M: M \cap U| < \infty$ for all $M \in \kappa$.  The conditions of Lemma~\ref{boxlem} are therefore satisfied and so there is a finite subset $\{K_1,\dots,K_n\} \subseteq \kappa$ and finite index subgroups $L_i$ of $K_i$ such that $L = \prod^n_{i=1} L_i$ is a normal subgroup of $U$ that is commensurated by $G$.  
{In the present setting, the subgroups $L_i$ provided by Lemma~\ref{boxlem} are} 
closed and so $L$ is compact.  Now $\gamma = [L]$ is fixed by the conjugation action of~$G$ on $\lnorm(G)$ and we can express $\gamma$ as
\[ \gamma = \bigvee^n_{i=1}[L_i] = \bigvee^n_{i=1}[K_i] = \bigvee^n_{i=1}g_i\alpha,\]
where $g_1,\dots,g_n$ are elements of $G$ such that $K_i = g_iKg\inv_i$.

Suppose that $\beta \in \lnorm(G)^{G}$ with $\beta \geq \alpha$. Then $\beta = g_i \beta \geq g_i \alpha$ for all $i = 1, \dots, n$ and so~$\beta \geq \gamma$. Hence $\gamma$ is the unique minimal element of the set $\{\beta \in \lnorm(G)^G \; | \; \beta \geq \alpha\}$.
\end{proof}

\begin{cor}\label{cor:FixedPointsGeneration}
Let $G \in \sclass$ such that $G$ satisfies \sur.  Let $H$ be a non-trivial compact locally normal subgroup of $G$.  Then there exist $g_1, \dots, g_n \in G$ such that 
\[
G = \langle g_1Hg\inv_1, g_2Hg\inv_2, \dots, g_nHg\inv_n \rangle.
\]
In particular, if $H$ is topologically finitely generated, then so is $G$.
\end{cor}

\begin{proof}
Since $G$ is topologically simple, we have $G = \overline{\lla H \rra}$.  By Proposition~\ref{prop:TrivialQC:2}, we have $\QZ(G)=\triv$, so $H$ must be infinite.  Applying Lemma~\ref{boxlem}, we produce a set $\{L_1,\dots,L_d\}$ of infinite subgroups of $G$, such that for all $1 \le i \le d$, there exists $g_i \in G$ such that $L_i \le g_iHg\inv_i$, and such that $L = \prod^n_{i=1} L_i$ is a commensurated compact locally normal subgroup of $G$.  Since $G$ has \sur, in fact $L$ is an open subgroup of $G$.  Now since $\lla H \rra$ is dense in $G$ and $G$ is compactly generated,  we can find a finite subset $\{g_{d+1},\dots,g_n\}$ such that
\[
G = \langle g_{d+1}Hg\inv_{d+1}, g_{d+2}Hg\inv_{d+2}, \dots, g_{n}Hg\inv_{n} , L \rangle.
\]
Since $L \le \langle g_iHg\inv_{i} \mid 1 \le i \le d\rangle$, we conclude that the set $\{g_iHg\inv_i \mid 1 \le i \le n\}$ generates $G$ as claimed.

If $H = \overline{\langle X \rangle}$ for some finite set $X$, then $G$ is topologically generated by finitely many conjugates of $X$, so $G$ is topologically finitely generated. 
\end{proof}

\subsection{On the non-existence of virtually abelian locally normal subgroups}\label{sec:noqz}

We would like to investigate (topologically) simple \tdlc groups by means of the Boolean algebras defined in \cite{CRW-Part1}, as we have better control over their structure than we do for the structure lattice as a whole.  Recall that to ensure the centraliser and  local decomposition lattices are indeed Boolean algebras, we needed to assume certain local properties concerning centralisers of locally normal subgroups, the strongest of which was that $G$ should not have any non-trivial compact abelian locally normal subgroups.  Fortunately, this last condition turns out to be true for a class of groups that includes every non-discrete compactly generated topologically simple \tdlc group.

{For the sake of clarity, we shall first state and prove the absence of virtually abelian locally normal subgroups in groups belonging to the class $\sclass$. We will then generalise this to the framework of Hecke pairs.

\begin{thm}\label{thm:noqz} \
Let $G \in \sclass$.  Then $\QZ(G)=\triv$ and $G$ has no non-trivial abelian locally normal subgroups. In particular $G$ is locally C-stable by Proposition~\ref{cstab}.
\end{thm}

In particular, we recover  \cite[Theorem 2.2]{Willis}, which asserts that the compact open subgroups of $G$ cannot be soluble. 

\begin{rem}
It is shown in \cite{Willis} that there are non-discrete topologically simple \tdlc groups whose compact open subgroups are virtually abelian, so the condition of compact generation is essential in Theorem~\ref{thm:noqz}.
\end{rem}

\begin{proof}[Proof of Theorem~\ref{thm:noqz}]
That $\QZ(G)=\triv$ has already been proved in Proposition~\ref{prop:TrivialQC:2}. Let now $K$ be a non-trivial abelian locally normal subgroup. Then $\overline{K}$ is also an abelian locally normal subgroup, so we may assume that $K$ is closed.  Moreover, $K$ is non-discrete since $\QZ(G)=\triv$, so by intersecting $K$ with a compact open subgroup, we may assume that $K$ is compact. By Theorem~\ref{boxcor}, the group $G$ has a commensurated compact locally normal subgroup $L$ of the form $L = \prod_{i=1}^n L_i$, where the $L_i$ are locally normal, abelian (in fact they are each conjugate to an open subgroup of $K$), and normal in $L$. By Fitting's theorem, it follows that $L$ is nilpotent; in particular $\Z(L) > 1$. On the other hand $L$ is an infinite, compact, commensurated, locally normal subgroup of $G$, and thus has a trivial quasi-centraliser by Proposition~\ref{prop:TrivialQC:2}. This is a contradiction. 
\end{proof}

We have now proved Theorem~\ref{thmintro:QZ}; we can now also deduce Theorem~\ref{thmintro:algebraicLN} from Theorem~\ref{localprimetopsimp}, since the only missing ingredient was the fact that all groups in $\sclass$ are locally C-stable.

The conclusion of Theorem~\ref{thm:noqz} may be extended to compactly generated topologically  \emph{characteristically} simple \tdlc groups.


\begin{lem}[See for instance {\cite[Lemma~8.2.3]{RZ}}]\label{lem:profchar}
Let $G$ be a profinite group that is topologically characteristically simple.  Then $G$ is a direct product of finite simple groups.  In particular, $\QZ(G)$ is dense in $G$.
\end{lem}

\begin{prop}\label{noqzchar}Let $G$ be a compactly generated, topologically characteristically simple \tdlc group.  Then exactly one of the following holds:
\begin{enumerate}[(i)]
\item $G$ is finite.
\item $G$ is countably infinite and discrete.
\item $G$ is an infinite profinite group and $\QZ(G)$ is dense.
\item $G$ is locally C-stable and $\QZ(G)=\triv$.
\end{enumerate}
\end{prop}

\begin{proof}
If $G$ is compact (which is to say that $G$ is profinite, since we have assumed $G$ is totally disconnected), it follows that $\QZ(G)$ is dense by Lemma~\ref{lem:profchar}.  If $G$ is infinite and discrete, then it is finitely generated, hence countable.  From now on we may assume $G$ is neither compact nor discrete.

By \cite[Corollary D]{CM}, there is a finite set of topologically simple closed normal subgroups $\{M_1,\dots,M_n\}$, such that the product $\prod^n_{i=1}M_i$ is a dense subgroup of $G$, and $M_i \cap M_j = \triv$ for $i$ and $j$ distinct.  Let $C_i = \CC_G(M_i)$; note that $C_i$ contains $M_j$ for all $j \not=i$.  We claim that $\Z(G)=\triv$: otherwise, $G$ would be abelian and the $M_i$ would have to be cyclic of prime order, but then $G$ would itself be finite and hence discrete, a case we have already removed.  Thus $M_i \cap C_i = \triv$ for $1 \le i \le n$.  It follows that $G/C_i$ is topologically simple, since every non-trivial closed normal subgroup of $G/C_i$ has non-trivial intersection with $M_iC_i/C_i$ and is therefore dense in $G/C_i$.  Thus $G/C_i$ is a locally C-stable \tdlc group with trivial quasi-centre by Theorem~\ref{thm:noqz}.  Since the properties of being quasi-central and being abelian locally normal pass to quotients, it follows that $C_i$ contains $\QZ(G)$ and all abelian locally normal subgroups of $G$ for all $i$.  Since $K = \bigcap^n_{i=1} C_i$ is a characteristic subgroup of $G$, it is trivial.  We conclude that the quasi-centre and all abelian locally normal subgroups of $G$ are trivial, hence $G$ is locally C-stable by Proposition~\ref{cstab}.
\end{proof}

In the case of an arbitrary compactly generated \tdlc group $G$, as a result of Theorem~\ref{thm:noqz}, there is an interesting interaction between the \textbf{$\VA$-regular radical}\index{A-regular radical@$\VA$-regular radical}\index{RA(G)@$\R_\VA(G)$} of $G$, denoted by $\R_\VA(G)$ (see \cite[Theorem~III]{CRW-Part1}) and the quotients of $G$ occurring in  Theorem~\ref{cmquotthm} (which was borrowed from  \cite[Theorem~A]{CM}).

\begin{cor}Let $G$ be a compactly generated \tdlc group.
\begin{enumerate}[(i)]
\item Let $N$ be a cocompact normal subgroup of $G$ and $R$ be a closed normal subgroup of $N$ such that $N/R$ is non-discrete and topologically simple.  Then ${N \cap \R_\VA(G) \le R}$.  Defining $N_2 = \overline{N\R_\VA(G)}$ and $R_2 = \{g \in N_2 \mid  \forall h \in N: \; [g,h] \in R\}$, we have: $N_2/\R_\VA(G)$ is a cocompact normal subgroup of $G/\R_\VA(G)$; $R_2$ is a closed normal subgroup of $N_2$ which contains $\R_\VA(G)$; $N_2/R_2$ is non-discrete, topologically simple and isomorphic to a quotient of $G/\R_\VA(G)$; and the homomorphism $\phi: N/R \rightarrow N_2/R_2$ given by $\phi(gR) = gR_2$ is injective and has dense image. All non-discrete simple quotients of cocompact normal subgroups of $G$ are thus accounted for by the quotient $G/\R_\VA(G)$.
\item Suppose $G$ is non-compact and  $\VA$-regular (\ie $G = \R_\VA(G)$).  Then $G$ has an infinite discrete quotient.\end{enumerate}
\end{cor}

\begin{proof}
(i) Since $N$ is cocompact in $G$, it is compactly generated, and hence so is the topologically simple quotient $N/R$. From Theorem~\ref{thm:noqz}, we deduce that $\R_\VA(N/R)=\triv$, so $\R_\VA(N) \leq R$. By \cite[{Proposition~6.15}]{CRW-Part1}, we have $ \R_\VA(G) \cap N \leq \R_\VA(N)$, and hence ${N \cap \R_\VA(G) \le R}$.

Let $N_2 = \overline{N\R_\VA(G)}$ and $R_2 = \{g \in N_2 \mid \forall h \in N: \; [g,h] \in R\}$.  Then $N_2$ is cocompact in $G$, since it contains $N$, and so $N_2/\R_\VA(G)$ is cocompact in $G/\R_\VA(G)$.  Since the centre of $N/R$ is trivial, we see that $N \cap R_2 = R$, and hence that the homomorphism $\phi$ is injective.  On the other hand, we have $[N,\R_\VA(G)] \le N \cap \R_\VA(G) \le R$, and so $\R_\VA(G) \le R_2$ and $\phi$ has dense image.  It remains to show that $N_2/R_2$ is topologically simple. Let $K/R_2$ be a proper closed normal subgroup of $N_2/R_2$.  Then $K$ does not contain $N$, and it follows that $\phi\inv(K/R_2)$ is a proper closed normal subgroup of $N/R$ and, by topological simplicity of $N/R$, must be equal to $\triv$.  Thus $K/R_2$ and $NR_2/R_2$ are normal subgroups of $N_2/R_2$ with trivial intersection.  We conclude that $K/R_2$ centralises $NR_2/R_2$, that is, $[g,h] \in R_2$ for all $g \in K$ and $h \in N$.  Since $[g,h] \in N$ as well in this case, in fact $[g,h] \in N \cap R_2 = R$ and so $K \le R_2$.  Thus $K/R_2 = \triv$, as required.

(ii) Follows from (i) and Theorem~\ref{cmquotthm}.
\end{proof}

We now proceed to generalise Theorem~\ref{thm:noqz} to the abstract framework of Hecke pairs. }

\begin{thm}\label{noqz:Hecke}
Let $(G,U)$ be a Hecke pair such that $U < G$ and such that $U$ is infinite.  Suppose that $G$ is generated by finitely many cosets of $U$, and also that $G = \langle U,N \rangle$ for any non-trivial normal subgroup $N$ of $G$.

{Then the following assertions hold for any non-trivial subgroup $K$   of $U$ such that $\N_U(K)$ has finite index in $U$:
\begin{enumerate}[(i)]
\item $K$ is not virtually abelian.

\item  If $K$ is commensurated by $G$, then $\CC_G(K)=\triv$.
\end{enumerate}
}
\end{thm}

\begin{proof}
Let $K$ be a non-trivial subgroup of $U$ such that $|U:\N_U(K)|$ is finite.

Let us first prove (ii) in case $K$ is infinite. Suppose for a contradiction that  $C = \CC_G(K)>1$.  
Then $G$ is generated by $U$ together with a finite set $\{g_1Cg\inv_1,\dots,g_nCg\inv_n\}$ of conjugates of $C$, and without loss of generality we may assume that $g_1 = 1$.  Moreover $U$ has finite orbits on the set of conjugates of $K$, since $\N_U(K)$ is a commensurated subgroup of $G$ that has finite index in $U$.  Hence we may take $\{g_1Kg\inv_1,\dots,g_nKg\inv_n\}$ to be a $U$-invariant set.  Consequently the group $L = \bigcap^n_{i=1}g_iKg\inv_i$ is normal in $G$, since it is normalised by $U$ and centralised by $\langle g_1Cg\inv_1,\dots,g_nCg\inv_n \rangle$.  Since  { $L \leq U$ (because $L \leq  g_1 Kg_1^{-1} = K \leq U$)}, we conclude that $L=1$, so $K$ cannot have been an infinite commensurated subgroup of $G$.  This proves (ii)  in the case where $K$ is infinite.

If $K$ is finite, we see that $L = \CC_U(K)$ has finite index in $U$ and is thus an infinite commensurated subgroup of $G$. What we have just proved ensures that $\CC_G(L) = \triv$, so that $K = \triv$, a contradiction. This completes the proof of (ii).

This implies  that $\CC_G(V)=\triv$ for every subgroup $V$ of $G$ that is commensurate to $U$.

We now prove (i), and assume thus that $K$ is virtually abelian. As noted previously, part (ii) implies that $K$ must be infinite. Moreover, the hypotheses imply that $U$ has trivial core in $G$. In particular $U$ is residually finite, so there is an injective map $\pi: U \rightarrow \hat{U}$ from $U$ to its profinite completion.  We see that $\overline{\pi(K)}$ has an abelian closed subgroup of finite index; by a compactness argument, there exists an open normal subgroup $W$ of $\hat{U}$ such that $\overline{\pi(K)} \cap W$ is abelian.  The group $K' = \pi\inv(W) \cap K$ is then an abelian subgroup of finite index in $K$, such that $K'$ is normalised by $\N_U(K)$. Therefore, in order to prove (i), we may assume henceforth that $K$ is abelian and derive a contradiction.  We have $G = \langle U, \kappa \rangle$ where $\kappa$ is the set of $G$-conjugates of $K$.  By Lemma \ref{boxlem}, there is a set $\{L_1,\dots,L_n\}$ of subgroups of $G$, each a finite index subgroup of a conjugate of $K$, such that $L_i$ and $L_j$ normalise each other for all all pairs $(i,j)$ and $L = \prod^n_{i=1} L_i$ is an infinite normal subgroup of $U$ that is commensurated by $G$.  Now each $L_i$ is abelian, so $L$ is nilpotent by Fitting's theorem, and hence $\CC_G(L) \ge \Z(L) > 1$.  This is impossible, as we have already shown that any infinite normal subgroup of $U$ that is commensurated by $G$ must have trivial centraliser.
\end{proof}

It is straightforward to deduce Theorem~\ref{thm:noqz} above from Theorem~\ref{noqz:Hecke}, thereby providing an alternative proof of the former. The following particularisation  of  Theorem~\ref{noqz:Hecke}, whose proof is straightforward, is another special case of independent interest:

\begin{cor}\label{noqzcor}
Let $G$ be a finitely generated simple group with an infinite commensurated subgroup $U$.  Then $\CC_G(U)=\triv$, and $U$ does not have any virtually abelian normal subgroups except for the trivial group.
\end{cor}

\subsection{Five possible types of structure lattice}\label{sec:types}

Our aim in this subsection is to prove Theorem~\ref{thmintro:types}.  We first make an easy observation about the structure of $\lnorm(G) \smallsetminus \{0\}$.

An \textbf{upper subset}\index{upper subset} of a poset $X$ is a subset $Y$ such that for all $y \in Y$ and $x \in X$, if $y \leq x$ then $x \in Y$.  A \textbf{filter}\index{filter} of a lattice is an upper subset that is also closed under meets.

\begin{lem}\label{lem:lnorm_filter}
Let $G$ be a non-discrete locally C-stable \tdlc group and let $\mcF := \lnorm(G) \smallsetminus \{0\}$.  Then $\mcF$ is closed under meets  
if and only if $\lcent(G) = \{0,\infty\}$.
\end{lem}

\begin{proof}
$\mcF$ is clearly an upper subset of  $\lnorm(G)$, so it is a filter if and only if it is closed under meets.

If $\lcent(G)$ is non-trivial, then there exists an infinite compact locally normal subgroup $K$ of $G$ such that $\CC_G(K)$ is infinite.  Then $\alpha = [K]$ and $\beta = [\CC_G(K)]$ are elements of $\mcF$ such that $[K] \wedge [\CC_G(K)] = [\Z(K)] = 0$, recalling that $G$ has no non-discrete abelian locally normal subgroups by Proposition~\ref{cstab}.  Thus $\mcF$ is not closed under meets.  Conversely, if there exist $\alpha, \beta \in \mcF$ such that $\alpha \wedge \beta = 0$, then we have $0 < \beta \le \alpha^\bot < \infty$, so $\alpha^\bot$ is an element of $\lcent(G)$ other than $0$ and $\infty$.
\end{proof}

For $G \in \sclass$, the set of non-zero \emph{fixed points} of $G$ acting on $\lnorm(G)$ forms a filter, even if $\lcent(G)$ is non-trivial.

\begin{lem}\label{lnormfix}Let $G$ be a non-discrete \tdlcsc group.
\begin{enumerate}[(i)]
\item If $G$ is topologically simple, then $\lnorm(G)^G \smallsetminus \{0\}$ is an upper subset of $\lnorm(G)$.
\item If $G$ is monolithic, locally C-stable and has trivial quasi-centre, then $\lnorm(G)^G \smallsetminus \{0\}$ is closed under meets.
\end{enumerate}
Hence if $G$ is topologically simple and locally C-stable, then $\lnorm(G)^G \smallsetminus \{0\}$ is a filter on $\lnorm(G)$.\end{lem}

\begin{proof}Let $\mcF = \lnorm(G)^G \smallsetminus \{0\}$.

(i)
Let $K$ be a compact locally normal subgroup of $G$ such that $[K] \in \mcF$.  Then $G$ acts on the set $\mcR$ of elements of $\lnorm(G)$ above $[K]$; the kernel $R$ of this action is closed and normal in $G$.  Moreover $R$ contains $K$, since for every $\beta > [K]$, there is a representative $L$ of $\beta$ containing $K$.  Hence $R \ge \overline{\lla K \rra}$.  Since $G$ is topologically simple, we conclude that $R = G$, so $G$ acts trivially on $\mcR$.  Thus $\beta > [K]$ implies $\beta \in \mcF$.

(ii) 
Note that the hypotheses ensure $\QZ(H)=\triv$ for every non-trivial closed locally normal subgroup $H$ of $G$ by Proposition~\ref{cstab}.  We suppose that $\alpha, \beta \in \mcF$ are such that $\alpha \wedge \beta = 0$; it suffices to derive a contradiction.  

Let $M$ be the monolith of $G$ and let $\mu = [M] \in \lnorm(G)$.  Then $\mu \in \lnorm(G)^G$, so $\gamma \wedge \mu \in \lnorm(G)^G$ for all $\gamma \in \mcF$. Moreover we have $\gamma \wedge \mu > 0$ for all $\gamma \in \mcF$, since otherwise $M$ would have a non-trivial quasi-centraliser in $G$, contradicting Proposition~\ref{prop:TrivialQC}.  In particular, we infer that $\alpha \wedge \mu$ and $\beta \wedge \mu$ both belong to $\mcF$.  Since $\alpha \wedge \beta = 0$, there are infinite compact subgroups $K$ and $L$ of $M$ such that $K \in \alpha \wedge \mu$, $L \in \beta \wedge \mu$, $K$ and $L$ normalise each other and $K \cap L = \triv$, so $L \le \CC_M(K)$.  The desired contradiction follows by  {Proposition~\ref{prop:TrivialQC}}.
In particular $\alpha \wedge \beta  > 0$, so $\alpha \wedge \beta \in \mcF$.
\end{proof}

We immediately derive an important property of the action of $G$ on $\lcent(G)$.

\begin{cor}\label{lnormfix:lcent}
Let $G$ be a topologically simple, locally C-stable \tdlcsc group.  Then $\lcent(G)^{G} = \{0,\infty\}$.  
\end{cor}

\begin{proof}
Let $\alpha \in \lcent(G)^{G}$ and suppose that $\alpha \not\in \{0,\infty\}$.  Since the map $\bot \colon \lcent(G) \rightarrow \lcent(G)$ is invariant under the $G$-action, it follows that $\alpha^\bot$ is also fixed by $G$; at the same time,  $\alpha^\bot \not\in \{0,\infty\}$.  Thus $\alpha$ and $\alpha^\bot$ are non-zero elements of $\lnorm(G)^G$.  It follows from Lemma~\ref{lnormfix}(ii) that $\alpha \wedge \alpha^\bot > 0$; however, $\alpha \wedge \alpha^\bot = 0$ by Theorem~\ref{thm:Boolean-part1}.
%
\end{proof}

Theorem~\ref{thmintro:types} is now straightforward to prove.

\begin{proof}[Proof of Theorem~\ref{thmintro:types}]
It is clear that the five types are mutually exclusive.

Let $G \in \sclass$.  Since $G$ is non-discrete, we have $|\lnorm(G)| \ge 2$.  By Theorem~\ref{thm:noqz}, we have $\QZ(G)=\triv$ and $G$ has no non-trivial abelian locally normal subgroup, so $\lcent(G)$ and $\ldlat(G)$ are Boolean algebras by Theorem~\ref{thm:Boolean-part1}.

Suppose $|\lnorm(G)| = 2$.  Then $\lnorm(G) = \{0,\infty\}$, so every compact locally normal subgroup of $G$ is finite or open.  Moreover, $G$ has no non-trivial finite locally normal subgroups since $\QZ(G)=\triv$.  In particular, given any compact open subgroup $U$ of $G$, then every non-trivial closed locally normal subgroup of $U$ is open, in other words $U$ is h.j.i.  Thus $G$ is locally h.j.i.  From now on we assume $|\lnorm(G)| > 2$.

Suppose that $\lcent(G) = \{0,\infty\}$.  Then $\mcF$ is a filter by Lemma~\ref{lem:lnorm_filter}.  If $\mcF$ is a principal filter then the least element of $\mcF$ must be fixed by $G$, by uniqueness; consequently by Lemma~\ref{lnormfix}, $G$ acts trivially on $\mcF$ and hence on $\lnorm(G)$.  The fact that $G$ has non-trivial fixed points on $\lnorm(G)$ ensures that $G$ does not satisfy \sur, so by Theorem~\ref{baireabs}, $G$ is not abstractly simple.  Thus $G$ is of atomic type in this case.   If instead $\mcF$ is a non-principal filter, then $G$ is of NPF type.

Suppose that $\lcent(G)$ is non-trivial.  If $\ldlat(G) = \{0,\infty\}$ then $G$ is weakly decomposable, otherwise $G$ is locally decomposable.

Given the list of types, it is clear that we can recover the type of $G$ from the isomorphism type of the poset $\lnorm(G)$.
\end{proof}

\subsection{Dense normal subgroups of topologically simple groups}\label{sec:fixedpoints}

In this subsection, $G$ will be a non-discrete \tdlcsc group.  Recall that $\lnorm(G)^G$ is the set of fixed points of the action of $G$ on $\lnorm(G)$, in other words the set of local equivalence classes of commensurated compact locally normal subgroups of $G$.  Evidently $\lnorm(G)^G$ is a sublattice of $\lnorm(G)$ and contains both $0$ and $\infty$.  By Theorem~\ref{baireabs}, if $G$ is abstractly simple then $\lnorm(G)^G = \{0,\infty\}$.  Thus if $G$ is topologically simple, the existence of a non-trivial element of $\lnorm(G)^G$ implies the existence of a proper dense normal subgroup of $G$.  Our goal is to obtain restrictions on this situation by describing the structure of the dense normal subgroups that could arise.

Before stating our results, we recall some terminology and results that will be used in this subsection and the next.

\begin{defn}
Let $G$ be a \tdlc group and let $\mu$ be a commensurability class (or subset of a commensurability class) of compact subgroups of $G$.  Let $K \in \mu$ and consider $\Comm_G(K)$.  There is then a unique group topology $\mathcal{T}_{(\mu)}$ on $\Comm_G(K)$ such that the inclusion $L \rightarrow \Comm_G(K)$ is continuous and open, for every $L \in \mu$ (both the group $\Comm_G(K)$ and the topology are uniquely determined by the pair $(G,\mu)$).  The \defbold{localisation}\index{localisation}\index{localised topology}\index{Gmu@$G_{(\mu)}$} $G_{(\mu)}$ of $G$ at $\mu$ is then the topological group $(\Comm_G(K),\mathcal{T}_{(\mu)})$.
\end{defn}

The existence of the topology $\mathcal{T}_{(\mu)}$ follows from the fact that in $\Comm_G(K)$, every left coset of an open subgroup of $L$ is a finite union of right cosets of an open subgroup of $L$, and vice versa.  The uniqueness of the topology follows from the general observation that a group topology is uniquely determined by its restriction to a neighbourhood of the identity.
Starting with a first countable \tdlc group $G$, one sees that the \tdlc groups with the same group structure as $G$ but a finer topology are precisely the groups $G_{(\mu)}$, where $\mu$ is a $G$-invariant commensurability class of compact subgroups of $G$.  Indeed, we recall by Lemma~\ref{reflem} that it suffices to consider $G$-invariant commensurability classes of compact \emph{locally normal} subgroups of $G$.

\begin{defn}
Following H.~Abels \cite{Abels-CptPresent}, we say that a locally compact group $G$ is \textbf{compactly presented}\index{compactly presented} if there is a surjective homomorphism $\theta \colon  F_X \rightarrow G$, where $F_X$ is the abstract free group on the set $X$, so that $\theta(X)$ is compact in $G$ and the kernel of $\theta$ is generated by words in the alphabet $X \cup X\inv$ of bounded length.
\end{defn}

Many well known properties of finitely presented discrete groups carry over to compactly presented locally compact groups (see \cite[Chapter 8]{CornulierHarpe} for a detailed account). A sufficient condition for $G$ to be compactly presented is that $G$ admits some continuous, proper cocompact action on a simply connected proper geodesic metric space (see \cite[Corollary~8.A.9]{CornulierHarpe}). In particular any centreless simple Lie group $G$ is compactly presented, since the coset   space $G/K$ modulo a  maximal compact subgroup $K$ is contractible.

\begin{lem}[{\cite[Proposition~8.A.10]{CornulierHarpe}}]\label{lem:CompactlyPresented:DiscreteExtension}
Let $\varphi \colon \tilde G \to G$ be a continuous surjective homomorphism of locally compact groups, whose kernel $ \Ker(\varphi)$ is discrete. If $G$ is compactly presented and $\tilde G$ is compactly generated, then $\Ker(\varphi)$ is finitely generated as a normal subgroup of $\tilde G$. 
\end{lem}


Let $\alpha \in \lnorm(G)^G \smallsetminus \{0\}$, let $L$ be a compact locally normal representative of $\alpha$ and let $D = \lla L \rra$, equipped with the $\mathcal T_{(\alpha)}$-subspace topology.  Then $D$ is a $\mathcal T_{(\alpha)}$-open subgroup of $G$.  Moreover, equipping the direct product $G \times D$ with the product topology, then by \cite[Proposition~7.9]{CRW-Part1}, the map 
\[ G \times D \rightarrow D; \; (g,d) \mapsto gdg\inv \]
is continuous.  This ensures that the natural semidirect product $D \rtimes G$ is a topological group when equipped with the product topology, which allows us to analyse the structure of the embedding $D \hookrightarrow G$, especially when $D$ is dense in $G$ (as is certainly the case if $G$ is topologically simple).

\begin{prop}\label{prop:compression_factoring}
Let $G$ be a non-discrete \tdlcsc group and  $L$ be an infinite commensurated compact locally normal subgroup of $G$. Let $\alpha = [L]$ and   $D = \lla L \rra$ be the abstract normal closure of $L$ in $G$, equipped with the $\mathcal T_{(\alpha)}$-subspace topology. Let $\psi \colon D \rightarrow G$ be the natural inclusion map and let $U$ be a compact open subgroup of $\N_G(L)$ containing $L$.  Suppose that $D$ is dense in $G$.

Then there is a \tdlc group $\tilde{G}$, a closed continuous injective homomorphism   $\iota \colon D \rightarrow \tilde{G}$, and a quotient homomorphism $\pi \colon \tilde{G} \rightarrow G$ such that
\begin{enumerate}[(1)]
\item $\psi=\pi \circ \iota$;
\item $\iota(D)$ is a closed cocompact normal subgroup of $\tilde{G}$, with $\tilde{G}/\iota(D) \cong U/L$;
\item $\Ker(\pi)$ is discrete and centralises $\iota(D)$;
\item Every element of $\Ker(\pi)$ lies in a finite conjugacy class of $G$;
\item $\tilde{G} = \overline{\iota(D)\Ker(\pi)}$;
\item If $\Ker(\pi)$ or $U/L$ is virtually abelian, then $D \ge [G,G]$.
\end{enumerate}

Furthermore, if $G$ is compactly presented, then in addition, $\Ker(\pi)$ is finitely generated.  Consequently, in this case $\Ker(\pi)$ is centralised by a finite index subgroup of $\tilde{G}$ (so $\Ker(\pi)$ is virtually abelian), and moreover $U/L$ is centre-by-finite, finite-by-abelian, and topologically finitely generated.
\end{prop}

\begin{proof}
Form $G^{\rtimes} = D \rtimes U$ where $U$ acts on $D$ by conjugation, and equip $G^{\rtimes}$ with the product topology. By \cite[Proposition~7.9]{CRW-Part1}, the action of $U$ on $D$ is  continuous, so that $G^{\rtimes}$ is a topological group. There is also a natural inclusion $\iota_1 \colon D \rightarrow G^{\rtimes}$, which is a closed embedding with cocompact image.

Given a subgroup $K$ of $U$, let 
\[
\Delta_K := \{(k\inv,k) \in G^{\rtimes} \mid k \in K \cap D\}.
\]
The set $\Delta_K$ is a subgroup of $G^{\rtimes}$ that centralises $\iota_1(D)$, and if $\psi^{-1}(K)$ is compact, then $\Delta_K$ is compact.  If $K$ is normal in $U$, it follows further that $\Delta_K$ is normal in $G^{\rtimes}$. In particular, the set $\Delta_L$ is a compact normal subgroup of $G^{\rtimes}$.

We now set
\[ 
\tilde{G} := G^{\rtimes}/\Delta_{L}
\]
let $\rho\colon G^{\rtimes}\rightarrow\tilde{G}$ be the usual projection.  Note that $\rho$ is a quotient map with compact kernel, so it is a closed map.  Set $\iota:=\rho\circ \iota_1$; it is immediately clear that $\iota$ is continuous and injective.  Moreover, $\iota$ is a closed map, since $\iota_1$ and $\rho$ are both closed.  Thus $\iota$ is a closed embedding.  

The homomorphism $\pi \colon \tilde{G}\rightarrow G$ is defined by 
\[
\pi\left((d,u)\Delta_{L}\right) = \psi(d)u.
\]
The image of $\pi$ is dense and contains the open subgroup $U$ of $G$, hence the map is surjective. It is easy to check this map is also continuous. Since $G$ is second countable, so are  $G^{\rtimes}$ and $\tilde G$. By \cite[Corollary to Theorem~8]{Arens}, a continuous surjective homomorphism between second countable locally compact groups is open. Therefore, the homomorphism  $\pi$ is open, hence a quotient map, and $G$ isomorphic to the quotient $\tilde{G}/\Ker(\pi)$.  Setting $\tilde{U}:=\rho(U)$, the group $\tilde{U}$ is a compact open subgroup of $\tilde{G}$ since $\rho(U) = \rho(U\Delta_L)$ and since the group $U\Delta_L = L\rtimes U$ is  open in $G^{\rtimes}$. Furthermore, $\tilde{U}\cap \Ker(\pi)=\triv$, hence $\Ker(\pi)$ is discrete.

We now check the desired properties hold of $\tilde{G}$, $\iota$ and $\pi$. Part (1) is immediate. That $\iota(D)$ is cocompact follows since $\rho$ induces a continuous surjective map $G^{\rtimes}/\iota_1(D)\rightarrow \tilde{G}/\iota(D)$. We thus have verified (2).   

The kernel of $\pi$ is exactly $K:=\Delta_{U}/\Delta_{L}$.  In particular, we see that $K$ is discrete. The group $\Delta_{U}$ is centralised by $\iota_1(D)$, hence $K$ is centralised by $\iota(D)$, verifying (3).  We conclude that $\CC_{\tilde{G}}(K)$ is cocompact in $\tilde{G}$. On the other hand, $K$ is a discrete normal subgroup of $\tilde{G}$, so $\CC_{\tilde{G}}(x)$ is open in $\tilde{G}$ for every $x \in K$.  Consequently, for every $x \in K$, then $\CC_{\tilde{G}}(x)$ is both cocompact and open, so it has finite index in $\tilde{G}$, verifying (4).

Since $\pi$ is a quotient map and $\psi(D)$ is dense in $H$, we see that $\pi\inv(\psi(D))$ is dense in $\tilde{G}$.  In other words, $\iota(D)\Ker(\pi)$ is dense in $\tilde{G}$, verifying (5).

We see from properties (2) and (5) that $\Ker(\pi)$ is isomorphic to a dense subgroup of $U/L$, so $\Ker(\pi)$ is virtually abelian if and only if $U/L$ is virtually abelian.  Suppose that $U/L$ is virtually abelian, in other words, there exists a compact open subgroup $V$ of $U$ such that $L \le V$ and $V/L$ is abelian.  We now perform the same construction of $\tilde{G}$ as before, with $V$ in place of $U$.  Now since $V/L$ is abelian, property (2) ensures that $\tilde{G}/\iota(D)$ is abelian, so $\iota(D)$ contains the derived group of $G$; since $G = \pi(\tilde{G})$, it follows that $D = \pi(\iota(D)) \ge [G,G]$, proving (6).

Now suppose that $G$ is compactly presented.  In particular, $G$ is compactly generated, so $D$ is compactly generated by Lemma~\ref{lem:CptGenNC}.  Since $\iota_1(D)$ is cocompact in $\tilde{G}$, it follows that $\tilde{G}$ is compactly generated.  We can now apply Lemma~\ref{lem:CompactlyPresented:DiscreteExtension} to conclude that $\Ker(\pi) = \lla X \rra$ for a finite subset $X$ of $G$.  By part (4), each element of $X$ only has finitely many $\tilde{G}$-conjugates, so $\Ker(\pi) = \langle Y \rangle$ for a finite set $Y$.    Let $C = \bigcap_{y \in Y} \CC_{\tilde G}(y)$, so that $C$ centralises $\Ker(\pi)$. Then $C$ is closed in $\tilde{G}$ (since every centraliser is closed) and of finite index (since each element of $Y$ has a finite conjugacy class).  It follows that $C$ is open in $\tilde{G}$ and $C\iota(D)/\iota(D)$ is an open subgroup of $\tilde{G}/\iota(D)$ of finite index.  By property (5), the image of $\Ker(\pi)$ in $\tilde{G}/\iota(D)$ is dense, so $C\iota(D)/\iota(D)$ has dense centraliser in $\tilde{G}/\iota(D)$; consequently, it is central in $\tilde{G}/\iota(D)$.  Thus $\tilde{G}/\iota(D)$ is centre-by-finite; since $\tilde{G}/\iota(D) \cong U/L$, it follows that $U/L$ is centre-by-finite.  Since $\Ker(\pi)$ has dense image in $\tilde{G}/\iota(D)$, it follows that $\tilde{G}/\iota(D)$ is topologically finitely generated, so $U/L$ is topologically finitely generated.

By a classical result of I. Schur \cite{Schur} (see also \cite{Rosenlicht} for a short proof of a stronger result), every centre-by-finite group is finite-by-abelian.  Thus $U/L$ is also finite-by-abelian in the case that $G$ is compactly presented.
\end{proof}

We can now establish a particular circumstance in which there are no non-trivial fixed points in the structure lattice of $G$.

\begin{cor}\label{cor:FixedPointsSimple}
Let $G\in \sclass$.  Assume that $G$ is compactly presented. If $G$ is abstractly perfect, or if some compact open subgroup $
U$ of $G$ is such that $U/\overline{[U, U]}$ is finite, then $\lnorm(G)^G =\{0, \infty\} $.
\end{cor}

\begin{proof}
Let $\alpha \in \lnorm(G)^G \smallsetminus \{0\}$, let $L$ be a compact locally normal representative of $\alpha$ and let $D = \lla L \rra$.  Then by Proposition~\ref{prop:compression_factoring}, we have $D \ge [G,G]$ and $U/L$ is finite-by-abelian.  If $G = [G,G]$, it follows that $D = G$.  If $U/\overline{[U,U]}$ is finite, then $U/L$ must be finite, so $L$ is open in $G$ and hence $D = G$, since $D$ is dense in $G$.  In either case, $G = \lla L \rra$.  Hence $G$ is $L$-meagre by Lemma~\ref{countsize}, in other words, the index of $L$ in $G$ is countable.  By the Baire Category Theorem, it follows that $L$ is open in $G$, so $\alpha = \infty$.  This proves that $\lnorm(G)^G =\{0, \infty\}$ under the given hypotheses.
\end{proof}

Given $G \in \sclass$ and $D$ as in Proposition~\ref{prop:compression_factoring}, it is not clear if $D \in \sclass$.  However, we can show that the closure in the $\mathcal T_{(\alpha)}$-topology of the derived group of $D$ belongs to the class $\sclass$.

\begin{prop}\label{compression:simple}
Let $G \in \sclass$, let $\alpha \in \lnorm(G)^G \smallsetminus \{0\}$, let $L$ be a compact locally normal representative of $\alpha$, and let $D = \lla L \rra$ be the abstract normal closure of $L$ in $G$, equipped with the $\mathcal T_{(\alpha)}$-subspace topology.  Let $S = \overline{[D,D]}$, where the closure is taken with respect to the topology of $D$.  Then $S \in \sclass$.  Moreover, $S = \lla M \rra$ for some infinite commensurated compact locally normal subgroup $M$ of $G$ contained in $L$, and the topologies $\mathcal T_{(\alpha)}$ and $\mathcal T_{(\beta)}$ coincide on $S$, where $\beta = [M]$.
\end{prop}

\begin{proof}
Form the semidirect product $G^{\rtimes} = D \rtimes G$, equipped with the product topology.  As observed in the proof of Proposition~\ref{prop:compression_factoring}, $G^{\rtimes}$ is a topological group.  In particular, given any $\mathcal T_{(\alpha)}$-closed subgroup $K$ of $D$ such that $\N_G(K)$ is dense in $G$, then $K \rtimes \triv$ is a closed subgroup of $G^{\rtimes}$ that is normalised by a dense subgroup of $\triv \rtimes G$, hence $K \rtimes \triv$ is normalised by all of $\triv \rtimes G$ (since the normaliser of any closed subgroup is closed), and so $K$ is normal in $G$.

By Theorem~\ref{thm:noqz}, the group $G$ has no non-trivial virtually abelian locally normal subgroups.  Consequently, both $L$ and $\overline{[L,L]}$ are non-trivial, hence non-(virtually abelian).  This ensures that $S \ge \overline{[L,L]}$ is non-discrete.  We now claim that any closed subgroup $K$ of $D$ that is normalised by $S = \overline{[D,D]}$ must satisfy $K \ge S$.  This will ensure in particular that $S$ is topologically simple.

Let $K$ be a closed subgroup of $D$ such that $S \le \N_D(K)$.  Since $S$ is a closed normal subgroup of   $D$, we deduce from the first paragraph of the proof above that  $S$ is normal in $G$.  Since $S$ is non-trivial and $G$ is topologically simple, it follows that $S$ itself is dense in $G$.  In turn, this means that $K$ is normalised by a dense subgroup of $G$, so  using again the first paragraph above, we infer that $K$ is normal in $G$, and hence $K$ is dense in $G$.

In $G^{\rtimes}$, both $K^* := K \rtimes \triv$ and $\Delta_K := \{(k,k\inv) \in G^{\rtimes} \mid k \in K\}$ are closed normal subgroups.  Since $K$ is dense in $G$, we observe that $K^*\Delta_K$ is a dense normal subgroup of $G^{\rtimes}$.  Moreover, $\Delta_K$ centralizes $D^* := D \rtimes \triv$.  Thus in the quotient $G^{\rtimes}/K^*$, we see that $D^*/K^*$ is a subgroup with dense centraliser, so it is central in $G^{\rtimes}/K^*$.  In particular, $D^*/K^*$ is abelian.  Since $D^*/K^* \cong D/K$, it follows that $D/K$ is abelian, that is, $[D,D] \le K$.  Since $K$ is closed by assumption, in fact $S \le K$, and the claim is proven.

Let $M$ be a compact  subgroup of $S$ relatively open for the $\mathcal T_{(\alpha)}$-topology, and such that $M \le L$.  Then $S = \lla M \rra$, since $S$ is topologically simple. Moreover, since $M$ is relatively $\mathcal T_{(\alpha)}$-open, it follows that  the topologies $\mathcal T_{(\alpha)}$ and $\mathcal T_{(\beta)}$ coincide on $S$, where $\beta = [M]$. Since the conjugation action of $G$ respects the topology of $D$, and hence also the topology of its closed subgroup $S$, we see that $M$ is commensurated in $G$.  By Lemma~\ref{profcomm}, by replacing $M$ with a finite index open subgroup of $M$, we may ensure that $M$ is locally normal in $G$.  Finally, we conclude from Lemma~\ref{lem:CptGenNC} that $S$ is compactly generated.  Thus $S$ is in $\sclass$.
\end{proof}

Suppose that there is a compact open subgroup $U$ of $G$ that is topologically finitely generated.  In this case, we can use a special case of far-reaching results due to N.~Nikolov and D.~Segal~\cite{NS} (the proof of which relies on the classification of the finite simple groups) to restrict the structure of dense normal subgroups of $G$, and hence that of commensurated compact locally normal subgroups.

\begin{thm}[Nikolov--Segal\footnote{We attribute this result to Nikolov--Segal as it is easily derived from the results in \cite{NS}, although not explicitly stated there in this form; we include such a derivation for clarity.}]\label{thm:NS}
Let $P$ be a topologically finitely generated profinite group having finitely many isomorphism types of composition factors. Then any dense normal subgroup of $P$ contains the derived group $[P, P]$, which is closed in $P$. 
\end{thm}
\begin{proof}
Following \cite{NS}, we denote by $P_0$ the intersection of all open normal subgroups $T$ of $P$ such that there is a non-abelian finite simple group $S$ for which $P/T$ is isomorphic to some subgroup of $\mathrm{Aut}(S)$ containing $\mathrm{Inn}(S)$.  The hypothesis that $P$ has finitely many types of composition factors then imposes a bound on $|P:T|$ for such open subgroups $T$.  We recall (\cite[Proposition~2.5.1]{RZ}) that a topologically finitely generated profinite group has only finitely many open subgroups of a given index; thus $P_0$ is open in $P$.

Let now $N$ be a dense normal subgroup of $P$. 
We proceed as in the discussion preceding \cite[Corollary~1.8]{NS}. Since $N$ is dense and $P_0$ is open in $P$, {the image of $N$ in $P/P_0$ is onto, and the image of $N$ in $P/\overline{[P, P]}$ is dense. Using the fact that} $P$ is topologically finitely generated, we may then find a finite set of elements $y_1, \dots, y_r \in N$ such that $P_0 \overline{\langle{y_1, \dots y_r} \rangle} = \overline{ [P, P] \langle{y_1, \dots y_r} \rangle}=P$. Invoking \cite[Theorem~1.7]{NS}, we  then infer  that $[P_0, P] $ is entirely contained in $N$. Therefore $[P, P] = [NP_0, P] \leq N [P_0, P] N \leq N$, hence $[P, P]$ is entirely contained in $N$, as desired. 

The fact that  $[P, P]$ is closed follows from \cite[Corollary~5.9]{NS}. 
\end{proof}

We recall from Proposition~\ref{localprime:short} that if $G$ is a compactly generated \tdlc group whose only compact normal subgroup is the trivial one, then every compact open subgroup of $G$ has  finitely many isomorphism types of composition factors. The hypotheses of Theorem~\ref{thm:NS} are thus naturally satisfied in that context. We point out the following consequence,  part of which  was first noticed in a conversation with Nikolay Nikolov several years ago (see the unpublished preprint \cite[Proposition 4]{Nikolov}). 

\begin{cor}\label{cor:NS:dense_normal}
Let $G$ be a \tdlcsc group.  Let $U$ be a compact open subgroup of $G$, and suppose that $U$ is topologically finitely generated with finitely many types of composition factors. 
\begin{enumerate}[(i)]
\item Let $D$ be a dense normal subgroup of $G$.  Then $D \ge [U,U]$.
\item Let $L$ be a commensurated compact subgroup of $G$, and suppose that $D = \lla L \rra$ is dense in $G$ and $L$ is normal in $U$.  Then $L \cap [U,U]$ has finite index in $[U,U]$ and $D \ge [G,G]$.
\item Suppose $G$ is topologically simple. If $G$ is not compactly generated, assume also that  $U$ is non-abelian. Then $[G,G] = \lla [U,U] \rra$ is the unique smallest dense normal subgroup of $G$.  If $[U,U]$ has finite index in $U$ or $G = [G,G]$, then $G$ has no proper dense normal subgroups.
\end{enumerate}
\end{cor}

\begin{proof}
Let now $D$ be a dense normal subgroup of $G$.  Then $D \cap U$ is a dense normal subgroup of $U$, so part (i) follows by Theorem~\ref{thm:NS}.

Now suppose that $D = \lla L \rra$, where $L$ is a commensurated compact subgroup of $G$ that is normal in $U$.  By part (i), we see that $[U,U] \le D$.  Moreover, $D$ is $L$-meagre by Lemma~\ref{countsize}, so $[U,U]$ is $L$-meagre.  Since $[U,U]$ is closed by Theorem~\ref{thm:NS}, in fact $[U,U]$ is contained in the union of finitely many cosets of $L$.

In particular, we see that $U/L$ is virtually abelian, so $D \ge [G,G]$ by Proposition~\ref{prop:compression_factoring}, finishing the proof of (ii).

Suppose  $G$ is topologically simple. If $G$ is compactly generated, then $U$ cannot be abelian by Theorem~\ref{thm:noqz}. Thus the locally normal subgroup  $L = [U,U]$ is non-trivial.  By part (i), the group $D = \lla L \rra$ is contained in every dense normal subgroup of $G$; moreover, since $G$ is topologically simple, $D$ is dense in $G$, so it is the unique smallest dense normal subgroup of $G$.  By part (ii), we have $D \ge [G,G]$; clearly $D \le [G,G]$, so in fact $D = [G,G]$.  If $G = [G,G]$ then $D = G$; alternatively, if $[U,U]$ has finite index in $U$, then $D$ is open and dense in $G$, so $D =G$.  In either case, the minimality of $D$ ensures that there are no proper dense normal subgroups of $G$, so $G$ is abstractly simple.
\end{proof}

\subsection{Commensurated open subgroups}

We have seen in Proposition~\ref{compression:simple} how a non-trivial commensurated compact locally normal subgroup in a group $G \in \sclass$ gives rise to another group $S \in \sclass$ embedded in $G$ as a dense normal subgroup. The intersection of $S$ with the compact open subgroups of $G$ are thus commensurated open subgroups of $S$. The algebraic structure of those open subgroups may be described using Proposition~\ref{prop:compression_factoring}. In order to do this, it is convenient to invoke the following definition: a topological group is called an \textbf{\FCbar-group}\index{FC-group@\FCbar-group} if the conjugacy class of each element has compact closure.

\begin{lem}\label{lem:FCbar}
Let $G$ be a \tdlcsc group, let $U < G$ be a compact open subgroup and $L$ be a non-trivial closed normal subgroup of $U$ which is commensurated by $G$. Let $D = \lla L \rra$ and set $\alpha = [L]$; suppose $D$ is dense in $G$. Then, with respect to the topology $\mathcal T_{(\alpha)}$, the subgroup $D \cap U < D$ is  open, \FCbar, and  commensurated by $D$. 
\end{lem}

\begin{proof}
It is clear that $D \cap U$ is $\mathcal T_{(\alpha)}$-open. Since $G$ commensurates $U$, it follows that $D$ commensurates $D \cap U$. We observe that if $\pi$ is the quotient map given in Proposition~\ref{prop:compression_factoring}, then the construction of $\pi$ given in the proof ensures that the quotient $(D\cap U)/L$ is isomorphic to $\Ker(\pi)$ as an abstract group.  In particular, by Proposition~\ref{prop:compression_factoring}(3), $(D\cap U)/L$ is a group all of whose conjugacy classes are finite. Since $L$ is compact, this implies that $D \cap U$ is indeed \FCbar{} with respect to $\mathcal T_{(\alpha)}$. 
\end{proof}

We now show that, conversely, an open commensurated \FCbar-subgroup of a group $G \in \sclass$ naturally yields another group of $\sclass$ that admits non-trivial fixed points in its structure lattice. In order to do so, we need a few basic facts on \FCbar-groups, due to U\v sakov. We recall that a subgroup of a locally compact group is called \textbf{locally elliptic}\index{locally elliptic subgroup} if it is the directed union of its compact subgroups. 

\begin{prop}\label{prop:Ushakov}
Let $G$ be a locally compact \FCbar-group.
\begin{enumerate}[(i)]
\item $G$ is the directed union of its compactly generated closed normal subgroups. 

\item The set of compact subgroups of $G$ is a directed set (for the relation of inclusion), so that its union $R$ is a closed characteristic locally elliptic subgroup of $G$ such that $G/R$ is torsion-free abelian. In particular $G$ is amenable.

\item If $G$ is compactly generated, then $R$ is compact.
\end{enumerate}

\end{prop}

\begin{proof}
It is clear from the definition that every element is contained in a compactly generated closed normal subgroup. The fact that the collection of all those subgroups is directed (for the relation of inclusion) is straightforward. This proves (i). The assertions (ii) and (iii) are proved in \cite{Ushakov}. An alternative (and easier to find) reference is \cite{Wang1971}. 
\end{proof}

\begin{lem}\label{lem:OpenComm}
Let $G \in \sclass$. Let $O < G$ be an open   \FCbar-subgroup which is commensurated by $G$. Then there exist a   group $H \in \sclass$, a continuous injective homomorphism with dense image $\varphi \colon G \to H$, and a non-trivial commensurated compact locally normal subgroup   $M < G$ such that $\overline{\varphi(O)}$ is a compact open subgroup of $H$ and $\varphi(M)$ is a commensurated compact locally normal subgroup of $H$.
\end{lem}

\begin{proof}
Since $O$ is non-discrete and closed in $G$, it contains a non-trivial element which generates a cyclic subgroup with compact closure. By Proposition~\ref{prop:Ushakov}, this compact subgroup is contained in a closed, compactly generated normal subgroup of $O$, say $P$. By construction, the largest compact subgroup of $P$ afforded by Proposition~\ref{prop:Ushakov}, say $R$, is non-trivial, and normal in $O$.

By hypothesis $(G, O)$ is a Hecke pair. By Theorem~\ref{thm:Hecke:complete} and Proposition~\ref{prop:Hecke}, there exists a  group $H \in \sclass$ and a   homomorphism $\varphi \colon G \to H$ with dense image such that  $V = \overline{\varphi(O)}$ is a compact open subgroup of $H$. Since $O$ is open, the homomorphism $\varphi$ is continuous.  Since moreover $H$ is non-trivial and $G$ is topologically simple, it follows that $\varphi$ is injective.

Now $K = \varphi(R)$ is a non-trivial compact subgroup of $V$ normalised by $\varphi(O)$. Therefore $K$ is a non-trivial compact locally normal subgroup of $H$. Since $\varphi(G)$ is dense in $H$, it acts transitively on the conjugacy class of $K$. In particular $\varphi(G)$ contains the normal closure $\lla K \rra$ of $K$ in $H$. By Lemma~\ref{boxlem}, there is a finite set $\{K_1, \dots, K_n\}$ of conjugates of $K$ in $H$ and for each $i$, an open subgroup $L_i \leq K_i$ such that the product $L = L_1 \dots L_n$ is closed normal subgroup of $V$ which is commensurated by $H$.  
Since $L$ is contained in $\lla K \rra$, it is also contained in $\varphi(G)$. Let $M = \varphi\inv(L)$. Then $M$ is indeed compact in $G$ (because it is a finite product of compact subgroups of $G$). Since $L$ is normal in $V$ and commensurated by $H$, it follows that $M$ is normalised by $O$ and commensurated by $G$, so that $M$ is indeed a commensurated compact locally normal subgroup of $G$.
\end{proof}

The problem of determining whether all groups in $\sclass$ have no non-trivial fixed points in their structure lattice may now be reformulated as follows. 

\begin{prop}\label{prop:NoFPconj}
The following assertions are equivalent. 
\begin{enumerate}[(i)]
\item For all $G \in \sclass$, every infinite commensurated compact subgroup is open. 

\item For all $G \in \sclass$, every commensurated open \FCbar-subgroup is compact.

\item For all $G \in \sclass$, the set of compact open subgroups of $G$ is precisely the set of all  open commensurated \FCbar-subgroups of $G$.
\end{enumerate}

\end{prop}
\begin{proof}
Assume that (i) holds and that $G \in \sclass$ possesses an open commensurated \FCbar-subgroup $O$. We apply Lemma~\ref{lem:OpenComm}, which affords a group $H \in \sclass$, a dense embedding $\varphi \colon G \to H$ and a non-trivial commensurated compact locally normal subgroup   $M < G$ such that $\varphi(M)$ is a commensurated compact locally normal subgroup of $H$.
By (i) the group  $\varphi(M)$ must be open in $H$, so that $\varphi(O)$ is open with compact closure. Therefore 
$\varphi(O)$ is compact. If $U < O$ is a compact open subgroup of $O$, it has countable index, so that   $\varphi(U)$ is a closed subgroup of countable index in the compact group $\varphi(O)$. That index must therefore be finite. Since $\varphi$ is injective, we infer that  $U$ has finite index in $O$, so that $O$ is indeed compact.  Since every compact open subgroup of $G \in \sclass$ is a commensurated \FCbar-subgroup of $G$, we infer that   (i) indeed implies (iii).

That (iii) implies (ii) is clear.

Assume that (ii) holds and let $G \in \sclass$ with an infinite commensurated compact subgroup. By Lemma~\ref{profcomm}, the commensurability class of that compact subgroup has a closed locally normal representative $L$. By Proposition~\ref{compression:simple}, there is an infinite commensurated compact locally normal subgroup $M \leq L$ whose abstract normal closure $S = \lla M \rra$ in $G$, endowed with  the topology $\mathcal T_{(\beta)}$ with $\beta = [M]$, belongs to the class $\sclass$. Let $U < G$ be a compact open subgroup containing $M$ as a normal subgroup.   By Lemma~\ref{lem:FCbar}, the subgroup $S \cap U$ is a commensurated subgroup of $S$ which is open and \FCbar{} with respect to the topology $\mathcal T_{(\beta)}$. It then follows by  applying (ii) to the group $S \in \sclass$ that $S \cap U$ is $\mathcal T_{(\beta)}$-compact.   Since $M \leq S \cap U$ is $\mathcal T_{(\beta)}$-open by definition of the latter topology, it follows that $[S\cap U : M]$ is finite.  Since $S \cap U$ is dense in $U$, it follows that $M$ has finite index in $U$, and so $M$ is open in $G$.  Thus (ii) implies (i) and the cycle of implications is complete.
\end{proof}

\begin{rem}
Since \FCbar-groups are \{locally elliptic\}-by-\{torsion-free abelian\} by Proposition~\ref{prop:Ushakov}, one may wonder whether Proposition~\ref{prop:NoFPconj} remains true if one replaces open \FCbar-subgroups by open locally elliptic subgroups in (ii). It turns out that this is not the case (assuming that any of the statements (i)--(iii) in Proposition~\ref{prop:NoFPconj} are true); indeed, there are examples  of groups $G \in \sclass$ which act continuously but non-properly by automorphisms on regular locally finite trees (see \cite{Boudec}). The stabiliser of a vertex in $G$ is an open, but non-compact, locally elliptic subgroup, which is commensurated by $G$.
\end{rem}

\section{Dynamics of conjugation of locally normal subgroups}

The conjugation action of a \tdlc group~$G$ on its closed locally normal subgroups has an interesting dynamical property under the conditions that $G$ is compactly generated and locally C-stable, has an identity neighbourhood containing no non-trivial compact normal subgroups, and acts faithfully on the Boolean algebra $\lcent(G)$. We begin this section with the relevant definitions and results for actions of groups on general Boolean algebras, $\mcA$, and profinite spaces, $\Omega$, before proceeding the case when $\mcA$ is $\lcent(G)$ and $\Omega$ is the corresponding Stone space.

\subsection{Rigid stabilisers}

Let $\mcA$ be a Boolean algebra.  Then by the Stone representation theorem, $\mcA$ defines a {profinite space}\index{profinite space} (that is, a compact zero-dimensional space), the \defbold{Stone space}\index{Stone space} $\mfS(\mcA)$ of $\mcA$, whose points are the ultrafilters of $\mcA$ and the topology is generated by subsets of the form $\{\mfp \in \mfS(\mcA) \mid \alpha \in \mfp\}$ with $\alpha \in \mcA$.  Conversely, given a profinite space $\mfX$, the set $\mcA(\mfX)$ of clopen subsets of $\mfX$ form a Boolean algebra.  This correspondence produces a natural isomorphism between $\Aut(\mcA)$ and $\Aut(\mfS(\mcA))$.  In practice we will often find it convenient to abuse notation and treat elements of $\mcA$ as subsets of $\mfS(\mcA)$, identifying $\alpha \in \mcA$ with 
\[\mfS(\alpha) := \{\mfp \in \mfS(\mcA) \mid \alpha \in \mfp\} \in \mcA(\mfS(\mcA)).\]
The expressions `$\alpha \in \mfp$' and `$\mfp \in \alpha$' for $\alpha \in \mcA$ and $\mfp \in \mfS(\mcA)$ can therefore be taken to be synonymous.

\begin{defn}
Let $G$ be a topological group acting on a set ${\mcA}$.  Say that the action of $G$ on ${\mcA}$ is {\defbold{smooth}}\index{smooth action} if every point stabiliser is open. In particular, for every compact open subgroup $U$ of $G$, the orbits of $U$ on ${\mcA}$ are all finite.\end{defn}

The following lemma shows the relevance of smooth actions for  \tdlc groups:

\begin{lem}[{\cite[Lemma~5.11]{CRW-Part1}}]\label{lem:SmoothActionContinuous}
Let $G$ be a \tdlc group, let $\mfX$ be a profinite space and  $\mcA$ a Boolean algebra.

\begin{enumerate}[(i)]
\item If $G$ acts on $\mfX$ by homeomorphisms, and if the $G$-action is continuous with respect to the topology of uniform convergence, then the corresponding $G$-action on the Boolean algebra of clopen subsets of $\mfX$ is smooth.

\item If $G$ acts on $\mcA$ by automorphisms, and if the $G$-action is smooth, then the corresponding $G$-action on the Stone space $\mfS(\mcA)$ is continuous with respect to the topology of uniform convergence.  In particular, the action map $(g,x) \mapsto gx$ is a continuous map from $G \times \mfS(\mcA)$ to $\mfS(\mcA)$.
\end{enumerate}
\end{lem}

\begin{defn}
Let $G$ be a group acting on a set $Z$.  Given a subset $\upsilon \subseteq Z$, the \textbf{rigid stabiliser}\index{rigid stabiliser} of $\upsilon$ is the subgroup\index{rist@$\rist_G(\upsilon)$}
\[ \rist_G(\upsilon) := \{ g \in G \mid gz = z \quad \forall z \in Z \smallsetminus \upsilon\}.\]
\end{defn}

Notice that if $\mcA$ is a subalgebra of $\lcent(G)$ or of $\ldlat(G)$, then $\QC_G(\alpha) \le \rist_G(\alpha^\bot)$ for all $\alpha \in \mcA$. We emphasize that, in the notation $\rist_G(\alpha^\bot)$, we have relied on the (abusive) convention explained above to identify an element of a Boolean algebra with the corresponding subset of the associated Stone space under the Stone correspondence.  In particular, every non-zero element of $\mcA$ has an infinite rigid stabiliser in $G$.  Moreover $\rist_G(\upsilon)$ is a closed subgroup of $G$ for any $\upsilon \subseteq \mfS(\mcA)$, since $\rist_G(\upsilon)$ can be expressed as an intersection of point stabilisers, each of which is closed by the fact that the action of $G$ is continuous.

\begin{defn}
\label{defn:Booleanwkdecomp}
Let $G$ be a \tdlc group acting on a Boolean algebra $\mcA$ with kernel~$K$.  Say that the action is \defbold{weakly decomposable} if it is smooth and the quotient group $\rist_G(\alpha)/K$ is non-trivial for every $\alpha \in \mcA \smallsetminus \{0\}$.  Say the action is \defbold{locally weakly decomposable}   if moreover $\rist_G(\alpha)/K$ is non-discrete for every $\alpha \in \mcA \smallsetminus \{0\}$.
\end{defn}

When $G$ acts faithfully and has trivial quasi-centre, which will often be the case, then a faithful action is weakly decomposable if and only if it is locally weakly decomposable: $\rist_G(\alpha)$ is a normal subgroup of $\stab_G(\alpha)$, which is open in~$G$, so that $\rist_G(\alpha)$ would be in the quasi-centre of~$G$ if it were discrete. Under these conditions, the weakly decomposable condition also ensures that $\rist_G(\alpha)$ is not locally equivalent to $\rist_G(\beta)$ for any two distinct elements $\alpha,\beta \in \mcA$.

In view of Lemma~\ref{lem:SmoothActionContinuous}, the definition of weak decomposability of the the action of~$G$ on a Boolean algebra given in Definition~\ref{defn:Booleanwkdecomp} is equivalent with the definition of weak decomposability of the action of~$G$ on a profinite space given before Theorem~\ref{thmintro:WeaklyDecomposable}. 

\begin{prop}
\label{prop:wkdecompequivalence}
Let $G$ be a \tdlc group, $\mfX$ be a profinite space and $\mcA$ a Boolean algebra.

\begin{enumerate}[(i)]
\item A continuous action of~$G$ on~$\mfX$ by homeomorphisms is (locally) weakly decomposable if and only if the corresponding $G$-action on the Boolean algebra of clopen subsets of $\mfX$ is (locally) weakly decomposable.
\item A smooth action of~$G$ on~$\mcA$ by automorphisms is (locally) weakly decomposable if and only if the corresponding $G$-action on the Stone space $\mfS(\mcA)$ is (locally) weakly decomposable.
\end{enumerate}
\end{prop}

We recall the following result from \cite{CRW-Part1}, which shows that locally weakly decomposable actions appear naturally in the context of the centraliser or local decomposition lattices. 

\begin{prop}[See {\cite[Proposition~5.16]{CRW-Part1}}]\label{prop:WB-part1}
Let $G$ be a locally C-stable \tdlc group such that $\QZ(G)=\triv$ and let $\mcA$ be a $G$-invariant subalgebra of $\lcent(G)$.  Then the $G$-action on $\mcA$ is locally weakly decomposable as soon as it is faithful.  More precisely, if the $G$-action is faithful, then for each $\alpha \in \mcA$ we have
\[ \rist_G(\alpha) = \QC_G(\QC_G((\alpha)) = \CC_G(\QC_G(\alpha)) = \QC_G(\alpha^\bot) \text{ and } \alpha = [\rist_G(\alpha)],\]
where $\QC_G(\alpha)$ denotes the quasi-centraliser of any compact representative of $\alpha$.
\end{prop}

We close this section with an auxiliary assertion which will be used several times in the sequel. We use the following terminology. 
Let $\mcA$ be a Boolean algebra.  A \defbold{partition}\index{partition} $\mcP$ of $\alpha \in \mcA$ is a finite subset of $\mcA$ such that the join of $\mcP$ is $\alpha$ and the meet of any two distinct elements of $\mcP$ is $0$; a partition of $\mcA$ is just a partition of $\infty$ in $\mcA$. A partition  $\mcP_1$ is called \textbf{finer} than (or a \textbf{refinement}\index{refinement}\index{partition!refinement of a} of) a partition $\mcP_2$ if for every $\alpha_1 \in \mcP_1$, there exists $\alpha_2 \in \mcP_2$ with $\alpha_1 \leq \alpha_2$. 

\begin{lem}\label{lem:small_invariant_partition}
Let  $\mcA$ be a Boolean algebra and $G$ be a \tdlc group endowed with a    smooth faithful action on $\mcA$ by automorphisms. For all compact open subgroups $V \leq U \leq G$, there exists a $U$-invariant partition $\mcC$ of  $\mcA$ such that $\bigcap_{\gamma \in \mcC} U_\gamma \leq V$. 
\end{lem}

\begin{proof}
Let $\mathcal T$ be the collection of all $U$-invariant partitions of $\mcA$. We write $\mcP_1 \geq \mcP_2$ if the partition $\mcP_1$ is a refinement of the partition $\mcP_2$.  Since the $G$-action is smooth, every $U$-orbit on $\mcA$ is finite. In particular,   any finite set of partitions of $\mcA$ has a common refinement   that is $U$-invariant. In particular  $\mathcal T$ is a directed set and $\bigcup \mathcal T$ generates $\mcA$. 

Suppose now for a contradiction that for each $\mathcal C \in \mathcal T$, there exists $z_\mcC \in \bigcap_{\gamma \in \mcC} U_\gamma$ with $z_\mcC \not \in V$. Since $U$ is compact and $V$ is open, the net $(z_\mcC)_{\mathcal C \in \mathcal T}$ has a subnet $(z_{\mcC})_{\mcC \in \widetilde{\mathcal T}}$ converging to some   $z \in U \smallsetminus V$, where $\widetilde{\mathcal T}$ is a final subset of the directed set $ \mathcal T$.

We now consider an arbitrary element $\alpha \in \mcA$. By the first paragraph of the proof above, there exists $\mcP \in \widetilde{\mathcal T}$   which contains  a partition of $\alpha$. Therefore $\mathcal C$ contains a partition of $\alpha$ for every partition $\mathcal C \geq \mcP$. In particular,  we have $z_{\mathcal C} \in \bigcap_{\gamma \in \mathcal C} U_\gamma \leq U_\alpha$ for all $\mathcal C \in \mathcal T$ with $\mathcal C \geq \mcP$. Therefore $z = \lim_{\mathcal C \in \widetilde{\mathcal T}} z_{\mathcal C}$ is contained in $U_\alpha$. It follows that $z$ is an element of $U \smallsetminus V$ acting trivially on $\mcA$. This contradicts the hypothesis that the $G$-action on $\mcA$ is faithful.
\end{proof}

\subsection{A simplicity criterion}\label{sec:SimplicityCriterion}

We now give the  proof of the simplicity criterion from Proposition~\ref{prop:SimplicityCriterion}. It relies on the following subsidiary fact. 

\begin{lem}\label{lem:Derived}
Let $A,N$ be subgroups of a group $G$. Assume that $N$ is normal, and contains an element $t \in N$ such that $[A, tAt\inv]= \triv$. Then $N$ contains $[A, A]$. 
\end{lem}
\begin{proof}
Let $a, b, c \in A$. Since $a$ and $b$ commute with $tct\inv$, we have $[a, b] = [a, b tct\inv]$. Setting $c = b\inv$, we infer that $[a, b] = [a, [b, t]]$, which belongs to $N$ since $N$ is normal. 
\end{proof}

\begin{proof}[Proof of Proposition~\ref{prop:SimplicityCriterion}]
We must show that the monolith of $G$ (viewed as a discrete group) is non-trivial. 

We first show that the rigid stabiliser of any non-empty open set  $U \subseteq X$ is non-abelian. Since the $G$-action is micro-supported, we already know that the rigid stabiliser $\rist_G(U)$ is non-trivial. Pick $g \in \rist_G(U)$ and $x \in X$ with $gx \neq x$. Since $X$ is Hausdorff, there exists a neighbourhood $V$ of $x$ contained in $U$ such that $gV \cap V = \varnothing$. The rigid stabiliser $\rist_G(V)$ is non-trivial and contained in $\rist_G(U)$. Moreover the respective supports of $\rist_G(V)$ and $g\rist_G(V)g^{-1} = \rist_G(gV)$ are disjoint, so that the commutator $[g, \rist_G(V)]$ is non-trivial. In particular $\rist_G(U)$ is non-abelian. 

Let now $O$ be a non-empty open compressible subset of $X$. The previous paragraph implies that $\rist_G(O)$ is non-abelian. Let $N$ be any non-trivial normal subgroup of $G$. Then we may find an element $t \in N$ and a non-empty open subset $V$ of $X$ such that $tV \cap V = \varnothing$. Since $O$ is compressible there exists $g \in G$ such that $Q = gO \subset V$. In particular we have $[\rist_G(Q), t\rist_G(Q)t^{-1}] =1$. By Lemma~\ref{lem:Derived}, the group $N$ contains $[\rist_G(Q), \rist_G(Q)]$, hence also  $[\rist_G(O), \rist_G(O)] = g^{-1} [\rist_G(Q), \rist_G(Q)] g$. Therefore the monolith $\Mon(G)$ contains $[\rist_G(O), \rist_G(O)]$ and is thus non-trivial. 

What we have just established implies moreover that the $\Mon(G)$-action on $X$ is micro-supported. If we assume in addition that $\Mon(G)$ has a non-empty compressible open set, then the first part of the proof implies that $\Mon(\Mon(G))$ is non-trivial. Since the monolith is a characteristic subgroup, it follows that $\Mon(\Mon(G))$ is a non-trivial normal subgroup of $G$, which must thus contain $\Mon(G)$. In other words we have $\Mon(\Mon(G)) = \Mon(G)$, which means precisely that $\Mon(G)$ is simple. 
\end{proof}

\subsection{Minorising actions on Boolean algebras}\label{sec:MinorBool}

Recall, for $\Omega$ a topological $G$-space, the following definition from the introduction. Let $\omega \in \Omega$ A subset $F \subset \Omega$ is \textbf{compressible to the point $\omega$}\index{compressible to a point} if there is a basis of neighbourhoods $(V_i)$ of  $\omega \in \Omega$ such that for each $V_i$,  there is $g \in G$ with $gF \subset V_i$. We say that $F$ is \textbf{compressible}  if for every non-empty open set $O$ in $\Omega$, there is $g \in G$ with $gF \subset O$. Remark that $F$ is compressible if and only if  it is compressible to every point. Moreover, if $F$ is compressible to the point $\omega$ and if the $G$-orbit of $\omega$ is dense, then $F$ is compressible. 

In this subsection we introduce the closely related concept of a minorising action, which will play a critical role in obtaining structural properties of groups in $\sclass$ with non-trivial centraliser lattice.

\begin{defn}Let $\mcA$ be a poset with least element $0$ and let $G$ be a group acting on $\mcA$ by automorphisms.  Say a subset $\mcA'$ of $\mcA$ is \defbold{minorising}\index{minorising}\index{minorising!set} if for all $\alpha \in \mcA \smallsetminus \{0\}$ there is some $\beta \in \mcA'$ such that $0 < \beta < \alpha$.  (We emphasize that the inequality is understood to be strict. In particular that if~$\mcA$ has a minorising set, then it cannot have any atoms.)  Say $\mcA'$ is \defbold{minorising under the action of $G$} if $\bigcup_{\alpha \in \mcA'}G\alpha$ is minorising.  We say the action of $G$ on $\mcA$ is \defbold{minorising}\index{minorising!degree} (of \textbf{degree} $d \in \bN$) if there is a finite minorising set for the action (and the minimum number of elements in such a set is $d$).\end{defn}

The following lemma shows the relationship between minorising elements of the Boolean algebra and compressible subsets of the Stone space.

\begin{lem}\label{lem:CompressibleMinorising}
Let $G$ be a group of homeomorphisms of the profinite space $\Omega$, let $\alpha$ be a non-empty clopen subset of $\Omega$ and let $\mcA$ be the Boolean algebra of clopen subsets of $\Omega$.  Let $X$ be the set (possibly empty) of points $\omega \in \Omega$ such that $\alpha$ is $G$-compressible to $\omega$.  Then $X$ is a closed $G$-invariant set; moreover, $\alpha$ is minorising for the action of $G$ on $\mcA$ if and only if $X = \Omega$.
\end{lem}

\begin{proof}
Suppose $\omega \in X$ and let $(V_i)$ be a base of neighbourhoods of $\omega$.  For each $i$, there is some $g_i \in G$ such that $g_i\alpha \subseteq V_i$.  Now let $g \in G$.  Since $g$ is a homeomorphism, the net $(gV_i)$ forms a base of neighbourhoods of $g\omega$, and clearly we have $gg_i\alpha \subseteq gV_i$.  Thus $\alpha$ is $G$-compressible to $g\omega$, so $X$ is $G$-invariant.

Suppose $\omega \not\in X$.  Then by the definition of $X$, there is some neighbourhood $O$ of $\omega$ that does not contain any $G$-image of $\alpha$.  We see that in fact $\alpha$ is not compressible to $\omega'$ for any $\omega' \in O$.  Thus $\Omega \smallsetminus X$ is open, in other words $X$ is closed.

Suppose $X=\Omega$ and let $\beta \in \mcA \smallsetminus \{0\}$. Then $\beta$ is a non-empty open subset of $\Omega$, so $\beta$ is a neighbourhood of some point $\omega \in \beta$.  Hence there exists $g \in G$ such that $g\alpha \subset \beta$.  Hence the $G$-orbit of $\alpha$ is minorising in $\mcA$.

Conversely, suppose that $G\alpha$ is minorising in $\mcA$.  Since $\Omega$ is a profinite space, for every point $\omega \in \Omega$ there is a base of neighbourhoods of $\omega$ consisting of clopen sets.  For every such clopen set $\beta$, there exists $g \in G$ such that $g\alpha < \beta$.  Thus $\alpha$ is $G$-compressible to every point in $\Omega$.
\end{proof}

In the case of minorising actions on Boolean algebras, we can obtain a canonical structure in the Stone space that accounts for the degree of the minorising action.

\begin{lem}\label{minoropen}Let $G$ be a group with a minorising action of degree $d$ on the Boolean algebra $\mcA$.
\begin{enumerate}[(i)]
\item Let $\mcU$ be the set of minimal non-empty $G$-invariant open subsets of $\mfS(\mcA)$.  Then $|\mcU|=d$, and every $G$-invariant subset of $\mfS(\mcA)$ with non-empty interior contains some element of $\mcU$.
\item If $d=1$, then the union of the dense orbits of $G$ on $\mfS(\mcA)$ is a non-empty (hence dense) open set.  If $d > 1$, there are no dense orbits of $G$ on $\mfS(\mcA)$.\end{enumerate}\end{lem} 

\begin{proof}
(i) Let $G$ be a group with a minorising action on the Boolean algebra $\mcA$ and let $\{\alpha_1,\dots,\alpha_d\}$ be a minorising set for the action of minimal size.  For each~$1 \le i \le d$, let $\upsilon_i$ be the union of all $G$-translates of $\alpha_i$ in $\mfS(\mcA)$. We claim that~$\mcU = \left\{\upsilon_1,\dots, \upsilon_d\right\}$. 

Suppose that~$B \subseteq \mfS(\mcA)$ has non-empty interior. Then there is~$\beta\in \mcA$ contained in~$B$ and the minorising property implies that there are $i\in\{1,\dots, d\}$ and~$g\in G$ such that $g\alpha_i \leq \beta$. If~$B$ is also $G$-invariant, it hence follows that~$\upsilon_i$ is contained in~$B$. Once it is shown that the sets $\{\upsilon_1,\dots,\upsilon_d\}$ are pairwise disjoint, it will follow that every minimal open $G$-invariant subset of~$\mfS(\mcA)$ is one of the sets~$\upsilon_i$, and conversely. Let us suppose for a contradiction that~$\upsilon_i\cap \upsilon_j \ne \varnothing$. Then, by definition of~$\upsilon_i$ and~$\upsilon_j$, there are $g_i,g_j\in G$ such that~$g_i\alpha_i \cap g_j\alpha_j \ne \varnothing$. Since~$g_i\alpha_i \cap g_j\alpha_j$ is open, there are $k\in \{1,\dots, d\}$ and $g\in G$ such that $g\alpha_k \subseteq g_i\alpha_i \cap g_j\alpha_j$. Hence the set obtained by removing~$\alpha_i$ and~$\alpha_j$ from~$\{\alpha_1,\dots,\alpha_d\}$ and including~$\alpha_k$ is minorising for the action of~$G$. That contradicts minimality of the set unless $i = j = k$. Hence the elements of~$\left\{\upsilon_1,\dots, \upsilon_d\right\}$ are pairwise disjoint and this set is equal to~$\mcU$. 

(ii) Suppose $d=1$. Then $\mcU = \{\upsilon_1\}$ and for any open set $\alpha$ in $\mfS(\mcA)$ and any point $\mfp \in \upsilon_1$, there exists $g \in G$ with $g\mfp \in \alpha$ by the definition of $\upsilon_1$.  Therefore the $G$-orbit of every element of $\upsilon_1$ is dense. Since $\upsilon_1$ is open and non-empty, no element outside $\upsilon_1$ can have a dense orbit. 

Suppose now that $d>1$. Then $\upsilon_1$ and $\upsilon_2$ are disjoint by the above, and both are non-empty and $G$-invariant. Therefore no element can have a dense orbit. 
\end{proof}


\subsection{Skewering automorphisms of Boolean algebras}

The minorising property ensures the existence of \emph{skewering} elements, defined as follows.

\begin{defn}
Let $\mcA$ be a poset and let $g$ be an automorphism of $\mcA$.  Say $g$ is \defbold{skewering}\index{skewering element} if there is some $\alpha \in \mcA$ such that $g\alpha < \alpha$. Again, we emphasize that the inequality is understood to be strict. 
\end{defn}

\begin{lem}\label{lem:skewer}
Let $\mcA$ be a poset with least element $0$ and let $G$ be a group acting on $\mcA$, such that the action is minorising.  Let $\mcC$ be a finite minorising set under the action of $G$ such that $|\mcC|$ is minimised.  Then for each~$\alpha \in \mcC$ and $0 < \beta \le \alpha$, there is~$g \in G$ such that $g\alpha < \beta$.  In particular, there are elements of $G$ whose action on $\mcA$ is skewering.
\end{lem}

\begin{proof}By definition, there is some $\gamma \in \mcC$ and $g \in G$ such that $0 < g\gamma < \beta$.  We must have $\gamma = \alpha$ because~$\mcC \smallsetminus \{\alpha\}$ would be a minorising set otherwise.
\end{proof}

Before the next proposition, let us recall some terminology.

Given an automorphism $f$ of $G$, the \defbold{contraction group}\index{contraction group}\index{con(g)@$\con(g)$} $\con(f)$ of $f$ on $G$ is given by
\[ \con(f) = \{ u \in G \mid f^n(u) \rightarrow 1 \text{ as } n \rightarrow +\infty\}.\]

One sees that $\con(f)$ is a subgroup of $G$, although not necessarily a closed subgroup.  Given an element $g \in G$, we define $\con(g)$ to be the contraction group of the automorphism induced by left conjugation.

We define the \defbold{Tits core}\index{Tits core} of a totally disconnected, locally compact group $G$ to be the subgroup
\[ G^\dagger := \langle \overline{\con(g)} \mid g \in G \rangle.\]

The \defbold{nub} $\nub(f)$ of an automorphism $f$ of a \tdlc group $G$ is the largest compact $f$-invariant subgroup $N$ of $G$ such that $f$ does not stabilise any proper open subgroup of $N$; see \cite{WilNub} for more details.

\begin{prop}\label{goodshrink}
Let $G$ be a \tdlc group acting faithfully   on a Boolean algebra $\mcA$ such that the action is locally weakly decomposable.  Let $U$ be a compact open subgroup of $G$.  Suppose that there is some $\alpha \in \mcA$ and $g \in G$ such that $g\alpha < \alpha$; let $\beta = \alpha \smallsetminus g\alpha$.
\begin{enumerate}[(i)]
\item There is a natural number $n_0$ and a closed subset $\kappa$ of $\alpha$ satisfying the following properties: We have $g^{n_0} \beta \subset \kappa$, the group $\rist_U(\kappa)$ is non-trivial and moreover
\[ g\,\rist_U(\kappa)g^{-1} \times \rist_U(g^{n_0}\beta) \le \rist_U(\kappa).\]
 In particular  $\left\{ g^n\rist_U(\kappa)g^{-n}\right\}_{n\geq0}$ is a strictly decreasing sequence of compact subgroups of~$U$. Moreover $\bigcap_{n \ge 0} g^n\rist_U(\kappa)g^{-n}=1$ and $\con(g)$ contains the group $\rist_U(\kappa)$.  

\item There is a closed subgroup of $G$ of the form
\[ \overline{\langle L_0, g \rangle} \cong \prod_{i \in \bZ} L_i \rtimes \langle g \rangle,\]
where $L_i = g^i\rist_W(\beta)g^{-i}$ for some compact open subgroup $W$. In particular, $\nub(g)$ contains the direct product $\prod_{i \in \bZ} L_i$, and hence~$\con(g)$ is not closed.
\end{enumerate}
\end{prop}

\begin{proof} (i)  Let $V = U \cap g\inv Ug$. The action of~$U$ on~$\mcA$ is faithful and smooth. By Lemma~\ref{lem:small_invariant_partition}, there  is a $U$-invariant partition $\mcC$ of $\mcA$ such that $\bigcap_{\gamma \in \mcC} U_{\gamma} \leq V$.  

By hypothesis we have $g\alpha < \alpha$. Therefore, for each $\gamma \in \mcC$ we have either: (a) $g^n\alpha \ge \gamma$ for all $n \ge 0$; or (b) there is~$n_0\geq0$ such that~$g^n\alpha \not\ge \gamma$ for all $n\geq n_0$. Choose a fixed~$n_0$ sufficiently large that either~(a) or~(b) holds for every~$\gamma \in \mcC$.

Let $\xi$ be the interior of $\bigcap_{n \in \bZ} g^n\alpha$ in the Stone space $\mfS(\mcA)$.  Then the set 
$$\kappa := g^{n_0}\alpha \smallsetminus \xi$$ 
is closed.  Since $g\xi = \xi$, we have $g\kappa \subset \kappa$. Furthermore $\kappa \smallsetminus g\kappa = g^{n_0}\alpha \smallsetminus g^{n_0+1}\alpha = g^{n_0}\beta$, so $g\kappa \cup g^{n_0}\beta \subseteq \kappa$.  In particular $\rist_U(\kappa)$ contains $\rist_U(g\kappa) \times \rist_U(g^{n_0}\beta)$ as a subgroup and hence is non-trivial (and even non-discrete) since the $G$-action on $\mcA$ is locally weakly decomposable.  Moreover, given $\gamma \in \mcC$, either~(a) holds, in which case $\gamma \subseteq \xi$, or else~(b) does, in which case $\gamma \not\subseteq g^{n_0}\alpha$. Both cases imply that $\kappa\not\supseteq \gamma$ and hence that $\rist_U(\kappa)$ fixes some $\mfp \in \gamma$. Then $\rist_U(\kappa)$ fixes $\gamma$ because $\mcC$ is a $U$-invariant partition and, hence, the distinct $U$-translates of $\gamma$ are disjoint.  Therefore we have $\rist_U(\kappa) \leq \bigcap_{\gamma \in \mcC} U_{\gamma} \leq V$.    In particular $g\,\rist_U(\kappa)g\inv \le U$ and hence $g\,\rist_U(\kappa)g\inv \le \rist_U(g\kappa)$. Therefore $g\,\rist_U(\kappa)g\inv \times \rist_U(g^{n_0}\beta) \le \rist_U(\kappa)$ as claimed.

Let $R = \bigcap_{n \ge 0}g^n\rist_U(\kappa)g^{-n}$.  Note that {every point that is not fixed by} $R$ is contained in $\bigcap_{n \ge 0} g^n\kappa$. By construction, the intersection $\bigcap_{n \ge 0}g^n\kappa$ has empty interior, so the set of fixed points of $R$ is dense in $\mfS(\mcA)$.  Since $R$ acts by homeomorphisms on $\mfS(\mcA)$, the action of $R$ on $\mfS(\mcA)$ must be trivial.  Since the action of $G$ on $\mfS(\mcA)$ is faithful, we have $R = \triv$ as claimed. Finally, since~$(g^n\rist_U(\kappa)g^{-n})_{n \geq 0}$ is a descending chain of compact subgroups with trivial intersection, we have $\rist_U(\kappa) \le \con(g)$.

(ii) By part (i), we have $\con(g) \geq \rist_U(g^{n_0}\beta) = g^{n_0}\rist_{U'}(\beta) g^{-n_0}$, where $U' = g^{-n_0} U g^{n_0}$. Since $\con(g)$ is normalised by $g$, we infer that $\rist_{U'}(\beta) \leq \con(g)$. Next observe that the conditions that $g\alpha < \alpha$ and $\beta = \alpha \smallsetminus g\alpha$ are equivalent to ${g\inv\alpha^\bot < \alpha^\bot}$ and ${g\inv\beta = \alpha^\bot \smallsetminus g\inv\alpha^\bot}$. Therefore, the argument of part~(i) may be applied to $g\inv$ to conclude that 
$$\rist_W(\beta) \leq \con(g) \cap \con(g^{-1}),$$ 
where $W =g^{-n_0} U g^{n_0} \cap g^{n'_0} U g^{-n'_0}$ for some suitable $n'_0 >0$. 
This implies that every identity neighbourhood in $G$ contains all but finitely many $\langle g \rangle$-conjugates of $\rist_W(\beta)$.

Let $L_i = g^i\rist_W(\beta)g^{-i}$ for each $i \in \bZ$ and let $K =   \overline{\langle \bigcup_{i \in \mathbf Z} L_i\rangle}$. Note that for any distinct $i$ and $j$, the sets $g^i \beta$ and $g^j \beta$ are disjoint, hence the groups $L_i$ and $L_j$  commute. Moreover each  $L_i$ is compact,  and   every identity neighbourhood of $G$ contains all but finitely many $L_i$. It follows that $K$ is compact as well, and that moreover the  natural  homomorphism $\bigoplus_{i \in \bZ}L_i \rightarrow K$  extends  to a continuous homomorphism $\phi \colon \prod_{i \in \bZ}L_i \rightarrow K$. The image of $\phi$ is closed and dense in $K$, hence $\phi$ is surjective.  The kernel $N$ of $\phi$ commutes with each of the factors $L_i$, and is therefore contained in the centre of $\prod_{i \in \bZ}L_i$.  However, the fact that $G$ has a faithful locally weakly decomposable action ensures that $G$ is locally C-stable (see \cite[Theorem~5.18]{CRW-Part1}); hence each of the groups $L_i$ must have trivial centre, so $\prod_{i \in \bZ}L_i$ also has trivial centre.  Hence $N=\triv$, so $\phi$ is an isomorphism of profinite groups.  Since $K$ is compact and $g$ acts on $K$ by shifting the factors of the direct product, it is clear that $g$ does not normalise any proper open subgroup of $K$, so $K \le \nub(g)$.  We conclude that $\con(g)$ is not closed by \cite[Theorem~3.32]{BaumgartnerWillis}.

It remains to show that $\langle K,g \rangle \cong K \rtimes \langle g \rangle$ is a closed subgroup of $G$.  Notice that $\langle g \rangle$ is a discrete subgroup of $G$, since otherwise $\overline{\langle g \rangle}$ would be compact, which implies that $g$ normalises a basis of identity neighbourhoods in $G$ and, hence, that $\con(g)$ is trivial. Let now $H = \overline{\langle K, g \rangle}$ and notice that $K$ is a compact normal subgroup of $H$.  
The image of $\langle g \rangle$ in the quotient $H/K$ is dense since $\langle K \cup g \rangle$ is dense in $H$. Since $\langle g \rangle$ is discrete and torsion-free, it intersects every compact open subgroup trivially. This implies that $H/K$ is discrete and generated by the image of $\langle g \rangle$. Thus $K$ is open in $H$. Therefore we have $H = \langle K,g \rangle$ and $H$ is indeed closed.
\end{proof}

The fact, seen in the proof, that~$\con(g) \ge \rist_W(\beta)$ for some compact open subgroup $W$ implies the following. 

\begin{cor}\label{cor:goodshrink}
Let $G$ be a \tdlc group acting faithfully on a Boolean algebra $\mcA$, such that the action is locally weakly decomposable. Suppose that there is some $\alpha \in \mcA$ and $g \in G$ such that $g\alpha < \alpha$; let $\beta = \alpha \smallsetminus g\alpha$. Then the Tits core $G^\dagger$ contains $\rist_U(\beta)$ for some compact open subgroup $U$ of $G$.
\end{cor}


The existence of a minorising locally weakly decomposable action of $G$ imposes some restrictions on the algebraic structure of $G$.

\begin{prop}\label{prop:freemonoid}
Let $G$ be a \tdlc group acting faithfully on a Boolean algebra $\mcA$, such that the action is locally weakly decomposable, and let $\Gamma$ be a subgroup of $G$.  Suppose that the action of $\Gamma$ on $\mcA$ is minorising of degree $d$.  Then $\Gamma$ contains a free submonoid on $2$ generators that is discrete in $G$.
\end{prop}

\begin{proof}
Fix a minorising set $\mcC = \{\alpha_1,\dots,\alpha_d\}$ under the action of $\Gamma$ of smallest possible size.  Let $\alpha = \alpha_1$.  Then there exists $g \in \Gamma$ such that $g\alpha < \alpha$ {by Lemma~\ref{lem:skewer}}.  Moreover, there is $h \in \Gamma$ such that $h\alpha < \alpha \smallsetminus g\alpha$.  Let $S$ be the submonoid of $G$ generated by $g$ and $h$; we claim that $S$ is freely generated by $g$ and $h$ and that $|S \cap xG_\alpha| \le 1$ for all $x \in G$.  If it happened that either~$S$ were not freely generated or $|S \cap xG_\alpha| > 1$, there would be elements $x$ and $y$ of $S$ such that $x\alpha = y\alpha$, where $x$ and $y$ are represented as distinct strings $s$ and $t$ respectively in the alphabet $\{g,h\}$.  By deleting matching prefixes, we can ensure that the leftmost letters of $s$ and $t$ are different, say $s = gs'$ and $t = ht'$.  But then $x\alpha \le g\alpha$ and $y\alpha \le h\alpha$, so $x\alpha \wedge y\alpha \le g\alpha \wedge h\alpha = 0$, contradicting that $x\alpha = y\alpha$.  Thus $S$ is free on $2$ generators, and it is discrete in $G$ because~$\{xG_\alpha \mid x \in G\}$ is an open partition of $G$.
\end{proof}

\subsection{Minorising weakly decomposable actions}

The aim of this section will be to prove Theorem~\ref{cpctminormin}, which gives a sufficient condition for a compactly generated \tdlc group to have a minorising action on its centraliser lattice.  This condition is in particular satisfied by all groups in $\sclass$ with non-trivial centraliser lattice. 

We start by assembling some subsidiary lemmas which will be needed for the proof.

\begin{lem}\label{shiftlem}
Let $G$ be a group acting on a set $\Omega$ and let $\Phi\subseteq \Psi$ be subsets of $\Omega$.  Let $X$ be a symmetric subset of $G$ such that $G = \langle X, {\Stab_G(\Psi)} \rangle$. 
Suppose that 
\[ \bigcup \{ y\Phi \mid y \in G, \ y\Phi \subseteq \Psi \} \subseteq \bigcap_{x \in X} x^{-1}\Psi.\]
Then $y\Phi \subseteq \Psi$ for all $y \in G$.
\end{lem}

\begin{proof}
Let $l$ be the word length function on $G$ with respect to $X \cup {\Stab_G(\Psi)}$.  Suppose there is some $y \in G$ such that $y\Phi \not\subseteq \Psi$; assume $l(y)$ is minimal.  Then $y = xz$ where $l(x)=1$ and $l(z) < l(y)$.  The minimality of $l(y)$ ensures $x \not\in {\Stab_G(\Psi)}$, so $x \in X$.  Moreover $z\Phi \subseteq \Psi$, so $z\Phi \subseteq x^{-1}\Psi$  by hypothesis; but then $y\Phi \subseteq \Psi$, a contradiction.
\end{proof}

\begin{lem}\label{strminor}
Let $G$ be a \tdlc group acting faithfully on a Boolean algebra $\mcA$ such that the action is locally weakly decomposable and let $\Gamma$ be a subgroup of $G$.  Suppose the action of $\Gamma$ has a minorising set $\mcC$ of size  $n$.  
Then there is a subset $\mcD$ of $\mcA \smallsetminus \{0\}$ of size $n$, such that for each $\delta \in \mcD$ there is $\gamma \in \mcC$ with $\delta \leq \gamma$, and satisfying moreover the following property:
For all   compact open subgroups $U$, $V$ of $G$ and $\gamma \in \mcA \smallsetminus \{0\}$, there exist $y \in \Gamma$ and $\delta \in \mcD$ such that $y\rist_U(\delta)y\inv$ is a subgroup of $\rist_V(\gamma)$ of infinite index.\end{lem}

\begin{proof}
By Lemma~\ref{lem:skewer}, for each $\alpha \in \mcC$ there is~$g_\alpha \in \Gamma$ such that $g_\alpha \alpha < \alpha$. For each~$\alpha\in\mcC$ choose~$n_0$ as in Proposition~\ref{goodshrink} and set
$$ \delta_\alpha = g_\alpha^{n_0}(\alpha \smallsetminus g_\alpha \alpha) \mbox{ and } \mcD = \{ \delta_\alpha \mid \alpha \in \mcC\}.$$  
{By construction $|\mcD| \le n$. Since $\mcD$ is a minorising set under the action of $\Gamma$ on $\mcA$, we infer that $\mcD$ cannot have fewer than $n$ elements, and so $|\mcD|=n$. }

Fix $\gamma \in \mcA \smallsetminus \{0\}$. We must show that, for all compact open subgroups~$U$ and $V$ of $G$, there are $y\in\Gamma$ and $\delta\in \mcD$ such that~$y\,\rist_U(\delta)y^{\inv} <  \rist_V(\gamma)$.  To begin, fix $x \in \Gamma$ and $\alpha \in \mcC$ such that $x\alpha < \gamma$, and let $K := x\,\rist_U(\alpha)x^{-1}\cap V$, so that $K$ is an open subgroup of $x\,\rist_U(\alpha)x^{-1}$. 
Then, since~$x\,\rist_U(\alpha)x^{-1}= \rist_{xUx^{-1}} (x\alpha) < \rist_{xUx^{-1}} (\gamma)$, we see that $K \le \rist_V(\gamma)$.  

By Proposition~\ref{goodshrink} there is a closed subset~$\kappa$ of~$\alpha$ such that $\delta_\alpha \subseteq \kappa$ and such that $\left\{xg_\alpha^n\rist_U(\kappa)g_\alpha^{-n}x^{-1}\right\}_{n \ge 0}$ is a descending chain of closed subgroups of $x\,\rist_U(\alpha)x^{-1}$ with trivial intersection. Hence, ${xg_\alpha^n\rist_U(\delta_\alpha)g_\alpha^{-n}x^{-1} \le K}$ for some~$n$ by a standard compactness argument. Setting $y := xg_\alpha^n$,  it follows that~$y\rist_U(\delta_\alpha)y^{-1} \le \rist_V(\gamma)$. The index $[\rist_V(\gamma): y\rist_U(\delta_\alpha) y^{-1}]$ is infinite because $y\delta_\alpha =  xg_\alpha^n\delta_\alpha \le x\alpha < \gamma$.  The claimed minorising property of $\{\rist_U(\delta) \mid \delta \in \mcD\}$ is thus verified.
\end{proof}

\begin{thm}\label{cpctminormin}
Let $G$ be a non-trivial locally C-stable compactly generated \tdlc group and let $\mcA \subseteq \lcent(G)$ be a $G$-invariant subalgebra.  Suppose that $G$ acts faithfully on $\mcA$ and that some open subgroup of $G$ has trivial core. Then the set of minimal closed   normal subgroups of $G$ is finite; we denote it by $\mcM = \{M_1,\dots,M_d\}$. 
Moreover, the following assertions hold, where for each $1 \le i \le d$, the symbol $\upsilon_i$ denotes the complement of the fixed-point-set of $M_i$ on $\mfS(\mcA)$. 

\begin{enumerate}[(i)]
\item The $G$-action on $\mcA$ is minorising of degree $d$. Indeed, there is a set of $d$ elements in $\mcA$ that is minorising for the $G$-action on $\lcent(G)$. Moreover, the set $\{\upsilon_1,\dots,\upsilon_d\}$ coincides with the set $\mcU$ of all minimal non-empty $G$-invariant open subsets of $\mfS(\mcA)$, as in Lemma~\ref{minoropen}.  In particular, the sets $\upsilon_i$ are pairwise disjoint. 

\item Suppose that $\mcA$ contains the fixed points of the action of $G$ on $\lcent(G)$.  Let $\alpha_i =  [M_i]^{\bot^2} \in \lcent(G)$ for $1 \le i \le d$.  Then $\alpha_i = \overline{\upsilon_i}$, $\alpha_i \in \mcA$ and $\{\alpha_1,\dots,\alpha_d\}$ generates the subalgebra of fixed points of $G$ on $\lcent(G)$.


\item For each $i \in \{1,\dots,d\}$ there is an infinite compact subgroup $L_i$ of $M_i$ which is locally normal in~$G$ and such that: for every non-trivial closed locally normal subgroup~$K$ of~$G$, there are $i \in \{1,\dots,d\}$ and~$g \in G$ such that~$gL_ig^{-1}$ is a subgroup of~$K$ with infinite index.
\end{enumerate}
\end{thm}

\begin{proof}
Let us first note that any discrete normal subgroup of $G$ would act trivially on $\mcA$, so in fact $G$ has no non-trivial discrete normal subgroups.  The locally C-stable condition then ensures that $G$ has no non-trivial abelian normal subgroups by Proposition~\ref{cstab}.  The hypotheses ensure that there is a compact open subgroup $U$ of $G$ such that $\bigcap_{g \in G}gUg\inv = \triv$.  We can thus apply Proposition~\ref{cmsoc}, so that $G$ indeed has only finitely many minimal closed normal subgroups, such that every non-trivial closed normal subgroup contains a minimal one.

(i)
We first aim to show that the action of $G$ on $\mcA$ is minorising.

Since $G$ is compactly generated, we have $G = \langle U, X\rangle$ for a finite set $X$; we may assume $X = X^{-1}$.  Let $V = \bigcap_{x \in X}x^{-1}Ux$.  By Lemma~\ref{lem:small_invariant_partition}, there is a $U$-invariant partition  $\mcC$ 
of $\mcA$ such that $\bigcap_{\gamma \in \mcC} U_{\gamma} \leq V$. 

Let $\beta \in \lcent(G) \smallsetminus \{0\}$. Consider the action of~$G$ on itself by conjugation and let $\Phi = \rist_U(\beta)$ and $\Psi = U$.  Note that $\Phi$ is a non-trivial subgroup.  Since $U$ has trivial core in $G$, there must exist $g \in G$ such that $\rist_U(\beta)$ is not contained in $g\inv Ug$, or equivalently, $U$ does not contain every $G$-conjugate of $\rist_U(\beta)$.  The contrapositive of Lemma~\ref{shiftlem} then ensures the existence of $g \in G$ such that $g\,\rist_U(\beta)g^{-1}$ is contained in $U$ but not $V$.  For this~$g$, there is some $\gamma \in \mcC$ such that $g\,\rist_U(\beta)g^{-1}$, and hence $\rist_U(g\beta)$, does not stabilise $\gamma$.  Since~$\rist_U(g\beta)$ fixes $(g\beta)^\bot$ pointwise and the partition  $\mcC$ of $\mcA$ is $U$-invariant,  we must have $\gamma \cap (g\beta)^\bot = \varnothing = h\gamma \cap (g\beta)^\bot$ where~$h\in \rist_U(g\beta)$ and~$h\gamma = \gamma'$ for some~$\gamma' \in \mcC$ different from $\gamma$. Hence~$\gamma < g\beta$ and, since~$\beta$ was arbitrary, we have shown that $\bigcup_{\gamma \in \mcC}G\gamma$ is minorising for $G$.

Since $G$ is minorising of some finite degree, we can now appeal to Lemma~\ref{minoropen}.  Specifically, to show that the action is minorising of degree $d$, it suffices to prove that $\{\upsilon_1, \dots, \upsilon_d\}$ is the set of minimal open invariant subsets of $\mfS(\mcA)$ as claimed.  By definition, the subsets $\{\upsilon_1, \dots, \upsilon_d\}$ of $\mfS(\mcA)$ are open and $G$-invariant. We need to show that they are pairwise disjoint, and that every non-empty open $G$-invariant subset of $\mfS(\mcA)$ contains some $\upsilon_i$.

Suppose $\upsilon_i \cap \upsilon_j$ is non-empty for some distinct $i,j \in \{1,\dots,d\}$; say $\mfp \in \upsilon_i \cap \upsilon_j$.  There is then $g_i \in M_i$, $g_j \in M_j$ and an open neighbourhood $\delta$ of $\mfp$ contained in $\mfS(\mcA)$, such that both $g_i\delta$ and $g_j\delta$ are disjoint from $\delta$.  In particular, both $g_i\rist_G(\delta)g\inv_i$ and $g_j\rist_G(\delta)g\inv_j$ have support disjoint from that of $\rist_G(\delta)$ and therefore commute with $\rist_G(\delta)$ (since the action is faithful).  By Lemma~\ref{lem:Derived}, we conclude that both $M_i$ and $M_j$ contain the derived group of $\rist_G(\delta)$.  Since $M_i \cap M_j = \triv$, it follows that $\rist_G(\delta)$ is abelian; since $G$ is locally C-stable it then follows by Proposition~\ref{cstab} that $\rist_G(\delta)$ is trivial.  This is impossible, since the action of $G$ on $\mcA$ is weakly decomposable.  Thus the elements of $\{\upsilon_1, \dots, \upsilon_d\}$ are pairwise disjoint.

 
On the other hand,  if $\upsilon$ is a non-empty open $G$-invariant subset of $\mfS(\mcA)$, then $\rist_G(\upsilon)$ is a closed normal subgroup of $G$, which is non-trivial since $\alpha \subseteq \upsilon$ for some $\alpha \in \mcA \smallsetminus \{0\}$. Hence~$\rist_G(\upsilon)$ contains one of the minimal closed normal subgroups,~$M_i$, and it follows that~$\upsilon \supseteq \upsilon_i$.  

(ii) For $1 \le i \le d$, set $\alpha_i = [M_i]^{\bot^2}$.  Certainly $\alpha_i$ is a fixed point of the action of $G$ on $\lcent(G)$, so by hypothesis $\alpha_i \in \mcA$.   Since distinct elements of $\mcM$ commute, we have $[M_i]^\bot \geq [M_j]$, and so $[M_i] \leq [M_i]^{\bot^2} \leq [M_j]^\bot$ for any two distinct elements $M_i$ and $M_j$ of $\mcM$.  Thus $\{\alpha_1,\dots,\alpha_d\}$ is a pairwise disjoint set of elements of $\mcA$.  By Proposition~\ref{prop:WB-part1}, some open subgroup of $M_i$ is contained in $\rist_G(\alpha_i)$.  Since $M_i$ is non-discrete, $\rist_G(\alpha_i) \cap M_i > \triv$.  Since $\rist_G(\alpha_i)$ is closed and normal in $G$ and $M_i$ is a non-discrete minimal non-trivial closed normal subgroup of $G$, we must in fact have $M_i \le \rist_G(\alpha_i)$.  In particular, it follows that $\upsilon_i \subseteq \alpha_i$.

Since the complement of~$\bigcup_{i=1}^d \upsilon_i$ is~$G$-invariant and does not contain any~$\upsilon_i$, it is nowhere dense in $\mfS(\mcA)$ by (i). Therefore, the open set $\alpha_i \smallsetminus \overline{\upsilon_i}$, which is contained in the complement of $\bigcup_{i=1}^d \upsilon_i$, must be empty. This confirms that $\upsilon_i$ is indeed dense in~$\alpha_i$ for each~$i$.  Let $\mcA'$ be the set of fixed points for the action of~$G$ on $\lcent(G)$. Then~$\alpha_i \in \mcA'$ for all $i$ by construction and $\mcA' \subseteq \mcA$ by hypothesis.  Given $\beta \in \mcA' \smallsetminus \{0\}$, it follows from part (i) that $\upsilon_i \subseteq \beta$ for some $i \in \{1,\dots,d\}$.  Since $\beta$ is closed and $\upsilon_i$ is dense in $\alpha_i$, in fact $\beta \ge \alpha_i$.  Since $\mcA'$ is a subalgebra of $\mcA$, we conclude that $\mcA'$ is the subalgebra generated by $\{\alpha_1,\dots,\alpha_d\}$.

(iii) Let $\mcD {\ = \{\delta_1, \dots, \delta_d\}}$ be as in Lemma~\ref{strminor} (with $\Gamma = G$). Since $\mcD$ is a minorising set for the $G$-action, we must have $\delta_i \subseteq \upsilon_i$ for all $i$ (up to renumbering).  For each $i \in \{1,\dots,d\}$, set $L_i = \overline{[\rist_U(\delta_i),\rist_U(\delta_i)]}$. Note that $L_i$ is infinite by virtue of the fact that $G$ is locally C-stable.  The quasi-centre of $L_i$ is a discrete, hence finite, locally normal subgroup of~$G$. It must thus be trivial by Proposition~\ref{cstab}.

Let $K$ be a non-trivial closed locally normal subgroup of $G$. Then there is a compact open subgroup $V$ of $G$ containing $K$ as a closed normal subgroup. Since the $G$-action on $\mcA$ is faithful, there is a point $\mathfrak p \in \mfS(\mcA)$ which is not fixed by $K$. Therefore, there exist $k \in K$ and a sufficiently small clopen neighbourhood $\gamma \in \mcA$ of $\mathfrak p$ such that $\gamma$ and $k\gamma$ are disjoint. In particular the groups $\rist_V(\gamma)$ and $k\,\rist_V(\gamma)k^{-1}$ commute. It then follows from Lemma~\ref{lem:Derived} that $K$ contains the derived subgroup of $\rist_V(\gamma)$. By Lemma~\ref{strminor}, there is $g \in G$ and $i \in \{1,\dots,d\}$ such that $g\,\rist_U(\delta_i)g^{-1}$ is contained in $\rist_V(\gamma)$, whence $K$ contains $gL_ig^{-1}$.  The construction ensures that the index of $gL_ig^{-1}$ in $K$ is infinite because~$kgL_ig^{-1}k^{-1}$ is infinite, is contained in~$K$ and intersects trivially with $gL_ig^{-1}$.

Let $1 \le i \le d$.  Since $M_i$ is a non-trivial closed locally normal subgroup of $G$, we have $gL_jg<  M_i$ for some $g \in G$ and $1 \le j \le d$; since $M_i$ is normal, $L_j < M_i$.  Given that $\delta_j \subseteq \upsilon_j$ and $L_j$ fixes every point outside of $\delta_j$, we see that $L_j \cap M_i = \triv$ unless $i=j$.  Thus we must have $L_i < M_i$.
\end{proof}

\subsection{Properties of simple \tdlc groups}\label{sec:mono}

We now focus on consequences of Theorem~\ref{cpctminormin} for the structure of groups in $\sclass$.  For clarity we highlight the special case of Theorem~\ref{cpctminormin} when $G \in \sclass$.

\begin{cor}\label{cor:lcent:minorising}
Let $G \in \sclass$ and let $\mathcal A = \lcent(G)$ or $\mathcal A = \ldlat(G)$. Assume that  $|\mathcal A|>2$. Then $G$ has a minorising orbit on $\mathcal A$.  Moreover, there is an infinite compact locally normal subgroup $L$ of $G$ with the following property: for every non-trivial closed locally normal subgroup $K$ of $G$, there exists $g \in G$ such that $gLg^{-1}$ is a subgroup of $K$ of infinite index.
\end{cor}

\begin{proof} 
In view of Corollary~\ref{lnormfix:lcent}, the $G$-action on $\mathcal A$ is faithful. 
Since $G$ is topologically simple, we are in the situation of Theorem~\ref{cpctminormin} where the number of minimal closed normal subgroups is $1$.  The conclusions are now immediate from Theorem~\ref{cpctminormin}.
\end{proof}

\begin{rem}
For the conclusion of Corollary~\ref{cor:lcent:minorising}, in general $L$ cannot be chosen to be a rigid stabiliser of the action of $G$ on $\lcent(G)$, as demonstrated by the following example.  Let $T$ be the regular tree of degree $d\ge 6$, let $G$ be the simple group $\Aut(T)^+$ generated by edge stabilisers in $G$.  Let $e \in ET$ and let $L$ be the fixator of the set of vertices closer to $t(e)$ than $o(e)$.  Then certainly $G \in \sclass$, $\lcent(G)$ is non-trivial, and both $L$ and $K := \overline{[L,L]}$ are infinite non-open compact locally normal subgroups of $G$.  Indeed, considered as an element of $\lcent(G)$, then $[L]$ is minorising under the action of $G$.  However, for all $v \in VT$, the permutation induced by $K_v$ on the neighbours of $v$ is even, whereas elements of $L$ can induce odd permutations on the neighbours of infinitely many vertices.  Thus $L$ is not conjugate in $G$ to a subgroup of $K$; indeed, there does not exist $g \in G$ such that $[gLg\inv] \le [K]$.  Since $[L]$ is minorising in $\lcent(G)$, the same argument shows that $gMg\inv \nleq K$ for all $g \in G$, where $M$ is any compact representative of any given non-zero element of $\lcent(G)$.
\end{rem}

We can prove a stronger version of Corollary~\ref{cor:FixedPointsGeneration} under the additional assumption that $\lcent(G)$ is non-trivial.  In this case we obtain two global invariants of the group $G$ that control the `size' of locally normal subgroups up to conjugation.

\begin{cor}
Let $G\in \sclass$ be such that $|\lcent(G)|>2$.  Suppose that every infinite commensurated compact subgroup of $G$ is open (this is automatic if $G$ is abstractly simple, see Theorem~\ref{baireabs}). Then there are natural numbers $m$ and $n$ and an open subgroup $U$ of $G$ such that the following holds, for every non-trivial closed locally normal subgroup $K$ of $G$:
\begin{enumerate}[(i)]
\item There is a set $\{g_1,\dots,g_m\}$ of $m$ elements of $G$ such that the (not necessarily direct) product $\prod^m_{i=1}g_iKg\inv_i$ contains $U$;
\item There is a set $\{h_1,\dots,h_n\}$ of $n$ elements of $G$ such that
\[
G = \langle h_iKh\inv_i \mid 1 \le i \le n \rangle.
\]
\end{enumerate}
\end{cor}

\begin{proof}
Let $L$ be the compact locally normal subgroup of $G$ given by Corollary~\ref{cor:lcent:minorising}.

By Theorem~\ref{boxcor}, there is a finite set $\{g_1,\dots,g_m\}$ of $G$ such that the product $\prod^m_{i=1}g_iLg\inv_i$ contains a non-trivial commensurated compact locally normal subgroup $U$ of $G$. Then this subgroup is open by hypothesis.  Since $G$ is compactly generated and topologically simple and $U$ is open, $G$ is generated by finitely many conjugates of $U$; since $U$ is contained in a subgroup generated by finitely many conjugates of $L$, finitely many conjugates of $L$ will also suffice to generate $G$.  In other words,
\[
G = \langle h_iLh\inv_i \mid 1 \le i \le n \rangle
\]
for some finite set $\{h_1,\dots,h_n\}$.

We have now proved (i) and (ii) for a specific compact locally normal subgroup $L$.  But in fact the choice of $L$ simultaneously provides a solution for all non-trivial compact locally normal subgroups $K$: given such a $K$ we have $r\inv Lr \le K$ for some $r \in G$, in other words, $L \le rKr\inv$, and therefore $\prod^m_{i=1}g_ir Kr\inv g\inv_i$ contains $U$ and $\{h_irKr\inv h_i\inv \mid 1 \le i \le n\}$ generates $G$.  Thus the group $U$ and the numbers $m$ and $n$ can be chosen independently of the choice of $K$.
\end{proof}

We can now also prove Theorems~\ref{thmintro:FixedPointsAtomic}, \ref{thmintro:con} and \ref{thmintro:PropertyS1}.

\begin{proof}[Proof of Theorem~\ref{thmintro:FixedPointsAtomic}]
Let $G \in \sclass$ and let $\alpha \in \lnorm(G)$.  We suppose that $\alpha$ is a minimal non-zero element of $\lnorm(G)$, in other words $\alpha$ is a minimal element of $\mcF := \lnorm(G) \smallsetminus \{0\}$; this will clearly be the case if $\alpha$ has a \hji representative.

Suppose $\lcent(G)$ is non-trivial.  Then by Corollary~\ref{cor:lcent:minorising}, every infinite compact locally normal subgroup $K$ of $G$ contains an infinite closed locally normal subgroup $L$ of $G$, such that $L$ has infinite index in $K$.  In particular, there cannot be any minimal nonzero element of $\lnorm(G)$, contradicting the existence of $\alpha$.  Thus $\lcent(G)$ is trivial and $\mcF$ is a filter by Lemma~\ref{lem:lnorm_filter}.  Since $\alpha$ is minimal in $\mcF$, we see that in fact $\mcF$ is a principal filter, with $\alpha$ the unique least element of $\mcF$.  If $\alpha = \infty$ then $\lnorm(G) = \{0,\infty\}$, so $G$ is locally \hji; otherwise, we see from Theorem~\ref{thmintro:types} that $G$ is of atomic type.  Thus (ii) is proved.  Theorem~\ref{thmintro:types} also ensures that $G$ acts trivially on $\lnorm(G)$; part (i) follows immediately.

From now on we may suppose that $G$ is of atomic type.  It remains to show that there exists a continuous homomorphism $\phi: S \rightarrow G$ such that $S \in \sclass$, $S$ is not of atomic type and $\phi(S)$ is a proper dense normal subgroup of $G$.  Consider a compact representative $K$ of $\alpha$.  Since $\alpha < \infty$, we see that $K$ is not open in $G$.  However $\alpha$ is fixed by $G$, so $K$ is commensurated in $G$.  Proposition~\ref{compression:simple} then provides a continuous homomorphism $\phi:S \rightarrow G$ such that $S \in \sclass$ and $\phi(S)$ is normal in $G$.  Moreover, $\phi(S) = \lla M \rra_G$ for some infinite compact locally normal subgroup $M$ of $G$.  We see from the construction of $S$ that $0 < [M] \le \alpha$, so by the minimality of $\alpha$, in fact $M$ is a representative of $\alpha$.  We see that $M$ has countable index in $\phi(S)$, but uncountable index in $G$, so $\phi(S) < G$.  In particular, $G$ is not abstractly simple.

Suppose $\lnorm(S)$ has an atom $\beta$.  Then the same argument as for $G$ shows that $\beta$ is unique, and thus the commensurability class of $\phi(K)$ is preserved by conjugation in $G$, where $K$ is any compact representative of $\beta$.  By Lemma~\ref{lem:comm:locnorm} we can find a compact locally normal representative $L$ of $\beta$ such that $\phi(L)$ is locally normal in $G$.  Minimality of $\alpha$ then ensures $[\phi(L)] = \alpha$, which in turn implies that $L$ is open in $S$, and hence $\beta$ is the greatest element of $\lnorm(S)$.  In this case $\lnorm(S) = \{0,\infty\}$, so $S$ is locally \hji; in particular, $S$ is not of atomic type.\end{proof}

\begin{proof}[Proof of Theorem~\ref{thmintro:con} and Corollary~\ref{corintro:con}]
If there is a subgroup $R = \prod_{i \in \bZ}K_i \rtimes \langle g \rangle$ of the form described, it is clear that $\lcent(G)$ is non-trivial, since for instance the centraliser of the factor $K_0$ is neither open nor discrete.

Conversely, suppose $\lcent(G) \neq \{0,\infty\}$.  By Corollary~\ref{cor:lcent:minorising}, $G$ has a minorising orbit on $\lcent(G)$; in particular, there exists $\alpha \in \lcent(G)$ and $g \in G$ such that $g\alpha < \alpha$.  Hence by Proposition~\ref{goodshrink}, there is a closed subgroup $R$ of $G$ of the required form.  Proposition~\ref{goodshrink} also ensures that $\con(g)$ is not closed.
\end{proof}

\subsection{Minimality and strong proximality}\label{sec:strongprox}

Recall that the action of a group $G$ on a compact space $\Omega$ is \textbf{minimal}\index{minimal action} if every orbit is dense. It is called   \textbf{proximal}\index{proximal action} if every pair  $\{\eta, \xi\} \subset \Omega$  is compressible to some $\omega \in \Omega$ (see 
Section~\ref{sec:MinorBool} for the definition of compressibility),  and \textbf{strongly proximal}\index{strongly proximal action} if for any probability measure $\mu$ on $\Omega$, the $G$-orbit $G\mu$ contains Dirac measures in its closure in the space of probability measures on $\Omega$.  Our aim in this section is to establish minimality and strong proximality for the action of $G$ on $\mfS(\lcent(G))$ in the case that $G \in \sclass$, which will lead to proofs of Theorem~\ref{thmintro:WeaklyDecomposable} and its corollaries.

Let us first consider dense orbits in $\mfS(\lcent(G))$.

\begin{thm}\label{fixdense}
Let $G$ be a topologically simple, locally C-stable \tdlcsc group.  Suppose $\mcA$ is a $G$-invariant subalgebra of $\lcent(G)$ for which $|\mcA| > 2$.  Then:
\begin{enumerate}[(i)]
\item The action of $G$ on $\mfS(\mcA)$ has a dense orbit; indeed, every orbit $\omega$ of $G$ on $\mfS(\mcA)$ is either a singleton or dense in $\mfS(\mcA)$.

\item If $G$ is compactly generated, then the action of $G$ on $\mfS(\mcA)$ is minimal.

\item Suppose that $\mcA$ is countable. If  $G$ is abstractly simple, then the action of $G$ on $\mfS(\mcA)$ is minimal.
\end{enumerate}
\end{thm}

\begin{proof}Since $G$ is locally C-stable, $\QZ(G)$ is discrete.  The existence of a non-trivial subalgebra of $\lcent(G)$ ensures that $G$ is not discrete; since $G$ is topologically simple, it follows that $\QZ(G)= \triv$.  By Corollary~\ref{lnormfix:lcent}, the action of $G$ on $\mcA$ is faithful, with $\mcA^G = \{0,\infty\}$.  The action is locally weakly decomposable by Proposition~\ref{prop:WB-part1}.

(i) The action of $G$ on $\mfS(\mcA)$ is faithful because the action on $\mcA$ is faithful.  Let $\omega$ be an orbit of $G$ such that $|\omega| >1$; such an orbit exists since the action of $G$ on $\mfS(\mcA)$ is non-trivial.  Since $G$ is topologically simple, $G$ acts faithfully on $\omega$ and hence $\rist_G(\mfS(\mcA) \smallsetminus \omega)=\triv$. Since the $G$-action is weakly decomposable, we conclude that $\omega$ is dense.

For the proof of parts (ii) and (iii), let $\kappa$ be the set of fixed points of $G$ acting on $\mfS(\mcA)$.  The fact that $G$ acts non-trivially on $\mfS(\mcA)$ ensures that $\kappa$ is not dense.

(ii) Let $\beta \in \mcA$ be such that $\beta > 0$ and $\beta \cap \kappa = \varnothing$; such a $\beta$ exists because $\kappa$ is not dense.  Considering $\beta$ as an element of $\lnorm(G)$, by Theorem~\ref{boxcor} there exists a finite subset $\{g_1,\dots,g_n\}$ of $G$ such that $\gamma = \bigvee^n_{i=1}g_i\beta$ is fixed by $G$.  Thus $\gamma^\bot$ is a fixed point of the action of $G$ on $\lcent(G)$; since clearly $\gamma > 0$, we have $\gamma^\bot < \infty$ and hence $\gamma^\bot = 0$.  We now observe that the set $\{g_1\beta,\dots,g_n\beta\}$ has a least upper bound $\delta$ in $\mcA$, since $\mcA$ is a lattice.  Since $\mcA$ is a subset of $\lnorm(G)$ we have $\gamma \le \delta$, so $\delta^\bot \le \gamma^\bot = 0$.  Thus $\delta = \infty$.  At the same time, since $\kappa$ is $G$-invariant, we have $g_i\beta \cap \kappa = \varnothing$ for all $i$, so $\delta \cap \kappa = \varnothing$, and hence $\kappa = \varnothing$.  We conclude by part (i) that every orbit of $G$ on $\mfS(\mcA)$ is dense, in other words the action is minimal.

(iii) We will prove this part in contrapositive form.  Suppose that the action of $G$ on $\mfS(\mcA)$ is not minimal, that is, not all $G$-orbits are dense.  By part (i), $\kappa$ is non-empty.  Using this fact, we can exhibit a proper non-trivial normal subgroup of $G$ to show that $G$ is not abstractly simple.  Define
\begin{equation}
\label{eq:directed_union} 
N = \bigcup_{\alpha \in \mcA, \alpha \cap \kappa = \varnothing } \rist_G(\alpha), 
\end{equation}
a directed union of subgroups of $G$ that is normal because~$\kappa$ is $G$-invariant and non-trivial because the $G$-action is weakly decomposable.  Thus $N$ is a dense normal subgroup of $G$; it remains to show that $N \neq G$.

Let~$U$ be a compact open subgroup of $G$ and suppose $N \ge U$. Since~$\mcA$ is countable,~\eqref{eq:directed_union} is a countable directed union of closed sets.  The Baire Category Theorem then implies that there is~$\alpha \in \mcA$ intersecting $\kappa$ trivially and such that $\rist_U(\alpha)$ is open in $U$.

Since $\alpha \cap \kappa = \varnothing$ we have $\alpha \ne \infty$ and so $\alpha^\bot > 0$.  In other words, taking $K$ to be a compact representative of $\alpha$, then $\CC_G(K)$ is non-discrete since $[\CC_G(K)] = \alpha^\bot$ by Proposition~\ref{prop:WB-part1}. Then faithfulness of the action of $G$ on $\mcA$ implies that $\rist_U(\alpha) \cap \rist_U(\alpha^\bot) = 1$; since $\rist_U(\alpha)$ is open in $U$, it follows that~$\rist_G(\alpha^\bot)$ is discrete.  This is absurd as we clearly have $\CC_G(K) \le \rist_G(\alpha^\bot)$.  We conclude that $N \not\ge U$ and hence that $N \neq G$ as claimed.
\end{proof}

We now give a sufficient condition for an action to be strongly proximal.

\begin{prop}\label{prop:StronglyProx}
Let $G$ be a topologically simple locally compact group acting continuously by homeomorphisms on a profinite space $\Omega$. Assume that the action is minimal and locally weakly decomposable, and that there is a non-empty compressible open subset. 
Then the $G$-action is strongly proximal.
\end{prop}

\begin{proof}
Every non-empty open subset contains a non-empty clopen subset. Moreover, by minimality every $G$-orbit visits every non-empty clopen subset. Therefore, our hypotheses imply that every point of $\Omega$ admits a compressible clopen neighbourhood.
 
By \cite[Theorem~4.5]{BT} (see also \cite[Prop.~VI.1.6]{Margulis} for the special case when $\Omega$ is metrisable), an action of a group on a compact space is strongly proximal as soon as that action is minimal, proximal and contains a compressible open set. 
Therefore, in order to prove the proposition,   it suffices to show that every pair $\{\eta, \xi\}$ in $\Omega$ is compressible. Given any pair  $\{\eta, \xi\}$, and let $\alpha$ be a compressible clopen neighbourhood of $\eta$.

Let $N$ be the subgroup of $G$ generated by $\bigcup_{g \in G} \rist_G(g\alpha)$. Then $N$ is a non-trivial normal subgroup of $G$, which is thus dense. Since every $G$-orbit on $\Omega$ is dense, so is every $N$-orbit. In particular there is some $g \in N$ such that $g\xi \in \alpha$. Among all the $g \in N$ with $g\xi \in \alpha$,  we choose one   of minimal word length with respect to the generating set $\bigcup_{g \in G} \rist_G(g\alpha)$. Write $g = r_n  \dots r_1$ as a product of elements $r_i \in \rist_G(\alpha_i)$, where $\alpha_i$ belongs to the $G$-orbit of $\alpha$. Set $g_j = r_j\dots r_1$ for all $j =\{1, \dots, n\}$ and set $g_0 = 1$. By the minimality of $n$, we have $g_{j-1}\xi \in \alpha_j$ for all $j>0$, since otherwise we would have $g_j \xi = g_{j-1}\xi$ and hence 
$$ r_n \dots r_{j+1} g_{j-1}\xi = g\xi  \in \alpha,$$
which contradicts the minimality of $n$.

Now we distinguish two cases. Assume first that $\eta \not \in \alpha_j$ for all $j \in \{1, \dots, n\}$. It then follows that $r_j$ fixes $\eta$ for all $j$, and hence so does $g$. Therefore we have $g(\{\eta, \xi\}) = \{\eta, g.\xi\} \subset \alpha$. Since $\alpha $ is compressible, this implies that the pair $\{\eta, \xi\}$ is also compressible, and we are done. 

Assume next that there is some $j>0$ such that $\eta \in \alpha_j$. Let then $i = \min \{j>0 \; | \; \eta \in \alpha_j\}$. Thus $r_j$ fixes $\eta$ for all $j< i$, and hence $g_{i-1}\eta = \eta$. In particular $g_{i-1}\eta \in \alpha_i$. Moreover, we have seen above that $g_{i-1}\xi \in \alpha_i$. Therefore $g_{i-1}(\{\eta, \xi\})  \subset \alpha_i$, and hence  $\{\eta, \xi\}$ is compressible because $\alpha_i$ is so. 

Thus every pair in $\Omega$ is compressible, and the $G$-action is indeed proximal.
\end{proof}

\begin{proof}[Proof of Theorem~\ref{thmintro:WeaklyDecomposable}]
By Theorem~\ref{thmintro:QZ} and Proposition~\ref{cstab}, the group $G$ is locally C-stable. Therefore the hypotheses of Theorem~\ref{fixdense} are fulfilled. 
Therefore the $G$-action on $\Omega= \mfS(\lcent(G))$ is continuous by Lemma~\ref{lem:SmoothActionContinuous}. It is locally weakly decomposable by Proposition~\ref{prop:WB-part1} (in particular it is smooth), and minimal by Theorem~\ref{fixdense}(ii).
Moreover, by Corollary~\ref{cor:lcent:minorising} there is a minorising $G$-orbit in $\lcent(G)$. This implies that the corresponding clopen subset of  $\Omega$ is compressible by Lemma~\ref{lem:CompressibleMinorising}. The strong proximality follows from Proposition~\ref{prop:StronglyProx}.

This proves (i).  Since $G$ is locally C-stable with trivial quasi-centre, the assertion (ii) follows from (i) together with \cite[{Theorem~II(ii)}]{CRW-Part1}.
\end{proof}

\begin{proof}[Proof of Corollary~\ref{corintro:free}]
Follows from Theorem~\ref{thmintro:WeaklyDecomposable} and Proposition~\ref{prop:freemonoid}.
\end{proof}

\begin{proof}[Proof of Corollary~\ref{corintro:TopolRigidity}]
If $\lcent(G)$ is non-trivial, then the $G$-action on $\lcent(G)$ is non-trivial by Theorem~\ref{thmintro:WeaklyDecomposable}, hence faithful since $G$ is topologically simple. The uniqueness of the topology on $G$ amongst $\sigma$-compact locally compact topologies follows from  \cite[{Corollary~7.6}]{CRW-Part1}.  Indeed, the same argument shows that the topology of $G$ is the only Polish topology compatible with the group structure.

If in addition $G$ is abstractly simple, then the only fixed points of $G$ in the structure lattice are the trivial ones by Theorem~\ref{baireabs}; the desired statement therefore also follows from  \cite[{Corollary~7.6}]{CRW-Part1}.
\end{proof}

%
%

We recall that strong proximality may be viewed as an antipodal condition to amenability. In fact we have the following {result due to H.~Furstenberg}. 

\begin{prop}\label{prop:StronglyProx:amen}
Let $G$ be a locally compact  group admitting a closed amenable subgroup $A$ such that $G/A$ is compact. Let also $\Omega$ be a compact $G$-space which is minimal and strongly proximal. Then $A$ fixes a point in $\Omega$, and the $G$-action on $\Omega$ is transitive.  
\end{prop}

\begin{proof}
By  \cite[Proposition 4.4]{Furstenberg}, for every compact $G$-space $Z$ such that the $G$-action on $Z$ is minimal and strongly proximal, there is a unique $G$-equivariant continuous map $\psi \colon G/A \to Z$. Since the $G$-action on $Z$ is minimal, it follows that the continuous map $\psi$ must be surjective. The image of the trivial coset under $\psi$ is a $G$-fixed point in $Z$, and the transitivity of $G$ on $G/A$ implies the transitivity on $Z$. 
\end{proof}

Corollary~\ref{corintro:amen} now follows immediately from Theorem~\ref{thmintro:WeaklyDecomposable} and Proposition~\ref{prop:StronglyProx:amen}.

\begin{rem}\label{rem:ex_prox}
Let $G$ be a compactly generated, topologically simple \tdlc group and suppose $|\lcent(G)|>2$.  It seems plausible that every non-trivial orbit of $G$ on $\lcent(G)$ is minorising.  This would imply that, for every closed subset~$X\ne \mfS(\lcent(G))$, there is a net $\{g_i\}\subset G$ such that $\lim_i g_iX$ is a singleton. By definition, this means that the action of $G$ on $\mfS(\lcent(G))$ is \textbf{extremely proximal}\index{extremely proximal}. That condition implies in particular   that $G$ contains non-abelian discrete free subgroups by \cite[Theorem~3.4]{Glasner}.  In light of the results we have so far, an equivalent formulation of the property `every non-trivial orbit of $G$ on $\lcent(G)$ is minorising' in the present context is the following:

Let $K$ be a compact locally normal subgroup of $G$ such that $K>1$.  Then there is some $g \in G$ such that $[\CC_G(K), \CC_G(gKg\inv)] = \triv$.
\end{rem}

\subsection{Abstract simplicity}\label{sec:abs_simple}

Our main aim in this subsection is to complete the proofs of Theorems~\ref{thmintro:PropertyS1} and~\ref{thmintro:AbstractSimplicity}.  
We start by recalling the main result from \cite{CRW-TitsCore}.

\begin{thm}[{\cite[Theorem~1.1]{CRW-TitsCore}}]\label{thm:condensenorm}Let $G$ be a \tdlc group and let $D$ be a dense subgroup of $G$ such that $G^\dagger \le \N_G(D)$.  Then $G^\dagger \le D$.  In particular, every dense subnormal subgroup of $G$ contains the Tits core of $G$.
\end{thm}

\begin{proof}[Proof of Theorem~\ref{thmintro:PropertyS1}]
Suppose case (i) holds. By Corollary~\ref{cor:lcent:minorising}, the group $G$ has a minorising orbit on $\lcent(G)$. By Corollary~\ref{cor:goodshrink}, the Tits core $G^\dagger$ contains $K = \rist_U(\beta)$ for some compact open subgroup $U < G$ and some non-zero $\beta \in \lcent(G)$. Therefore $G^\dagger$ also contains an infinite commensurated compact locally normal subgroup $L$ by Lemma~\ref{boxlem}.  Since $G^\dagger$ is contained in every dense normal subgroup of $G$ by Theorem~\ref{thm:condensenorm}, we infer that every dense normal subgroup contains the infinite commensurated compact locally normal subgroup  $L$. Hence property \sdn holds.

Suppose case (ii) holds, that is, there is a topologically finitely generated compact open subgroup $U$ of $G$.  We recall that by Proposition~\ref{localprime:short}, the composition factors of $U$ are of bounded order.  Now, given a dense normal subgroup $N$ of $G$, the intersection $N \cap U$ is a dense normal subgroup of $U$, and must thus contain the closed subgroup $[U, U]$ by Theorem~\ref{thm:NS}. Notice that $[U, U]$ is infinite by Theorem~\ref{thmintro:QZ}, and thus represents a non-zero element of the structure lattice. By Lemma~\ref{boxlem}, it follows that $N$ contains a representative of a non-zero $G$-fixed point in $\lnorm(G)$.  Thus $G$ has property \sdn.

The conclusion about abstract simplicity follows by Theorem~\ref{baireabs}.
\end{proof}

We next obtain a sufficient condition for the Tits core of a \tdlc group to be open, which we can use to prove that a large class of topologically simple groups are in fact abstractly simple via the following consequence of Theorem~\ref{thm:condensenorm}.

\begin{cor}\label{cor:opentw}
Let $G$ be a non-discrete topologically simple \tdlc group such that $G^\dagger$ is open in $G$.  Then $G = G^\dagger$ and $G$ is abstractly simple.\end{cor}

\begin{proof}
We see that $G^\dagger$ is a non-trivial closed normal subgroup of $G$, so $G = G^\dagger$ by topological simplicity.  Now given any non-trivial normal subgroup $N$ of $G$, then $N$ is dense in $G$ by topological simplicity, so $N \ge G^\dagger$ by Theorem~\ref{thm:condensenorm} and hence $N = G$.  Thus $G$ is abstractly simple.
\end{proof}

\begin{prop}\label{denscrit}
Let $G$ be a locally C-stable \tdlc group with $\QZ(G)=\triv$ and let $\mcA$ be a subalgebra of $\ldlat(G)$.  Suppose that the action of $G$ on $\mfS(\mcA)$ is faithful and minimal, and that $g\alpha < \alpha$ for some $g \in G$ and $\alpha \in \mcA$.  Then $G^\dagger$ is open in $G$.
\end{prop}

\begin{proof}
Let $\beta = \alpha \smallsetminus g\alpha$.  By Corollary~\ref{cor:goodshrink} we have $\rist_U(\beta) \le G^\dagger$ for some compact open subgroup $U$ of $G$.  The minimality of the $G$-action on $\mfS(\mcA)$ implies that~$\bigcup_{g \in G} g\beta = \mfS(\mcA)$. Moreover the compactness of~$\mfS(\mcA)$ implies that in fact $\mfS(\mcA) = \bigcup^n_{i=1}g_i\beta$ for some finite set~$\{g_1,\dots,g_n\}$. Hence~$\infty = \bigvee_{i=1}^n g_i\beta$. Recall that~$\gamma = [\rist_V(\gamma)]$ for all compact open subgroups~$V$ and all~$\gamma \in \ldlat(G)$ (see Proposition~\ref{prop:WB-part1}). Since the join in $\ldlat(G)$ corresponds to taking products of representatives, we see that
\[ \langle g_i\rist_U(\beta)g^{-1}_i \mid 1 \le i \le n \rangle \]
is an open subgroup of $G$. Since~$G^\dagger$ contains this subgroup, it is open too.
\end{proof}

\begin{proof}[Proof of Theorem~\ref{thmintro:AbstractSimplicity}]
In all cases, $G$ is locally C-stable by Theorem~\ref{thm:noqz}.

Suppose case (i) holds, that is, $\ldlat(G) \neq \{0,\infty\}$.  Since $G \in\sclass$, the action of $G$ on $\ldlat({G})$ is faithful by Corollary~\ref{lnormfix:lcent}.  By Theorem~\ref{fixdense}(ii), the action of~$G$ on~$\mfS(\ldlat({G}))$ is faithful and minimal. Moreover it has a minorising orbit by Corollary~\ref{cor:lcent:minorising}. In particular, there exists $g \in G$ and $\alpha \in \ldlat({G})$ such that $g\alpha < \alpha$.

It has thus been shown that $G$ satisfies the hypotheses of Proposition~\ref{denscrit}. Hence~$G^\dagger$ is open in $G$, whence~$G^\dagger = G$ and $G$ is abstractly simple by Corollary~\ref{cor:opentw}.

Suppose case (ii) holds and let $U$ be a compact open subgroup $U$ of $G$ such that $U$ is finitely generated as a profinite group and $\overline{[U,U]}$ is open in $G$.  In fact $[U,U] = \overline{[U,U]}$ by Theorem~\ref{thm:NS}, so $[U,U]$ is open in $G$.   We conclude with a similar argument as in the proof of Theorem~\ref{thmintro:PropertyS1}:  any dense normal subgroup $N$ of $G$ must contain $[U,U]$ by Theorem~\ref{thm:NS}, and hence $N$ is open in $G$.  Since $N$ is dense, we have $N = G$.  Thus $G$ has no proper dense normal subgroups; since $G$ is topologically simple by hypothesis, it follows that $G$ is abstractly simple.

Finally, suppose case (iii) holds, that is, $\lnorm(G)^G = \{0, \infty\}$ and some non-trivial compact locally normal subgroup $K$ of $G$ is topologically finitely generated.  It follows that every compact representative of $\alpha = [K]$ is topologically finitely generated, and indeed every representative of $\beta \in \lnorm(G)$ is topologically finitely generated whenever $\beta$ can be represented as a finite join of $G$-conjugates of $\alpha$.  Thus by Theorem~\ref{boxcor}, there must be a non-trivial compact locally normal subgroup $L$ that is topologically finitely generated and commensurated by $G$. Since we assume that $G$ does not have any non-trivial fixed point in the structure lattice, we infer that $L$ is open in $G$.  Thus $G$ is abstractly simple by Theorem~\ref{thmintro:PropertyS1}(ii).
\end{proof}

\begin{rem}\label{rem:veeinNandD}
The decomposition lattice is a sublattice of the structure lattice.  In particular, the least upper bound operation in $\ldlat(G)$ satisfies 
$$
[H_1]\vee[H_2] = [H_1H_2].
$$
This identity plays a crucial role in Proposition~\ref{denscrit}.  By contrast, in $\lcent(G)$ we have only that $[H_1]\vee[H_2] = [H_1H_2]^{\bot^2}$, so for groups $G \in \sclass$ of weakly decomposable type, the strategy employed in Proposition~\ref{denscrit} is not sufficient to show that $G^\dagger$ is open.  Indeed, in this case we do not know if $G^\dagger$ is necessarily open (although it is certainly non-trivial).
\end{rem}

\appendix

\section{Known sources of examples}\label{sec:Examples}

Until very recently, the only known results on the class $\sclass$ focused on specific families of examples studied with the help of the extra structure provided by the very construction of these families. Although the purpose of this paper is to lay the foundations of a study of general groups in $\sclass$, for reference we list the sources of examples of particular groups in $\sclass$ that we are aware of at the time of writing. 

\begin{itemize}
\item \textbf{Isotropic semi algebraic groups over non-Archimedean local fields}.\index{algebraic group} {Let $k$ be a non-Archimedean local field, and $\mathbf G$ be an algebraic group defined over $k$. Assume that $\mathbf G$ is non-commutative, almost $k$-simple and $k$-isotropic. The group $\mathbf G(k)$, endowed with the Hausdorff topology induced by $k$, is a compactly generated \tdlc group (see \cite[Corollary I.2.3.5]{Margulis}).  Let  $\mathbf G(k)^+$ be the normal subgroup of   $\mathbf G(k)$ generated by the unipotent radicals of $k$-defined parabolic subgroups. Then $\mathbf G(k)^+$ is a closed cocompact subgroup of $\mathbf G(k)$ by \cite[\S 6.14]{BorelTits}, and the quotient of  $\mathbf G(k)^+$ by its centre $Z$ is abstractly simple by the main result of \cite{TitsSimple}. In particular $\mathbf G(k)^+/Z$ belongs to the class $\sclass$.} 
All of these simple algebraic groups have been classified by J. Tits~\cite{Tits-Corvallis}; each of them acts properly and cocompactly on a locally finite Euclidean building constructed by F. Bruhat and Tits~\cite{BruhatTits}.  We also remark that if $k$ is a $p$-adic field, then the simple group $\mathbf G(k)^+/Z$ also carries the structure of a $p$-adic Lie group. Conversely, every compactly generated topologically simple $p$-adic Lie group\index{p-adic Lie group@$p$-adic Lie group} is of that form\footnote{Every topologically simple $p$-adic Lie group whose adjoint representation is non-trivial is   isomorphic to a $p$-adic simple algebraic group of the form $\mathbf G(k)^+/Z$, and is thus compactly generated, see \cite[Proposition~6.5]{CCLTV} (we emphasize that analytic $p$-adic groups are implicitly assumed to be linear in the latter reference).  It is however not known whether a topologically simple $p$-adic Lie group can have a trivial adjoint representation (or even be one-dimensional, hence locally isomorphic to $\mathbf Z_p$). A topologically simple $p$-adic Lie group whose adjoint representation is trivial must have an abelian Lie algebra and, hence,   be  locally abelian. Therefore it cannot be compactly generated by Theorem~\ref{thmintro:QZ}.}. More generally, it follows from \cite[Corollary~1.4]{CStu} that simple algebraic groups are the only members of the class $\sclass$ that are (locally) linear over a local field.

\item \textbf{Complete  Kac--Moody groups of indecomposable type over finite fields}.\index{Kac--Moody group} These groups are constructed as completions of minimal Kac--Moody groups over finite fields defined by J. Tits \cite{Tits-KM}, with respect to a suitable topology. The resulting locally compact groups are topologically simple and act properly and cocompactly on a locally finite building that need be neither Euclidean nor hyperbolic, see \cite{Remy}. In fact, all of these groups are abstractly simple by the main result of \cite{Marquis}.  These groups also contain non-abelian discrete free groups, and so are non-amenable. Some variations of this construction are possible, see \cite{RR}.

\item \textbf{Groups of tree automorphisms with Tits' property (P)}.\index{tree} A simplicity criterion for groups acting on trees has been given by Tits \cite{Tits70} and can be used to produce many examples. Some concrete examples have been constructed and studied in the groundbreaking work by Burger--Mozes \cite{BurgerMozes}.  The article \cite{BurgerMozes} also contains a wealth of results on the structure of \tdlc groups acting on trees, which have provided an important source of inspiration of the present work.  A recent variation on the Burger--Mozes constructions is due to Smith \cite{SmithNew}. Smith's construction  takes as an input a pair of transitive permutation groups that are both generated by their stabilisers; one of them is required to be finite, and the other is to be compactly generated of countable degree with finite subdegrees. The output is a group in $\mathscr S$ acting on a biregular tree, with one class of vertices of finite degree and the other of infinite degree. By varying the infinite permutation group in the input, Smith obtains the first construction of uncountably many pairwise non-isomorphic group in $\mathscr S$.

\item \textbf{Other groups of tree automorphisms} Tits' property (P) has been generalised by C. Banks, M. Elder and G. Willis~\cite{BanksElderWillis} and Le Boudec \cite{Boudec} to construct new examples of groups in $\sclass$; the groups in \cite{Boudec} act continuously but non-properly on locally finite trees.  Another simplicity criterion was developed by R. M\"oller and J.~Vonk~\cite{MollerVonk}, although the present authors are not aware of any new examples arising from this criterion.

\item \textbf{Groups of tree spheromorphisms}. The prototypical examples in this family are the groups of tree spheromorphisms introduced by Neretin \cite{Neretin}: there is one such group for each regular locally finite tree $T$ of degree $d \ge 3$, with isomorphism type depending on $d$. Their simplicity is due to C. Kapoudjian~\cite{Kapoudjian}. Some variations have been constructed by Barnea--Ershov--Weigel~\cite{BEW} and P.-E.~Caprace and T.~De~Medts~\cite{CDM}. 

\item \textbf{Groups acting on CAT(0) cube complexes}. The aforementioned Tits' simplicity criterion for groups acting on trees can be generalised for groups acting on higher-dimensional CAT(0) cube complexes; examples have been obtained by F.~Haglund and F.~Pau\-lin~\cite{HaglundPaulin}, N.~Lazarovich~\cite{Lazarovich} and P.-E.~Caprace~\cite{Cap}.  A very general construction, incorporating many previous examples in this class and also generalising the work of Burger--Mozes and Smith, was recently obtained by T.~De Medts, A.~Silva and K.~Struyve \cite{DMSS}.

\end{itemize}


Recall that we have divided the class $\sclass$ into five types, according to properties of the structure, centraliser and decomposition lattices.  We summarise here the division of the known examples amongst the five types.

It is known that simple algebraic groups over local fields do not admit any non-trivial compact locally normal subgroup, since each of their compact open subgroups is \hji (see \cite{Riehm}).   It is unknown whether Kac--Moody groups of \textit{compact hyperbolic type} can be locally \hji groups.   However, in every other known example of groups in $\sclass$, the existence of a non-trivial compact locally normal subgroup has been observed.  The type of the simple Kac--Moody groups is not known in general, although certainly they cannot be of atomic type, as they are all abstractly simple by \cite{Marquis}.

Several families of groups in $\sclass$ acting on higher-dimensional CAT(0) cube complexes have been discovered.  For all known groups in these families, there exist non-discrete fixators of wings  in the sense of \cite{Cap} and their complements; this property in turn implies that the centraliser lattice is non-trivial, so these examples are either weakly or locally decomposable.  For the examples in \cite{Cap}, a characterisation is given of which of these examples are locally decomposable, from which it can be seen that not all examples are locally decomposable.  In fact, as far as the present authors are aware, the non-locally decomposable examples in \cite{Cap} (and potentially some of their generalisations in \cite{DMSS}) are the only known examples of groups $G$ in $\sclass$ with $\ldlat(G)$ trivial but $\lcent(G)$ non-trivial.

Once the examples of algebraic groups, Kac--Moody groups, and groups acting on higher-dimensional CAT(0) cube complexes are excluded, all remaining examples mentioned in this appendix are easily seen to be of locally decomposable type, that is, such that the decomposition lattice is non-trivial.  Many of these examples are obtained by using variations of Tits' simplicity criterion from \cite{Tits70}. That criterion, and its abstract version provided by Proposition~\ref{prop:SimplicityCriterion},  is quite flexible and it is very likely that more examples of groups in $\sclass$ may be constructed by exploiting them. However, the `wildness' which is inherent in the flexibility of this type of construction should be contrasted by the fact that the resulting groups in $\sclass$ will always be of weakly decomposable type, and thus fall into the sub-family for which the tools developed in the present paper are most powerful.

\begin{rem}
Although most known examples are known to have non-trivial centraliser lattice, it is general hard to determine if a given Boolean algebra in the centraliser lattice is actually the whole of $\lcent(G)$ or $\ldlat(G)$, or merely a subalgebra.  As far as the authors are aware, there are only two situations where $\lcent(G)$ or $\ldlat(G)$ is non-trivial and has been completely determined:
\begin{enumerate}[(i)]
\item If a compact open subgroup of $G$ is a branch group, we can appeal to a result of A. Garrido \cite{Garrido} to conclude that, given any branch action of a compact open subgroup $U$ of $G$ on a rooted tree $T$, then $\ldlat(U)$, and hence $\ldlat(G)$, is generated by the rigid stabilisers of $U$ acting on $T$, or equivalently the rigid stabilisers of $U$ acting on the boundary of $T$.  In other words, in this situation the profinite spaces $\mfS(\ldlat(G))$ and $\partial T$ are $G$-equivariantly isomorphic.  This description covers Neretin's tree spheromorphism groups and the known examples of groups acting on locally finite trees with Tits' property (P), and in these examples we see that $\ldlat(G)$ is $G$-equivariantly isomorphic to the algebra of clopen subsets of the space $\partial T$.
\item If $G \le \Aut(T)$, where $T$ is a locally finite tree, such that $G$ is topologically simple and acts transitively on $\partial T$, then either $\lcent(G) = \{0,\infty\}$ or $\mfS(\lcent(G))$ is $G$-equivariantly isomorphic to $\partial T$: see Theorem~\ref{thm:appendix}(i) below.
\end{enumerate}
\end{rem}

\section{Locally primitive groups of tree automorphisms}

A group $G$ of permutations of a set $X$ is \defbold{quasi-primitive}\index{quasi-primitive} if every non-trivial normal subgroup of $G$ acts transitively on $X$.

A group $G$ of automorphisms of a locally finite tree $T$ is \textbf{locally quasi-primitive}\index{quasi-primitive!locally quasi-primitive} if for every vertex $v$, the finite permutation group induced by the action of the stabiliser $G_v$ on the set $E(v)$ of edges emanating from $v$, is quasi-primitive.

Groups of tree automorphisms provide an important  source of examples of groups in the class $\sclass$. The fundamental work of Burger--Mozes \cite{BurgerMozes} has shown  that the case where the action is locally quasi-primitive is especially rich and interesting. The goal of this appendix is to illustrate our results by specialising to that class.  We first show that the only fixed points in the structure lattice are the trivial ones.

\begin{prop}\label{prop:PrimTrees}
Let $T$ be a locally finite tree and $G \leq \Aut(T)$ be a closed edge-transitive subgroup which is locally quasi-primitive. If $G$ is topologically simple, then  $\lnorm(G)^G = \{0, \infty\}$. 
\end{prop}

\begin{proof}
\medskip \noindent
Let $\alpha \in \lnorm(G)^G \smallsetminus \{0\}$ be a non-zero fixed point of $G$ in the structure lattice. Let $e$ be an edge of $T$. By Lemma~\ref{profcomm}, the compact open subgroup $G_e$ contains a closed normal subgroup $L$ which  is a representative of $\alpha$. Since $G$ acts properly and cocompactly on a tree, it is compactly presented by \cite[Corollary~8.A.9]{CornulierHarpe}. Therefore, the last assertion of Proposition~\ref{prop:compression_factoring} ensures that the quotient $G_e / L$ is a topologically finitely generated, finite-by-abelian, profinite group. In particular it has a unique maximal finite subgroup. Let $K$ denote  the preimage in $G_e$ of the maximal finite subgroup of $G_e/L$. Thus $K$ is a closed normal subgroup of $G_e$ representing $\alpha$, and $G_e/K$ is torsion-free. Moreover it follows from the construction that  $K$ is the unique maximal closed subgroup of $G_e$ representing $\alpha$. 

Let $v_1$ and $v_2$ be the two adjacent vertices that share the edge $e$. The compact subgroup $K$ cannot be normal both in $G_{v_1}$ and $G_{v_2}$, since it would then be normal in $\langle G_{v_1} \cup G_{v_2} \rangle = G$. Upon renaming $v_1$ and $v_2$, we assume henceforth that $K$ is not normal in $G_{v_1}$. Let $M_1$ be the (abstract) normal closure of $K$ in $G_{v_1}$. By \cite[{Corollary~7.13}]{CRW-Part1} we have $[M_1] = [K] = \alpha$. Note that $K$ is contained in $M_1 \cap G_e$. Since $M_1 \cap G_e$ is commensurate with  $K$, and since $K$  is the unique maximal closed subgroup of $G_e$ representing $\alpha$, we infer that $M_1 \cap G_e = K$. 

Since $M_1$ is normal in $G_{v_1}$ while $K$ is not, it follows from the equality  $M_1 \cap G_e = K$ that $M_1$ acts non-trivially on the set of edges $E(v_1)$ emanating from $v_1$. By hypothesis the $G_{v_1}$-action on $E(v_1)$ is quasi-primitive, hence $M_1$ is transitive on $E(v_1)$. 

Now we distinguish two cases. Assume first that $K$ is not normal in $G_{v_2}$. We then denote the (abstract) normal closure of $K$ in  $G_{v_2}$ by $M_2$. The same arguments as above show that $M_2 \cap G_e = K$ and that $M_2$ is transitive on $E(v_2)$. It follows that $M_1 \cap M_2 = K$, and that $T$ is equivariantly isomorphic to the Bass--Serre tree of the amalgamated product $M_1 *_K M_2$. In particular the group $D = \langle M_1 \cup M_2\rangle $ acts properly  and edge-transitively on $T$, and is thus closed in $G$. Since $D$ is edge-transitive, it acts transitively on the $G$-conjugacy class of $K$. Therefore $D$ coincides with the normal closure of $K$ in $G$. Since $G$ is topologically simple, it follows that $D=G$. Thus  $M_1 = G_{v_1}$ and $M_2 = G_{v_2}$, and both are open in $G$, so that $\alpha = \infty$.

Consider next the case when $K$ is normal in $G_{v_2}$. We enumerate all the vertices at distance~$1$ from $v_2$ by $w_1 = v_1, w_2, \dots, w_d$, and for each $i$ we set $e_i = \{w_i, v_2\}$. Since $G_{v_2}$ is transitive on $E(v_2) = \{e_1, \dots, e_d\}$ by hypothesis, it follows that $K$ is contained as a closed normal subgroup in $G_{e_i}$, and is the unique maximal closed subgroup of $G_{e_i}$ representing $\alpha$. For each $i = 1, \dots, d$, we denote by $M_i$ the normal closure of $K$ in $G_{w_i}$. We claim that $M_i  \cap G_{e_i} = K$ and that $M_i$ is transitive on $E(w_i)$ for all $i$. Indeed, this has already been established for $i =1$, and the claim for $i>1$ follows from the transitivity of $G_{v_2}$   on $E(v_2)$. Noting that $M_i \cap M_j = K$ for all $i \neq j$ in $\{1, \dots, d\}$, we see that $T$ coincides with the Bass--Serre tree of the amalgamated product $M_1 *_K M_2 *_K \dots *_K M_d$ (viewed as the fundamental group of a finite tree of groups, the finite tree in question being isomorphic to the star $S= e_1 \cup \dots \cup e_d$). It follows that the group $D = \langle M_1 \cup \dots \cup M_d \rangle $ acts properly on $T$ with the star $S$ as a fundamental domain. In particular $D$ is closed in $G$, and $D$ is transitive on the $G$-conjugacy class of $K$. As in the previous case, this implies that $D$ is normal, hence open in $G$, so that $K$ is of countable index in $G$. Therefore $K$ is open and $\alpha = \infty$. 
\end{proof}

We emphasise that, under the hypothesis that $G$ is locally quasi-primitive, it follows from \cite[Prop.~1.2.1]{BurgerMozes} that $G$ is topologically simple as soon as it has no non-trivial discrete normal subgroup, and no proper open normal subgroup of finite index. 

We next turn to the case where $G$ is doubly transitive on the set of ends $\partial T$. This implies that $G$ is locally $2$-transitive, hence locally primitive.  In fact, if $G \leq \Aut(T)$ is closed, non-compact and transitive on $\partial T$, then it is $2$-transitive on $\partial T$ and locally $2$-transitive: see \cite[Lemma.~3.1.1]{BurgerMozes}. 

\begin{thm}\label{thm:appendix}
Let $T$ be an infinite locally finite tree and $G \leq \Aut(T)$ be a closed subgroup which is topologically simple and acts transitively on the set of ends $\partial T$. Then:
\begin{enumerate}[(i)]
\item The Stone space $\mfS(\lcent(G))$ is either trivial, or equivariantly homeomorphic to $\partial T$.  In particular $\lcent(G)$ is either trivial or countable.

\item $\lcent(G) \neq \{0, \infty\}$ if and only if the pointwise stabiliser of some half-tree is non-trivial.

\item If $\lcent(G) \neq \{0, \infty\}$, then $G$ is abstractly simple.
\end{enumerate}

\end{thm}

\begin{proof}
(i) Since $T$ is infinite and locally finite, the set of ends $\partial T$ is non-empty. Let $\xi \in \partial T$ and $P = G_\xi$ be the stabiliser of $\xi$. Then the set of elliptic elements of $P$ is an open normal subgroup $P_0$ of $P$, which is locally elliptic, and the quotient $P/P_0$ is cyclic. In particular $P$ is amenable. The canonical equivariant bijection $G/P \to \partial T$ is continuous, and must therefore be a  homeomorphism by \cite[Theorem~8]{Arens}. It follows that $G/P$ is compact. Moreover $P$ is a maximal subgroup of $G$ since the $G$-action on $G/P$ is doubly transitive (by  \cite[Lemma~3.1.1]{BurgerMozes}), hence primitive. Therefore, by Proposition~\ref{prop:StronglyProx:amen}, every compact $G$-space which is minimal and strongly proximal either is isomorphic to $G/P \cong \partial T$, or is a singleton. In particular, this applies to the Stone space $\mfS(\lcent(G))$ by Theorem~\ref{thmintro:WeaklyDecomposable}. This proves (i). 

(ii) Given any pair of half-trees $H_1, H_2$ in $T$, there is an element $g \in G$ mapping $H_1$ properly inside $H_2$. The `if' part follows easily from that observation. Assume conversely that $\lcent(G)$ is non-trivial. Then $\mfS(\lcent(G))$ can be identified with $\partial T$ by (i). Since the $G$-action on $\mfS(\lcent(G))$ is locally weakly decomposable by Theorem~\ref{thmintro:WeaklyDecomposable}, it follows that the pointwise stabiliser in $G$ of any non-empty proper clopen subset of $\partial T$ is non-trivial. The conclusion follows, since the pointwise stabiliser of a non-empty proper clopen subset of $\partial T$ fixes pointwise a half-tree of $T$. 

(iii) Upon replacing $T$ by a minimal $G$-invariant subtree, and discarding the vertices of degree $2$, we may assume that the $G$-action on $T$ is edge-transitive. By \cite[Lemma~3.1.1]{BurgerMozes}, the $G$-action on $\bd T$ is $2$-transitive, and the $G$-action on $T$ is locally $2$-transitive, hence locally quasi-primitive. The conclusion then follows from Proposition~\ref{prop:PrimTrees} and Theorem~\ref{thmintro:PropertyS1}.
\end{proof}

In particular, in this case we can solve the problem posed by Remark~\ref{rem:ex_prox}.

\begin{cor}
Let $T$ be an infinite locally finite tree and $G \leq \Aut(T)$ be a closed subgroup that is topologically simple and acts transitively on the set of ends $\partial T$.  Then every orbit of $G$ on $\lcent(G) \smallsetminus \{0,\infty\}$ is minorising.
\end{cor}

\begin{proof}
We may assume that $\lcent(G)$ is non-trivial.  Then by Theorem~\ref{thm:appendix}(i), the Stone space $\mfS(\lcent(G))$ is $G$-equivariantly homeomorphic to $\partial T$; in other words, we may identify $\lcent(G)$ with the Boolean algebra $\mcA$ of clopen subsets of $\partial T$.  Given a directed edge $e$ in $T$, let $S_e$ be the half-tree of vertices closer to $o(e)$ than to $t(e)$.  Then the set $\alpha_e$ of ends of $S_e$ is an element of $\mcA$; moreover, it is easily seen that the set
\[\mcH := \{\alpha_e \mid e \in ET\}\]
is minorising in $\mcA$.  Now \cite[Lemma.~3.1.1]{BurgerMozes} implies that $G$ is transitive on undirected edges of $T$; since $G$ is topologically simple, $G$ preserves the natural bipartition of the vertices.  Hence $G$ has exactly two orbits of the directed edges of $T$, so that $e$ and $e'$ lie in the same orbit if and only if $o(e)$ and $o(e')$ lie in the same part of the bipartition.  We see that for any two edges $e,e'$, there exists $g \in G$ such that $ge \neq e'$ and either $o(e') = t(ge)$ or there is an edge $e''$ from $o(e')$ to $t(ge)$.  In either case we see that $S_{ge} \subset S_{e'}$ and hence $g\alpha_{e} < \alpha_{e'}$.  Thus $G\alpha_e$ is minorising in $\mcH$, and hence in $\mcA$, for all $e \in ET$.  In particular, since $(\alpha_e)^\bot = \alpha_{\overline{e}}$, we obtain an element $\alpha \in \lcent(G)$ such that both $\alpha$ and $\alpha^\bot$ are minorising under the $G$-action.  Given $\beta,\gamma \in \lcent(G) \smallsetminus \{0,\infty\}$, there exist $g,h \in G$ such that $g\inv \alpha^\bot < \beta^\bot$ and $h\alpha < \gamma$; these two inequalities imply that $hg\beta < \gamma$.  Hence every orbit of $G$ on $\lcent(G) \smallsetminus \{0,\infty\}$ is minorising.
\end{proof}

\addcontentsline{toc}{section}{\protect\numberline{}Index}
\printindex

\end{document}